\documentclass[11pt]{aims}
\usepackage{amsmath, amssymb, mathrsfs}
  \usepackage{paralist}
  \usepackage{graphics} 
  \usepackage{epsfig} 
\usepackage{graphicx}  \usepackage{epstopdf}
 \usepackage[colorlinks=true]{hyperref}
 \hypersetup{urlcolor=blue, citecolor=red}
 \usepackage[toc,page]{appendix}
\usepackage{chngcntr}

\usepackage{titlesec}
\titleformat{\section}{\Large\bfseries}{\thesection.}{4pt}{}
\titleformat{\subsection}{\large\bfseries}{\thesection.\arabic{subsection}.}{4pt}{}
\titleformat{\subsubsection}{\bfseries}{\thesection.\arabic{subsection}.\arabic{subsubsection}.}{4pt}{}
\titleformat*{\paragraph}{\bfseries}
\titleformat*{\subparagraph}{\bfseries}
\usepackage[margin=1.1in]{geometry}

  \def\currentyear{}
   \def\currentmonth{}
    \def\ppages{}
     \def\DOI{}

\newtheorem{theorem}{Theorem}[section]

\newtheorem{lemma}[theorem]{Lemma}
\newtheorem{proposition}[theorem]{Proposition}
\theoremstyle{definition}
\newtheorem{definition}[theorem]{Definition}
\newtheorem{remark}[theorem]{Remark}

\newcommand{\RN}{\mathbb{R}^N}
\newcommand{\Rb}{\mathbb{R}}
\newcommand{\Lc}{\mathscr{L}}
\newcommand{\Ec}{\mathcal{E}}
\newcommand{\Hc}{\mathscr{H}}
\newcommand{\Oc}{\mathcal{O}}

\newcommand{\Cc}{\mathcal{C}}
\newcommand{\Mc}{\mathcal{M}}
\newcommand{\Tc}{\mathcal{T}}
\newcommand{\Dc}{\mathcal{D}}
\newcommand{\Gc}{\mathcal{G}}
\newcommand{\Fc}{\mathcal{F}}
\newcommand{\Vc}{\mathscr{V}}

\numberwithin{equation}{section}

\title[Blowup solutions for a non-variational semilinear parabolic system] 
      {Construction and stability of blowup solutions for a non-variational semilinear parabolic system}

\author[T. Ghoul, V. T. Nguyen, H. Zaag]{}
\subjclass{Primary: 35K50, 35B40; Secondary: 35K55, 35K57.}
 \keywords{Blowup solution, Blowup profile, Stability, Semilinear parabolic system}

 \email[T. Ghoul]{teg6@nyu.edu}
 \email[V. T. Nguyen]{Tien.Nguyen@nyu.edu}
 \email[H. Zaag]{Hatem.Zaag@univ-paris13.fr}

\thanks{H. Zaag is supported by the ERC Advanced Grant no. 291214, BLOWDISOL and by the ANR project ANA\'E ref. ANR-13-BS01-0010-03. \\ -----------------\\ \today}
\begin{document}
\maketitle

\centerline{\scshape Tej-Eddine Ghoul$^\dagger$, Van Tien Nguyen$^\dagger$ and Hatem Zaag$^\ast$}
\medskip
{\footnotesize
 \centerline{$^\dagger$New York University in Abu Dhabi, P.O. Box 129188, Abu Dhabi, United Arab Emirates.}
  \centerline{$^\ast$Universit\'e Paris 13, Sorbonne Paris Cit\'e, LAGA, CNRS (UMR 7539), F-93430, Villetaneuse, France.}
}

%
%
%

\bigskip

\begin{abstract} We consider the following parabolic system whose nonlinearity has no gradient structure:
$$\left\{\begin{array}{ll}
\partial_t u = \Delta u + |v|^{p-1}v, \quad & \partial_t v = \mu \Delta v + |u|^{q - 1}u,\\
u(\cdot, 0) = u_0, \quad & v(\cdot, 0) = v_0,
\end{array}\right. 
$$
in the whole space $\RN$, where $p, q > 1$ and $\mu > 0$. We show the existence of initial data such that the corresponding solution to this system blows up in finite time $T(u_0, v_0)$ simultaneously in $u$ and $v$ only at one blowup point $a$, according to the following asymptotic dynamics:
$$\left\{\begin{array}{c}
u(x,t)\sim \Gamma\left[(T-t) \left(1 + \dfrac{b|x-a|^2}{(T-t)|\log (T-t)|}\right)\right]^{-\frac{(p + 1)}{pq - 1}},\\
v(x,t)\sim \gamma\left[(T-t) \left(1 + \dfrac{b|x-a|^2}{(T-t)|\log (T-t)|}\right)\right]^{-\frac{(q + 1)}{pq - 1}},
\end{array}\right.$$
with $b = b(p,q,\mu) > 0$ and $(\Gamma, \gamma) = (\Gamma(p,q), \gamma(p,q))$. The construction relies on the reduction of the problem to a finite dimensional one and a topological argument based on the index theory to conclude. Two major difficulties arise in the proof: the linearized operator around the profile is not self-adjoint even in the case $\mu = 1$; and the fact that the case $\mu \ne 1$ breaks any symmetry in the problem. In the last section, through a geometrical interpretation of quantities of blowup parameters whose dimension is equal to the dimension of the finite dimensional problem, we are able to show the stability of these blowup behaviors with respect to perturbations in initial data.
\end{abstract}

\section{Introduction.} In this paper we are concerned with finite time blowup for the semilinear parabolic system:
\begin{equation}\label{PS}
\left\{\begin{array}{ll}
\partial_t u = \Delta u + |v|^{p-1}v, \quad & \partial_t v = \mu \Delta v + |u|^{q - 1}u,\\
u(\cdot, 0) = u_0, \quad & v(\cdot, 0) = v_0,
\end{array}\right. 
\end{equation}
in the whole space $\RN$, where
$$p, q > 1, \quad \mu > 0.$$
The local Cauchy problem for \eqref{PS} can be solved in $L^\infty(\RN) \times L^\infty(\RN)$. We denote by $T = T(u_0, v_0) \in (0, +\infty]$ the maximal existence time of the classical solution $(u,v)$ of problem \eqref{PS}. If $T < +\infty$, then the solution blows up in finite time $T$ in the sense that 
$$\lim_{t \to T}(\|u(t)\|_{L^\infty(\RN)} + \|v(t)\|_{L^\infty(\RN)})= +\infty.$$
In that case, $T$ is called the blowup time of the solution. A point $a \in \RN$ is said to be a blowup point of $(u,v)$ if $(u,v)$ is not locally bounded near $(a,T)$ in the sense that $|u(x_n, t_n)| + |v(x_n, t_n)| \to +\infty$ for some sequence $(x_n,t_n) \to (a, T)$ as $n \to +\infty$. We say that the blowup is \textit{simultaneous} if 
\begin{equation}\label{def:simultaneous}
\limsup_{t \to T}\|u(t)\|_{L^\infty(\RN)} = \limsup_{t \to T}\|v(t)\|_{L^\infty(\RN)} = +\infty,
\end{equation}
and that it is \textit{non-simultaneous} if \eqref{def:simultaneous} does not hold, i.e. if one of the two components remains bounded on $\RN \times [0,T)$. For the system \eqref{PS}, it is easy to see that the blowup is always \textit{simultaneous}. Indeed, if $u$ is uniformly bounded on $\RN \times [0, T)$, then the second equation would yield a uniform bound on $v$. More specifically, we say that $u$ and $v$ blow up simultaneously at the same point $a \in \RN$ if $a$ is a blowup point both for $u$ and $v$. \\

\medskip

In the case of a single equation, namely when system \eqref{PS} is reduced to the scalar equation 
\begin{equation}\label{eq:Scalar}
\partial_t u = \Delta u + |u|^{p-1}u, \quad u(\cdot, 0) = u_0, \quad p > 1,
\end{equation}
the blowup question for equation \eqref{eq:Scalar} has been studied intensively by many authors and no list can be exhaustive. Let us sketch the main results for the case of the equation \eqref{eq:Scalar}. Considering $u$ a blowup solution to \eqref{eq:Scalar} and $T$ its blowup time, we know from Giga and Kohn  \cite{GKiumj87} that 
$$\forall (x,t) \in \RN \times [0, T), \quad |u(x,t)| \leq C(T-t)^{-\frac{1}{p-1}},$$
for some positive constant $C$, provided that $1 < p \leq \frac{3N + 8}{3N - 4}$ or $1 < p < \frac{N + 2}{N - 2}$ with $u_0 \geq 0$. This result was extended by Giga, Matsui and Sasayama \cite{GMSiumj04} for all $1 < p < \frac{N + 2}{N-2}$ without assuming the non-negativity of initial data. 

The study of the blow-up behavior of solution \eqref{eq:Scalar} is done through the introduction of similarity variables:
$$W_{T,a}(y,s) = (T-t)^{\frac{1}{p-1}}u(x,t), \quad y = \frac{x - a}{\sqrt{T-t}}, \quad s = -\log(T-t),$$
where $a$ may or not be a blow-up point of $u$. From \eqref{eq:Scalar}, we see that $W_{T,a}$ solves the new equation in $(y,s) \in \RN \times [-\log T, +\infty)$:
\begin{equation}\label{eq:Wa}
\partial_s W_{T,a} = \Delta W_{T,a} - \frac{1}{2}y\cdot \nabla W_{T,a} - \frac{W_{T,a}}{p-1} + |W_{T,a}|^{p-1}W_{T,a}.
\end{equation}
According to Giga and Kohn in \cite{GKcpam89} (see also \cite{GKcpam85, GKiumj87}), we know that: If $a$ is a blow-up point of $u$, then
\begin{equation}\label{equ:limitw}
\lim_{t \to T} (T-t)^\frac{1}{p-1}u(a + y\sqrt{T-t},t) = \lim_{s \to +\infty} W_{T,a}(y,s) = \pm \kappa,
\end{equation}
uniformly on compact sets $|y| \leq R$, where $\kappa = (p-1)^{-\frac{1}{p-1}}$.

This estimate has been refined until the higher order by Filippas, Kohn and Liu \cite{FKcpam92}, \cite{FLaihn93}, Herrero and Vel\'azquez \cite{HVdie92}, \cite{HVaihn93},  \cite{VELcpde92}, \cite{VELiumj93}, \cite{VELtams93}. More precisely, they classified the behavior of $W_{T,a}(y,s)$ for $|y|$ bounded, and showed that one of the following cases occurs (up to replacing $u$ by $-u$ if necessary),\\

\noindent $\bullet$ either there exists $k \in \{1, \cdots, N\}$, 
\begin{equation}\label{equ:c1clas}
\sup_{|y| \leq K\sqrt{s}}\left|W_{T,a}(y,s) - \kappa\left(1 + \frac{p-1}{4ps}\sum_{i = 1}^ky_i^2\right)^{-\frac{1}{p-1}} \right| = \mathcal{O}\left(\frac{\log s}{s}\right).
\end{equation}
\noindent $\bullet$ or there exists an even integer $m \geq 4$ and constant $c_\alpha$ not all zero such that 
\begin{equation*}
\sup_{|y| \leq Ke^{\left(\frac{1}{2} - \frac{1}{m}\right)s}} \left|W_{T,a}(y,s) - \kappa\left(1 + e^{-\left(1 - \frac{m}{2}\right)s}\sum_{|\alpha| = m}c_\alpha y^\alpha\right)^{-\frac{1}{p-1}} \right| = o(1),
\end{equation*}
where the homogeneous multilinear form $\sum_{|\alpha| = m}c_\alpha y^\alpha$ is non-negative.

From Bricmont and Kupiainen \cite{BKnon94}, Herrero and Vel\'azquez \cite{HVaihn93}, we have examples of initial data leading to each of the above mentioned scenarios. Moreover, Herrero and Vel\'azquez \cite{HVasnsp92} proved that the asymptotic behavior \eqref{equ:c1clas} is generic in the one dimensional case, and they announced the same for the higher dimensional case, but they never published it. Note also that the asymptotic profile described in \eqref{equ:c1clas} with $k = N$ has been proved to be stable  with respect to perturbations in the initial data or the nonlinearity  by Merle and Zaag in \cite{MZdm97} (see also Fermanian, Merle and Zaag \cite{FMZma00}, \cite{FZnon00}, Nguyen and Zaag \cite{NZens16} for other proofs of the stability).\\

\medskip

As for system \eqref{PS}, much less result is known, in particular in the study of the asymptotic behavior of the solution near singularities. As far as we know, the only available results concerning the blowup behavior are due to Andreucci, Herrero and Vel\'azquez \cite{AHVihp97} and Zaag \cite{Zcpam01} where the system \eqref{PS} is considered with $\mu = 1$. 

When $\mu = 1$, according to Escobedo and Herrero \cite{EHjde91} (see also \cite{EHpams91}), we know that any nontrivial positive solution of \eqref{PS} which is defined for all $x \in \RN$ must necessarily blow up in finite time if 
$$pq > 1, \quad \text{and} \quad \frac{\max\{p,q\} + 1}{pq - 1} \geq \frac{N}{2},$$
and both functions $u(x,t)$ and $v(x,t)$ must blow up simultaneously. See also \cite{EHampa93} for the case of boundary value problems.

In \cite{AHVihp97}, the authors proved that if
\begin{equation} \label{eq:condAHV}
pq > 1, \quad \text{and} \quad q(p(N - 2)) < N + 2 \quad \text{or} \quad p(q(N-2))< N+2,
\end{equation}
then every positive solution $(u,v)$ of \eqref{PS} exhibits the \textit{Type I} blowup, namely that there exists some constant $C > 0$ such that
\begin{equation}\label{eq:blrate}
\|u(t)\|_{L^\infty(\RN)} \leq C\bar{u}(t), \quad \|v(t)\|_{L^\infty(\RN)} \leq C\bar{v}(t),
\end{equation}
where $(\bar{u}, \bar{v})$ solves the following ODE system 
$$\bar{u}' = \bar{v}^p, \quad \bar{v}' = \bar{u}^q, \quad \bar{u}(T) = \bar{v}(T) = +\infty,$$
whose solution is explicitly given by
$$\bar{u}(t) = \Gamma (T-t)^{-\frac{p+1}{pq-1}}, \quad \bar{v}(t) = \gamma(T-t)^{-\frac{q+1}{pq - 1}}$$
where $(\Gamma, \gamma)$ defined by 
\begin{equation}\label{def:Gamgam}
\gamma^p = \Gamma\left(\frac{p+1}{pq - 1}\right),\quad \Gamma^q = \gamma\left(\frac{q + 1}{pq - 1}\right).
\end{equation}
The estimate \eqref{eq:blrate} has also been proved by Caristi and Mitidieri \cite{CMjde94} in a ball under assumptions on $p$ and $q$ different from \eqref{eq:condAHV}. See also Deng \cite{Dzamp96}, Fila and Souplet \cite{FSnodea01} for other results relative to estimate \eqref{eq:blrate}.

The study of blowup solutions for system \eqref{PS} is done through the introduction of the following similarity variables for all $a \in \RN$ ($a$ may or may not be a blowup point):
\begin{equation}\label{def:simiVars}
\begin{array}{c}
\Phi_{T,a}(y,s) = (T-t)^{\frac{p+1}{pq -1}}u(x,t), \quad \Psi_{T,a}(y,s) = (T-t)^{\frac{q + 1}{pq-1}}v(x,t),\\
\\
\text{where}\quad y = \dfrac{x-a}{\sqrt{T - t}}, \quad s = -\log(T-t).
\end{array}
\end{equation}
From \eqref{PS}, $(\Phi_{T,a},\Psi_{T,a})$ (or $(\Phi, \Psi)$ for simplicity) satisfy the following system: for all $(y,s) \in \RN \times [-\log T, +\infty)$,
\begin{equation}\label{eq:PhiPsi}
\begin{array}{ll}
\partial_s\Phi &= \Delta\Phi - \dfrac{1}{2}y\cdot \nabla \Phi - \left(\dfrac{p+1}{pq-1}\right)\Phi + |\Psi|^{p-1}\Psi,\\
&\\
\partial_s\Psi &= \mu\Delta\Psi - \dfrac{1}{2}y\cdot \nabla \Psi - \left(\dfrac{q+1}{pq-1}\right)\Psi + |\Phi|^{p-1}\Phi.
\end{array}
\end{equation}
Assuming \eqref{eq:blrate} holds, namely that
$$\forall a \in \RN, \quad  \|\Phi_{T,a}(s)\|_{L^\infty(\RN)} + \|\Psi_{T,a}(s)\|_{L^\infty(\RN)} \leq C, \quad \forall s \geq -\log T,$$
and considering $a \in \RN$ a blowup point of $(u,v)$, we know from \cite{AHVihp97} that (remind that we are considering the case when $\mu = 1$)\\
\noindent $\bullet$ either $(\Phi_{T,a}, \Psi_{T,a})$ goes to $(\Gamma, \gamma)$ exponentially fast, \\
\noindent $\bullet$ or there exists $k \in \{1, \cdots, N\}$ such that after an orthogonal change of space coordinates and up to replacing $(u,v)$ by $(-u,-v)$ if necessary, 
\begin{equation}\label{eq:beh}
\begin{array}{ll}
\Phi_{T,a}(y,s) &= \Gamma - \dfrac{c_1}{s}(p+1)\Gamma\sum \limits_{i=1}^k(y_i^2 - 2) + o\left(\dfrac{1}{s}\right),\\
&\\
\Psi_{T,a}(y,s) &= \gamma - \dfrac{c_1}{s}(q+1)\gamma\sum\limits_{i=1}^k(y_i^2 - 2) + o\left(\dfrac{1}{s}\right),
\end{array}
\end{equation}
where $(\Gamma, \gamma)$ is given by \eqref{def:Gamgam} and
\begin{equation}\label{eq:valueb}
c_1 = c_1(p,q) = \frac{2pq + p + q}{8pq(p+1)(q+1)},
\end{equation}
and the convergence takes place in $\mathcal{C}^\ell_{loc}(\RN)$ for any $\ell \geq 0$.\\
In the first case, we have other profiles, some of them are different from those occurring in the scalar case of \eqref{eq:Scalar}, see Theorem 3 and 4 in \cite{AHVihp97} for more details. Note that the value of $c_1$ given in \eqref{eq:valueb} was not precised in \cite{AHVihp97}, but we can justify it by explicit computations as in \cite{AHVihp97}.

Beside the results already cited, let us mention to the work by Zaag \cite{Zcpam01} where the author obtained a Liouville theorem  for system \eqref{PS} that improves the results of \cite{AHVihp97}. Based on this theorem, he was able to derive sharp estimates of asymptotic behaviors as well as a localization property for blowup solutions of \eqref{PS}. For other aspects of system \eqref{PS}, especially concerning the blowup set, see Friedman and Giga \cite{FGfsut87}, Mahmoudi, Souplet and Tayachi \cite{MSTjde15}, Souplet \cite{Sjems09}.\\

\medskip

In this paper, we want to study the profile of the solution of \eqref{PS} near blowup, and the stability of such behavior with respect to perturbations in initial data. More precisely, we prove the following result.

\begin{theorem}[Existence of a blow-up solution for system \eqref{PS} with the description of its profile]\label{theo1} Consider $a \in \RN$. There exists $T > 0$ such that system \eqref{PS} has a solution $(u,v)$ defined on $\RN \times [0, T)$ such that:\\
\noindent $(i)$ $u$ and $v$ blow up in finite time $T$ simultaneously at one blowup point $a$ and only there.\\
\noindent $(ii)$ There holds that 
\begin{equation}\label{eq:asyTh1}
\begin{array}{l}
\left\|(T-t)^{\frac{p+1}{pq - 1}}u(x,t) - \Phi^*\left(\dfrac{x - a}{\sqrt{(T-t)|\log (T-t)|}}\right)\right\|_{L^\infty(\RN)} \leq \dfrac{C}{\sqrt{|\log (T-t)|}},\\
\\
\left\|(T-t)^{\frac{q+1}{pq - 1}}v(x,t) - \Psi^*\left(\dfrac{x - a}{\sqrt{(T-t)|\log (T-t)|}}\right)\right\|_{L^\infty(\RN)} \leq \dfrac{C}{\sqrt{|\log (T-t)|}},
\end{array}
\end{equation}
where 
\begin{equation}\label{def:fgpro}
\Phi^*(z) = \Gamma(1 + b|z|^2)^{-\frac{p+1}{pq - 1}} \quad \text{and} \quad  \Psi^*(z) = \gamma(1 + b|z|^2)^{-\frac{q+1}{pq - 1}},
\end{equation}
with $(\Gamma, \gamma)$ given by \eqref{def:Gamgam} and 
\begin{equation}\label{def:val_b}
b = b(p,q,\mu) = \frac{(pq-1)(2pq + p + q)}{4pq(p+1)(q+1)(1+\mu)} > 0.
\end{equation}

\noindent $(iii)\;$ for all $x \neq a$, $(u(x,t), v(x,t)) \to (u^*(x), v^*(x)) \in \Cc^2(\RN \backslash \{0\}) \times \Cc^2(\RN \backslash \{0\})$ with
$$u^*(x) \sim \Gamma \left(\frac{b|x - a|^2}{2|\log |x - a||} \right)^{-\frac{p+1}{pq-1}} \quad \text{and} \quad v^*(x) \sim \gamma \left(\frac{b|x - a|^2}{2|\log |x - a||} \right)^{-\frac{q+1}{pq-1}},$$
as $|x - a| \to 0$.
\end{theorem}

\begin{remark} The derivation of the blowup profile  \eqref{def:fgpro} can be understood through a formal analysis in Section \ref{sec:forap} below. However, we would like to emphasize on the fact that the particular value of $b = b(p,q,\mu) > 0$ given in \eqref{def:val_b} is crucially needed in various algebraic identities  in the rigorous proof.
\end{remark}

\begin{remark} The initial data for which system \eqref{PS} has a solution blowing up in finite time $T$ at only one blowup point $a$ and verifying \eqref{eq:asyTh1} is given by formula \eqref{eq:intialdata}, which is expressed in the original variables as follows:
\begin{align*}
u_0(x) &= T^{-\frac{p+1}{pq-1}}\left\{\frac{A \Gamma (p+1)}{|\log T|^2} \left(d_0 + d_1 \cdot \frac{x - a}{\sqrt T}\right)\chi_0 \left(\frac{x - a}{K\sqrt{|\log T| T}} \right)\right.\\
& \qquad  + \left. \Phi^*\Big(\frac{|x - a|}{\sqrt {|\log T| T}} \Big) + \frac{2b \Gamma (p\mu + 1)}{|\log T|(pq-1)} \right\},\\
v_0(x) &= T^{-\frac{q+1}{pq-1}}\left\{\frac{A \gamma (q+1)}{|\log T|^2} \left(d_0 + d_1 \cdot \frac{x - a}{\sqrt T}\right)\chi_0 \left(\frac{|x - a|}{K\sqrt{|\log T| T}} \right) \right.\\
& \qquad + \left. \Psi^*\Big(\frac{x - a}{\sqrt {|\log T| T}} \Big) + \frac{2b \gamma (q+ \mu)}{|\log T|(pq-1)} \right\},
\end{align*}
where  $(\Gamma,\gamma)$ is given by \eqref{def:Gamgam}, $A$ and $K$ are positive constants fixed sufficiently large, $d_0 \in \Rb$ and $d_1 \in \RN$ are parameters in our proof, and $\chi_0 \in \Cc^\infty_0([0, +\infty))$ with $\text{supp} (\chi_0) \subset [0,2]$ and $\chi_0 \equiv 1$ on $[0,1]$.
\end{remark}

\begin{remark} We will only give the proof when $N = 1$. Indeed, the computation of the eigenfunctions (Lemma \ref{lemm:diagonal}) of the linearized operator $\Hc + \Mc$ defined in \eqref{def:opH} and \eqref{def:matM}  and the projection of \eqref{eq:LamUp} on the eigenspaces (Lemma \ref{lemm:DefProjection}) become much more complicated when $N \geq 2$. Besides, the ideas are exactly the same.
\end{remark}

\begin{remark} Note that the constructed solution in Theorem \ref{theo1} is of \textit{Type I}, which means that it satisfies \eqref{eq:blrate}. Therefore, our result indicates that there exist solutions to \eqref{PS} exhibiting the \textit{Type I} blowup for all $p, q > 1$ and $N \geq 1$, even when \eqref{eq:condAHV} doesn't hold.
\end{remark}

\begin{remark} The result of Theorem \ref{theo1} holds for more general nonlinearities than \eqref{PS}, namely that the nonlinear terms in \eqref{PS} are replaced by 
$$F(u,v) = |u|^{p-1}u + f(u,v, \nabla u, \nabla v)\quad  \text{and} \quad G(u,v) = |v|^{q-1}v + g(u,v,  \nabla u, \nabla v),$$
where 
$$|f(u,v,\nabla u, \nabla v)|  \leq C(1 + |u|^{p_1} + |v|^{q_1} + |\nabla u|^{r_1} + |\nabla v|^{s_1}),$$
and 
$$|g(u,v,\nabla u, \nabla v)| \leq C(1 + |u|^{p_2} + |v|^{q_2} + |\nabla u|^{r_2} + |\nabla v|^{s_2}),$$
where 
$$0 \leq p_1 < \frac{p(q+1)}{p+1}, \; 0 \leq q_1 < p, \quad 0 \leq r_1 < \frac{p(q+1)}{p + \frac{1}{2}pq + \frac 12}, \quad 0 \leq s_1 < \frac{p(q+1)}{q + \frac{1}{2}pq + \frac 12},$$
and 
$$0 \leq p_2 < q, \; 0 \leq q_2 < \frac{q(p+1)}{q+1}, \quad 0 \leq r_2 < \frac{q(p+1)}{p + \frac{1}{2}pq + \frac 12}, \quad 0 \leq s_2 < \frac{q(p+1)}{q + \frac{1}{2}pq + \frac 12}.$$
Note that in the setting \eqref{def:simiVars}, the terms $f$ and $g$ turn to be exponentially small. Therefore, a perturbation of our method works although we need in addition some parabolic regularity results in order to handle the nonlinear gradient terms (see \cite{EZsema11} and \cite{TZpre15} for such parabolic regularity techniques). For simplicity, we only give the proof when the nonlinear terms are exactly given by $F(u,v) = |v|^{p-1}v$ and $G(u,v) = |u|^{q-1}u$. 
\end{remark}

\begin{remark} Our method can be naturally extended to the system of $m$ equations of the form
\begin{equation}\label{sys:meqs}
\left\{\begin{array}{ll}
\partial_t u_i &= \mu_i \Delta u_i + |u_{i+1}|^{p_i - 1}u_{i+1}, \quad i = 1, 2, \cdots, m-1,\\
\partial_t u_m &= \mu_m \Delta u_m + |u_1|^{p_m - 1}u_1,
\end{array}\right.
\end{equation}
where $p_i > 1$ and $\mu_i > 0$ for $i = 1, 2, \cdots, m$. Up to a complication in parameters, we suspect that our analysis yields the existence of a solution for \eqref{sys:meqs} which blows up in finite time $T$ only at one blowup point $a \in \RN$ and satisfies the asymptotic behavior: for $i = 1, 2, \cdots, m$,
$$(T-t)^{\alpha_i}u_i(x,t) \sim \gamma_i \left(1 + \frac{B|x - a|^2}{(T-t)|\log (T-t)|} \right)^{-\alpha_i} \quad \text{as}\; t \to T,$$
where $B = B(p_i, \mu_i) > 0$, $\gamma_i$ is given by
$$\gamma_1^{p_m} = \gamma_m \alpha_m, \quad \gamma_{i+1}^{p_i} = \gamma_i \alpha_i \quad \text{for}\quad i = 1,2, \cdots, m-1,$$
and 
$$\begin{pmatrix}
\alpha_1\\ \alpha_2 \\ \vdots \\ \alpha_{m-1} \\ \alpha_m
\end{pmatrix} = \begin{pmatrix}
-1 & p_1 & 0 &\cdots &  0\\
0 & -1 & p_2 & 0 &\cdots\\
\vdots & \ddots & \ddots &\ddots & \vdots\\
0 & \cdots & 0 & -1 &p_{m-1}\\
p_m & 0 & \cdots & 0 & -1
\end{pmatrix}^{-1} \begin{pmatrix}
1 \\ 1 \\ \vdots\\ 1 \\1
\end{pmatrix}.
$$
\end{remark}

\bigskip

As a consequence of our techniques, we show the stability of the constructed solution with respect to perturbations in initial data. More precisely, we have the following result.
\begin{theorem}[Stability of the blowup profile \eqref{def:fgpro}] \label{theo2} Let $(\hat u_0, \hat v_0)$ be the initial data of system \eqref{PS} such that the corresponding solution $(\hat u, \hat v)$ blows up in finite time $\hat T$ at only one blowup point $\hat a$ and $(\hat u(x, t), \hat v(x,t))$ satisfies \eqref{eq:asyTh1} with $T = \hat T$ and $a = \hat a$. Then, there exists a neighborhood $\mathscr{W}_0$ of $(\hat u_0, \hat v_0)$ in $L^\infty(\RN) \times L^\infty(\RN)$ such that for any $(u_0, v_0) \in \mathscr{W}_0$, system \eqref{PS} has a unique solution $(u,v)$ with initial data $(u_0, v_0)$ which blows up in finite time $T(u_0,v_0)$ at only one blowup point $a(u_0,v_0)$. Moreover, parts $(ii)$ and $(iii)$ of Theorem \ref{theo1} are satisfied, and
$$|T(u_0, v_0) - \hat T| + |a(u_0,v_0) - \hat a| \to 0$$
as $(u_0,v_0) \to (\hat u_0, \hat v_0)$ in $L^\infty(\RN) \times L^\infty(\RN)$.
\end{theorem}
\begin{remark} With the stability result, we expect that the blowup profile \eqref{def:fgpro} is generic, i.e. there exists an open, everywhere dense set $\mathscr{U}_0$ of initial data whose corresponding solution to \eqref{PS} either converges to the steady state \eqref{def:Gamgam} or blows up in finite time at a single point, according the asymptotic behavior \eqref{eq:asyTh1}. In particular, we suspect that a numerical simulation of \eqref{PS} should lead to the profile \eqref{def:fgpro}. Up to our knowledge, the only available proof for the genericity is given by Herrero and Vel\'azquez \cite{HVasnsp92} for the case of equation \eqref{eq:Scalar} in one-dimensional case. As in \cite{HVasnsp92}, a first step towards the genericity of the profile \eqref{def:fgpro} is to classify all possible asymptotic behaviors of the blowup solution of \eqref{PS} which was established in \cite{AHVihp97} (see also \cite{Zcpam01}) in the case when $\mu = 1$. 
\end{remark}

\bigskip

Let us now give the main idea of the proof of Theorem \ref{theo1}. Our proof uses some ideas developed by Merle and Zaag \cite{MZdm97} and Bricmont and Kupiainen \cite{BKnon94} for the equation \eqref{eq:Scalar}. This kind of method has been proved to be successful for various situations including parabolic and hyperbolic equations. For the parabolic equations, we would like to mention the work by Masmoudi and Zaag \cite{MZjfa08} (see also the earlier work by Zaag \cite{ZAAihn98}) for the complex Ginzburg-Landau equation with no gradient structure,
\begin{equation}\label{eq:GL}
\partial_t u = (1 + \imath \beta)\Delta u + (1 + \imath \delta)|u|^{p-1}u - \alpha u,
\end{equation}
where $u(t): x \in \RN \to u(x,t) \in \mathbb{C}$, $p > 1$, $(\beta, \delta, \alpha) \in \Rb^3$ satisfying $$p - \delta^2 - \beta \delta (p+1) > 0.$$
There are also the works by Nguyen and Zaag \cite{NZsns16} for a logarithmically perturbed equation of \eqref{eq:Scalar} (see also Ebde and Zaag \cite{EZsema11} for a weakly perturbed version of \eqref{eq:Scalar}), by Nouaili and Zaag \cite{NZcpde15} for a non-variational complex-valued semilinear heat equation, or the recent work by Tayachi and Zaag \cite{TZpre15} for the nonlinear heat equation with a critical power nonlinear gradient term,
$$\partial_t u = \Delta u + |u|^{p-1}u + \mu |\nabla u|^\frac{2p}{p+1} \quad \text{with}\;\;  p > 3, \; \mu > 0.$$
When $p \to +\infty$, this equation is reduced to 
$$\partial_t u = \Delta u + e^u + \mu |\nabla u|^2,$$
which is studied in \cite{GNZpre16a}. There are also the cases for the construction of multi-solitons for the semilinear wave equation in one space dimension by C\^ote and Zaag \cite{CZcpam13}, for the wave maps by Rapha\"el and Rodnianski \cite{RRmihes12}, for the Schr\"odinger maps by Merle, Rapha\"el and Rodnianski \cite{MRRmasp11}, for the critical harmonic heat flow by Schweyer \cite{Schfa12} and for the two-dimensional Keller-Segel equation by Rapha\"el and Schweyer \cite{RSma14}, Ghoul and Masmoudi \cite{GMarx16}.

One may think that the method used in \cite{MZdm97} and \cite{BKnon94} should work the same for system \eqref{PS} perhaps with some technical complications. This is not the case, since the fact that $\mu \neq 1$ breaks any symmetry in the problem, and makes the diffusion operator associated to \eqref{PS} not self-adjoint. In other words, the method we present here is not based on a simple perturbation of the equation \eqref{eq:Scalar} treated in \cite{MZdm97} and \cite{BKnon94}. More precisely, our proof relies on the understanding of the dynamics of the selfsimilar version \eqref{eq:PhiPsi} around the profile \eqref{def:fgpro}. In the setting \eqref{def:simiVars}, constructing a solution for \eqref{PS} satisfying \eqref{eq:asyTh1} is equivalent to construct a solution for \eqref{eq:PhiPsi} such that
$$\binom{\Lambda}{\Upsilon}(y,s) = \binom{\Phi}{\Psi}(y,s) - \binom{\Phi^*}{\Psi^*}\left(\frac{y}{\sqrt s}\right) \to \binom{0}{0} \quad \text{as} \; s \to +\infty.$$
Satisfying such a property is guaranteed by a condition that $\binom{\Lambda(s)}{\Upsilon(s)}$ belongs to some set $\Vc_A(s) \subset L^\infty(\RN) \times L^\infty(\RN)$ which shrinks to $0$ as $s \to +\infty$ (see Definition \ref{def:VA} below for an example). Since the linearization of system \eqref{eq:PhiPsi} around the profile $\binom{\Phi^*}{\Psi^*}$ gives $N + 1$ positive modes, $\frac{N(N+1)}{2}$
zero modes, and an infinite dimensional negative part (see Lemma \ref{lemm:diagonal} and Remark \ref{rema:012}), we can use the method of \cite{MZdm97} and \cite{BKnon94} which relies on two arguments:\\

\noindent - The use of the bounding effect of the heat kernel (see Proposition \ref{prop:dyn}) to reduce the problem of the control of $\binom{\Lambda(s)}{\Upsilon(s)}$ in $\Vc_A(s)$ to the control of its positive modes. Note that the linearized operator around the profile, that is $\Hc + \Mc$ defined in \eqref{def:opH} and \eqref{def:matM}, is not self-adjoint. This is one of the major difficulties arising in this paper. \\
\noindent - The control of the positive modes thanks to a topological argument based on the index theory.\\

In addition to the difficulties concerning the linearized operator mentioned above, we also deal with the number of parameters in the problem ($p, q$, and $\mu$) leading to actual complications in the analysis. According to the general framework of \cite{MZdm97}, some crucial modifications are needed. In particular, we have to overcome the following challenges:
\begin{itemize}
\item[(i)] Finding the profile $(\Phi^*, \Psi^*)$ is not obvious, in particular in determining the values of $b$ given by \eqref{eq:valueb}, which is crucial in many algebraic identities in the rigorous analysis. See Section \ref{sec:forap} for a formal analysis to justify such a profile. We emphasize that the formal approach actually gives us an appreciated profile to be linearized around (see \eqref{def:varphiys} and \eqref{def:psiys}).
\item[(ii)] Defining the shrinking set $\Vc_A$ (see Definition \ref{def:VA}) to trap the solution. Note that our definition of $\Vc_A$ is different from that of \cite{MZdm97}. Here, we follow the idea of \cite{MZjfa08} to find out such an appreciated definition for $\Vc_A$. In particular, it comes from many relations in our proof, one of them is related to the dynamics of the linearized problem stated in Proposition \ref{prop:dyn}. 
\item[(iii)] A good understanding of the dynamics of the linearized operator $\Hc + \Mc + V$ of equation
\eqref{eq:LamUp} around the appreciated profile $(\varphi, \psi)$ given in \eqref{def:varphiys} and \eqref{def:psiys} is needed, according to the definition of the shrinking
set $\Vc_A$. Because the behavior of the potential $V$ defined in \eqref{def:Vys} inside and outside the \textit{blowup region} is different, the effect of the linearized operator is therefore considered accordingly to this region. Outside the \textit{blowup region}, the linear operator $\Hc + \Mc + V$ behaves as one with fully negative spectrum, which greatly simplifies the analysis in this region (see Section \ref{sec:outerpart}). Inside the \textit{blowup region}, the potential $V$ is considered as a perturbation of the effect of $\Hc + \Mc$, therefore, a good study of the spectral properties of $\Hc + \Mc$ is needed. Note that the linear operator $\Hc + \Mc$ is not diagonal, but it is diagonalized (see Lemma \ref{lemm:diagonal}). Using this diagonalization, we then define the projection on subspaces of the spectrum of $\Hc + \Mc$ (see Lemma \ref{lemm:DefProjection}). 
\end{itemize}

For the proof of single blowup point (part $(i)$ of Theorem \ref{theo1}), we use part $(ii)$ and an extended result of \cite{GKcpam89} that is called \textit{no blow-up under some threshold} criterion for parabolic inequalities (see Proposition \ref{prop:Noblowup}). The derivation of the final profile $(u^*(x), v^*(x))$ (part $(iii)$ of Theorem \ref{theo1}) follows from part $(ii)$ by using the same argument as \cite{ZAAihn98} and \cite{Mercpam92}.\\

The rest of the paper is organized as follows:\\
\noindent - In Section \ref{sec:forap}, we first explain formally how we obtain the profile $(\Phi^*, \Psi^*)$ and give a suggestion for an appreciated profile to be linearized around.\\
\noindent - In Section \ref{sec:formu}, we give a formulation of the problem in order to justify the formal argument. We also give the spectral properties of the linear operator $\Hc + \Mc$ as well as the definition of the projection on eigenspaces of $\Hc + \Mc$. \\
\noindent - In Section \ref{sec:existence}, we give all the argument of the proof of Theorem \ref{theo1} assuming technical results, which are left to the next section.\\
\noindent - Section \ref{sec:dyn} is central in our analysis. It is devoted to the study of the dynamics of the linearized problem. In particular, we prove Proposition \ref{prop:dyn} from which we reduce the problem to a finite dimensional one.\\
\noindent - In Section \ref{sec:stab}, we give the proof of Theorem \ref{theo2}. Since its proof is a consequence of the existence proof (part $(ii)$ of Theorem \ref{theo1}), thanks to a geometrical interpretation of quantities of blowup parameters whose dimension is equal to the dimension of the finite dimensional problem, we only explain the main ideas of the proof there.

\section{A formal analysis.}\label{sec:forap}
In this section, we give a formal analysis leading to the asymptotic behaviors described in \eqref{eq:asyTh1} by means of matching asymptotic. For simplicity, we shall look for $(u,v)$, a positive solution of \eqref{PS} in one dimensional case. By the translation invariant in space, we assume that $(u,v)$ blows up in finite time $T > 0$ at the origin, and write $(\Phi, \Psi)$ instead of $(\Phi_{T,a}, \Psi_{T,a})$ for short. From the transformation \eqref{def:simiVars}, the behavior \eqref{eq:asyTh1} is equivalent to showing that 
\begin{equation}\label{eq:goalformalcl}
\Phi(y,s) \sim \Gamma\left(1 + \frac{b |y|^2}{s}\right)^{-\frac{p+1}{pq - 1}} \;\; \text{and} \;\; \Psi(y,s) \sim \gamma\left(1 + \frac{b |y|^2}{s}\right)^{-\frac{q+1}{pq - 1}},
\end{equation}
as $s \to +\infty$, where $\Gamma$, $\gamma$ are defined in \eqref{def:Gamgam} and $b$ is given in \eqref{def:val_b}.\\

We use here the method of \cite{MZjfa08} treated for the complex Ginzburg-Landau equation, which was slightly adapted from the method of Berger and Kohn \cite{BKcpam88} for equation \eqref{eq:Scalar}.  Following the approach of \cite{MZjfa08}, we try to search formally for system \eqref{eq:PhiPsi} a regular solution $(\Phi, \Psi)$ of the form 
\begin{equation}\label{eq:formPhiPsi}
\begin{array}{ll}
\Phi(y,s) &= \Phi_0\left(\frac{y}{\sqrt{s}}\right) + \frac{1}{s}\Phi_1\left(\frac{y}{\sqrt{s}}\right) + \cdots,\\
\Psi(y,s) &= \Psi_0\left(\frac{y}{\sqrt{s}}\right) + \frac{1}{s}\Psi_1\left(\frac{y}{\sqrt{s}}\right) + \cdots 
\end{array}
\end{equation}

Injecting \eqref{eq:formPhiPsi} into \eqref{eq:PhiPsi} and comparing elements of order $\frac{1}{s^j}$ with $j = 0, 1,\cdots$, we obtain for $j = 0$, 
\begin{equation}\label{eq:Phi0Psi0}
\begin{array}{c}
-\dfrac{z}{2}\Phi_0' - \dfrac{p+1}{pq-1}\Phi_0 + \Psi_0^p = 0,\\
\\
-\dfrac{z}{2}\Psi_0' - \dfrac{q+1}{pq-1}\Psi_0 + \Phi_0^q = 0,\\
\end{array} \quad \text{where}\quad z = \frac{y}{\sqrt{s}},
\end{equation}
and for $j = 1$, 
\begin{equation}\label{def:FG}
\begin{array}{l}
F(z):= \dfrac{z}{2}\Phi_1' + \left(\dfrac{p+1}{pq-1}\right)\Phi_1 - p\Psi_0^{p-1}\Psi_1 - \dfrac{z}{2}\Phi_0' - \Phi_0'' = 0,\\
\\
G(z):= \dfrac{z}{2}\Psi_1' + \left(\dfrac{q+1}{pq-1}\right)\Psi_1 - q\Phi_0^{q-1}\Phi_1 - \dfrac{z}{2}\Psi_0' - \mu\Psi_0'' = 0.\\
\end{array}
\end{equation}

Solving system \eqref{eq:Phi0Psi0} equipped with data at zero
$$\Phi_0(0) = \Gamma \quad \text{and} \quad \Psi_0(0) = \gamma,$$
we derive
\begin{equation}\label{eq:solPhi0Psi0}
\Phi_0(z) = \Gamma(1 + bz^2)^{-\frac{p+1}{pq-1}} \quad \text{and} \quad \Psi_0(z) = \gamma(1 + bz^2)^{-\frac{q+1}{pq-1}},
\end{equation}
for some integration constant $b$, and $(\Gamma, \gamma)$ is given by \eqref{def:Gamgam}. Since we want $(\Phi, \Psi)$ to be regular, we impose the condition 
$$b > 0.$$

Let us now determine the value of $b$ in \eqref{eq:solPhi0Psi0}. To do so, we first evaluate $F$ and $G$ at $z = 0$ by using \eqref{eq:solPhi0Psi0} to find 
\begin{equation*}
\begin{array}{c}
\left(\dfrac{p+1}{pq-1}\right)\Phi_1(0) - p\gamma^{p-1}\Psi_1(0) + 2b\left(\dfrac{p+1}{pq-1} \right)\Gamma = 0,\\
\\
\left(\dfrac{q+1}{pq-1}\right)\Psi_1(0) - q\Gamma^{q-1}\Phi_1(0) + 2\mu b\left(\dfrac{q+1}{pq-1} \right)\gamma = 0.\\
\end{array}
\end{equation*}
Using the definition of $(\Gamma, \gamma)$ given in \eqref{def:Gamgam}, one can simplify this system and obtain 
\begin{equation}\label{eq:Phi1Psi1z0}
\Phi_1(0) = \frac{2b \Gamma(p \mu + 1)}{pq-1} \quad \text{and} \quad \Psi_1(0) = \frac{2b\gamma (q + \mu)}{pq-1}.
\end{equation}
Let us now expand $(\Phi_1, \Psi_1)$ in power of $z$, namely 
\begin{equation}\label{eq:expanPhi1Psi1}
\begin{array}{l}
\Phi_1(z) = \Phi_1(0) + d_1z + d_2 z^2 + \Oc(z^3),\\
\Psi_1(z) = \Psi_1(0) + e_1z + e_2 z^2 + \Oc(z^3).
\end{array}
\end{equation}
Injecting these forms into \eqref{def:FG} and expanding $F$ and $G$ in powers of $z$, we obtain at the order $z$, 
\begin{equation*}
\begin{array}{l}
\left(\dfrac{1}{2} + \dfrac{\gamma^p}{\Gamma} \right)d_1 - p\gamma^{p-1}e_1 = 0,\\
\\
- q\Gamma^{q-1}d_1 + \left(\dfrac{1}{2} + \dfrac{\Gamma^q}{\gamma} \right)e_1 = 0,\\
\end{array}
\end{equation*}
which yields 
$$0 = \left(\frac{1}{2}\gamma^{p+1} + \frac{1}{2}\Gamma^{q+ 1} + \frac{1}{4}\Gamma \gamma - (pq-1)\Gamma^q\gamma^p\right)e_1 := Ae_1.$$
A straightforward computation gives $A < 0$, hence, 
$$d_1= e_1 = 0.$$
For the terms of order $z^2$ in the expansion of $F$ and $G$, we have 
\begin{equation*}
\begin{array}{l}
\left(\dfrac{1}{\gamma^p} + \dfrac{1}{\Gamma} \right)d_2 - \dfrac{p}{\gamma}e_2 + \dfrac{2b^2p(q + 1)(p-1)(q + \mu)}{(pq - 1)^2} - \dfrac{6b^2p(q+1)}{pq-1} + b = 0,\\
\\
\left(\dfrac{1}{\Gamma^q} + \dfrac{1}{\gamma} \right)e_2 - \dfrac{q}{\Gamma}d_2 + \dfrac{2b^2q(p + 1)(q-1)(p\mu + 1)}{(pq - 1)^2} - \dfrac{6\mu b^2q(p+1)}{pq-1} + b = 0.
\end{array}
\end{equation*}
Multiplying the second equation by $\frac{p(q+1)}{q(p+1)}$, then combining with the first equation, we find that the coefficients of $d_2$ and $e_2$ disappear leading to
\begin{equation*}
b = \frac{(pq-1)(2pq + p + q)}{4pq(p+1)(q+1)(1+\mu)},
\end{equation*}
which is the desired result. Note that our computation fits with the result of the case $\mu = 1$ by combining \eqref{eq:solPhi0Psi0}, \eqref{eq:beh} and \eqref{eq:valueb}. \\

In conclusion, we obtain the following profile for $(\Phi(y,s), \Psi(y,s))$:
$$(\Phi(y,s), \Psi(y,s)) \sim (\varphi(y,s), \psi(y,s)),$$
where
\begin{equation}\label{def:varphiys}
\varphi(y,s) = \Phi_0(\frac{y}{\sqrt{s}}) + \frac{1}{s}\Phi_1(0) = \Gamma\left(1 + \frac{b|y|^2}{s}\right)^{-\frac{p+1}{pq-1}} + \frac{2b\Gamma(p\mu + 1)}{(pq - 1)s},
\end{equation}
\begin{equation}\label{def:psiys}
\psi(y,s) =\Psi_0(\frac{y}{\sqrt{s}}) + \frac{1}{s}\Psi_1(0) = \gamma\left(1 + \frac{b|y|^2}{s}\right)^{-\frac{q+1}{pq-1}} +\frac{2b\gamma(q + \mu)}{(pq - 1)s},
\end{equation}
with $b$ given in \eqref{def:val_b}.

\section{Formulation of the problem.}\label{sec:formu}
 In this section, we give a formulation for the proof of Theorem \ref{theo1}. We will only give the proof in one dimensional case ($N = 1$) for simplicity, but the proof remains the same for higher dimensions $N \geq 2$. We want to prove the existence of suitable initial data $(u_0, v_0)$ so that the corresponding solution $(u,v)$ of system \eqref{PS} blows up in finite time $T$ only at one point $a \in \Rb$ and verifies \eqref{eq:asyTh1}. From translation invariance of equation \eqref{PS}, we may assume that $a = 0$. Through the transformation \eqref{def:simiVars}, we want to find $s_0 > 0$ and $(\Phi(y,s_0), \Psi(y,s_0))$ such that the solution $(\Phi, \Psi)$ of system \eqref{eq:PhiPsi} with initial data $(\Phi(y,s_0), \Psi(y,s_0))$ satisfies
\begin{equation}\label{eq:limPhiPsis3}
\lim_{s\to+\infty} \left\| \Phi(y,s) - \Phi^*\left(\frac{y}{\sqrt{s}}\right)\right\|_{L^\infty(\RN)} = \lim_{s\to+\infty} \left\| \Psi(y,s) - \Psi^*\left(\frac{y}{\sqrt{s}}\right)\right\|_{L^\infty(\RN)} = 0,
\end{equation}
where $\Phi^*$ and $\Psi^*$  are given by \eqref{def:fgpro}.

According to the formal analysis in the previous section, let us introduce $\Lambda(y,s)$ and $\Upsilon(y,s)$ such that 
\begin{equation}\label{def:LamUps}
\Phi(y,s) = \Lambda(y,s) + \varphi(y,s), \quad \Psi(y,s) = \Upsilon(y,s) + \psi(y,s),
\end{equation}
where $\varphi$ and $\psi$ are given in \eqref{def:varphiys} and \eqref{def:psiys}.

With the introduction of $(\Lambda,\Upsilon)$ in \eqref{def:LamUps}, the problem is then reduced to constructing functions $(\Lambda,\Upsilon)$ such that
\begin{equation*}
\lim_{s \to+\infty}\|\Lambda(s)\|_{L^\infty(\RN)} = \lim_{s \to+\infty}\|\Upsilon(s)\|_{L^\infty(\RN)} = 0.
\end{equation*}
and $(\Lambda,\Upsilon)$ satisfies the following system:
\begin{equation}\label{eq:LamUp}
\partial_s \binom{\Lambda}{\Upsilon} = \Big(\Hc + \Mc + V(y,s)\Big)\binom{\Lambda}{\Upsilon} + \binom{F_1(\Upsilon, y,s)}{F_2(\Lambda, y,s)} + \binom{R_1(y,s)}{R_2(y,s)},
\end{equation}
where 
\begin{equation}\label{def:opH}
\Hc = \begin{pmatrix}
\Lc_1 & 0 \\ 0 & \Lc_\mu
\end{pmatrix} \quad \text{where }\;\; \Lc_\eta = \eta\Delta - \frac{1}{2}y\cdot \nabla, \quad \eta = \{1, \mu\},
\end{equation}
\begin{equation}\label{def:matM}
\Mc = \begin{pmatrix}
 -\frac{p+1}{pq-1} &\; p\gamma^{p-1}\\
q\Gamma^{q-1} &\; -\frac{q+1}{pq-1}
\end{pmatrix},
\end{equation}
\begin{equation}\label{def:Vys}
V(y,s) = \begin{pmatrix} 
0 & p\big(\psi^{p-1} - \gamma^{p-1}\big)\\ q\big(\varphi^{q-1} - \Gamma^{q-1}\big) &0
\end{pmatrix} \equiv \begin{pmatrix}
0 & V_1\\ V_2 & 0
\end{pmatrix},
\end{equation}
\begin{equation}\label{def:Bys}
\binom{F_1(\Upsilon, y,s)}{F_2(\Lambda, y,s)} = \binom{|\Upsilon + \psi|^{p-1}(\Upsilon + \psi) - \psi^p - p\psi^{p-1}\Upsilon}{|\Lambda + \varphi|^{q-1}(\Lambda + \varphi) - \varphi^q - q\varphi^{q-1}\Lambda},
\end{equation}
and 
\begin{equation}\label{def:Rys}
\binom{R_1(y,s)}{R_2(y,s)} = \binom{-\partial_s \varphi + \Delta \varphi - \frac{1}{2}y\cdot \nabla \varphi - \left(\frac{p+1}{pq-1}\right)\varphi + \psi^p}{-\partial_s \psi + \mu\Delta \psi - \frac{1}{2}y\cdot \nabla \psi - \left(\frac{q+1}{pq-1}\right)\psi + \varphi^q}.
\end{equation}
Note that the term $\binom{F_1}{F_2}$ is built to be quadratic in the inner region $|y| \leq 2K\sqrt{s}$. Indeed, we have for all $K > 1$ and $s \geq 1$, 
\begin{equation*}
\sup_{|y|\leq 2K\sqrt{s}}\left|F_1(\Upsilon, y,s)\right| \leq C(K)|\Upsilon|^2, \quad \sup_{|y|\leq 2K\sqrt{s}}\left|F_2(\Lambda, y,s)\right| \leq C(K)|\Lambda|^2.
\end{equation*}
Note also that the term $\binom{R_1}{R_2}$ measures the defect preventing $(\varphi, \psi)$ from being an exact solution of \eqref{eq:PhiPsi}. Since $(\varphi, \psi)$ is an approximate solution of \eqref{eq:PhiPsi}, one easily checks that
\begin{equation}\label{eq:boundR1R2}
\|R_1(s)\|_{L^\infty(\RN)}+\|R_2(s)\|_{L^\infty(\RN)}  \leq \frac{C}{s}.
\end{equation}
Therefore, since we would like to make $(\Lambda, \Upsilon)$ go to zero as $s \to +\infty$ in $L^\infty(\RN)\times L^\infty(\RN)$, the dynamics of \eqref{eq:LamUp} are influenced by the asymptotic limit of its linear term, 
$$\Big(\Hc + \Mc + V(y,s)\Big)\binom{\Lambda}{\Upsilon} \quad \text{as}\quad s \to +\infty.$$

From the definition \eqref{def:Vys}, we see that the potential $V(y,s)$ has two fundamental properties that will influence strongly our analysis:\\

\noindent $(i)\;$ We have $(V_1(\cdot, s), V_2(\cdot, s)) \to (0,0)$ in $L^2_{\rho_1}(\RN) \times L^2_{\rho_\mu}(\RN)$ as $s \to +\infty$, where $L^2_{\rho_\eta}(\RN)$ is the weighted $L^2$ space associated with the weight $\rho_\eta$ defined by 
\begin{equation}\label{def:rhotheta}
\rho_\eta(y) = \frac{1}{(4\pi \eta)^{N/2}}e^{-\frac{|y|^2}{4\eta}}.
\end{equation}
In particular, the effect on $V$ inside the \textit{blowup region} or in the \textit{inner region} $|y| \leq K\sqrt{s}$ will be a perturbation of the effect of $\Hc + \Mc$. \\

\noindent $(ii)\;$ Outside the \textit{blowup region} or in the \textit{outer region} $|y|\geq K \sqrt s$, we have the following property: for all $\epsilon > 0$, there exist $K_{\epsilon} > 0$ and $s_\epsilon > 0$ such that 
\begin{equation*}
\sup_{s \geq s_\epsilon, |y| \geq K_\epsilon \sqrt{s}} \left|V_1(y,s) - (-p\gamma^{p-1}) \right| + \left|V_2(y,s) - (-q\Gamma^{q-1}) \right| \leq \epsilon.
\end{equation*}
In other words, outside the \textit{blowup region}, the linear operator $\Hc + \Mc + V$ behaves as 
$$\Hc + \left(\begin{array}{cc} -\frac{p+1}{pq-1} & \pm \epsilon_1\\ \pm \epsilon_2 & - \frac{q+1}{pq-1}
\end{array}\right).$$
Given that the spectrum of $\Hc$ is negative (see \eqref{eq:spectrumH} below) and that the matrix has negative eigenvalues for $\epsilon_1$ and $\epsilon_2$ small, we see that $\Hc + \Mc + V$ behaves as one with fully negative spectrum, which greatly simplifies the analysis in that region. \\

Since the behavior of the potential $V$ inside and outside the \textit{blowup region} is different, we will consider the dynamics for $|y| \geq K\sqrt{s}$ and $|y| \leq 2K\sqrt s$ separately for some $K$ to be fixed large.

Let us consider a non-increasing cut-off function $\chi_0 \in \mathcal{C}^\infty_0([0, +\infty))$, with $\text{supp}(\chi_0) \subset [0,2]$ and $\chi_0 \equiv 1$ on $[0,1]$, and introduce
\begin{equation}\label{def:chi}
\chi(y,s) = \chi_0\left(\frac{|y|}{K\sqrt{s}}\right),
\end{equation}
where $K$ is chosen large enough so that various technical estimates hold. We define
\begin{equation}\label{def:LeUe}
\binom{\Lambda_e}{\Upsilon_e} = (1 - \chi)\binom{\Lambda}{\Upsilon},
\end{equation}
$\binom{\Lambda_e}{\Upsilon_e}$ is the part of $\binom{\Lambda}{\Upsilon}$ for $|y| \geq K\sqrt{s}$. As announced a few lines above and as we will see in  Section \ref{sec:outerpart}, the spectrum of the linear operator of the equation
satisfied by $\binom{\Lambda_e}{\Upsilon_e}$ is negative, which makes the control of $\|\Lambda_e(s)\|_{L^\infty(\Rb)}$ and $\|\Upsilon_e(s)\|_{L^\infty(\Rb)}$ easily.

While the control of the outer part is easy, it is not the case for the part of $\binom{\Lambda}{\Upsilon}$ for $|y| \leq 2K\sqrt s$. In fact, inside the \textit{blowup region} $|y| \leq 2K\sqrt s$, the potential $V$ can be seen as a perturbation of the effect of $\Hc + \Mc$ whose spectrum has two positive eigenvalues, a zero eigenvalue in addition to infinitely negatives ones (see Lemma \ref{lemm:diagonal} below).  Therefore, we have to expand $\binom{\Lambda}{\Upsilon}$ inside the \textit{blowup region} with respect to these eigenvalues in order to control $\|\Lambda(s)\|_{L^\infty(|y| \leq 2K\sqrt s)}$ and $\|\Upsilon(s)\|_{L^\infty(|y| \leq 2K\sqrt s)}$. To do so, we need to find a basis where $\Hc + \Mc$ is diagonal or at least in Jordan blocks' form. Since the operator $\Hc$ is contributed from $\Lc_1$ and $\Lc_\mu$, let us first recall well-known spectral properties of the operator $\Lc_\eta$, where $\eta \in \{1, \mu\}$.\\

\noindent $\bullet$ \textbf{Spectral properties of $\Lc_\eta$:} Given $\eta > 0$, let us consider the Hilbert space $L^2_{\rho_\eta}(\RN, \Rb)$ which is the set of all $f \in L^2_{loc}(\RN, \Rb)$ such that 
$$ \|f\|^2_{\rho_\eta} = \big<f,f\big>_{\rho_\eta} < +\infty,$$
where 
\begin{equation}\label{def:normL2rho}
\big<f,g\big>_{\rho_\eta} =  \int_{\RN} f(y)g(y)\rho_\eta(y)dy,
\end{equation}
and $\rho_\eta$ is defined by \eqref{def:rhotheta}. Note that we can write $\Lc_\eta$ in the divergence form
$$\Lc_\eta v =  \frac{\eta}{\rho_\eta}\;\text{div}\,\Big(\rho_\eta\nabla v\Big),$$
and that $\Lc_\eta$ is self-adjoint with respect to the weight $\rho_\eta$. Indeed, for any $v$ and $w$ in $L^2_{\rho_\eta}(\RN, \Rb)$, it holds that
\begin{equation}\label{eq:Ladjoint}
\int_{\RN}v\Lc_\eta w \rho_\eta dy = \int_{\RN}w \Lc_\eta v \rho_\eta dy.
\end{equation}

Let us introduce for each $\alpha = (\alpha_1, \cdots, \alpha_N)\in \mathbb{N}^N$ the polynomial
\begin{equation*}
\tilde{h}_\alpha(y) = c_\alpha\prod_{i = 1}^N H_{\alpha_i}\left(\frac{y_i}{2\sqrt{\eta}}\right),
\end{equation*}
where $H_n$ is the standard one dimensional Hermite polynomial, i.e
\begin{equation}\label{def:Hermite}
H_n(x) = (-1)^ne^{x^2}\frac{d^n}{dx^n}(e^{-x^2}),
\end{equation}
and $c_\alpha \in \Rb$ is chosen so that the term of highest degree in $\tilde{h}_\alpha$ is $\prod_{i = 1}^Ny_i^{\alpha_i}$. In one-dimensional case, we have
\begin{equation}\label{eq:hntildeN1}
\tilde{h}_n(y) = \eta^\frac{n}{2}\sum_{j = 0}^{\left[\frac{n}{2}\right]}\frac{n!}{(n - 2j)! j!}(-1)^j \left(\frac{y}{\sqrt \eta}\right)^{n - 2j}.
\end{equation}
For example,
$$\tilde h_0 = 1, \quad \tilde h_1 = y, \quad \tilde{h}_2 = y^2 - 2\eta,$$
$$\tilde{h}_3 = y^3 - 6\eta y, \quad \tilde{h}_4 = y^4 - 12\eta y^2 + 12\eta^2.$$

The family of eigenfunctions of $\Lc_\eta$ constitutes an orthogonal basic in $L^2_{\rho_\eta}(\RN, \Rb)$ in the sense that for any different $\alpha$ and $\beta$ in $\mathbb{N}^N$, 
\begin{equation*}
\Lc_\eta \tilde{h}_\alpha = -\frac{|\alpha|}{2}\tilde{h}_\alpha, \quad |\alpha| = \alpha_1+\cdots + \alpha_N,
\end{equation*}
\begin{equation}\label{eq:orthohnhm}
\int_{\RN} \tilde{h}_\alpha(y) \tilde{h}_\beta(y)\rho_\eta(y)dy = 0,
\end{equation}
and that for any $f$ in $L^2_{\rho_\eta}(\RN, \Rb)$, one can express
$$f = \sum_{\alpha \in \mathbb{N}^N}\big<f, \tilde{h}_\alpha\big>_{\rho_\eta}\tilde{h}_\alpha = \sum_{\alpha \in \mathbb{N}^N} f_\alpha \tilde{h}_\alpha.$$

\begin{remark}  We remark that for any polynomial $P_n(y)$ of degree $n$, we have by \eqref{eq:orthohnhm},
$$\int_{\RN}\tilde{h}_\alpha(y) P_n(y)\rho_\eta(y)dy = 0 \quad \text{for all}\;\; |\alpha| > n.$$
\end{remark}

\medskip

\noindent $\bullet$ \textbf{Spectral properties of $\Hc$}: Let us consider the functional space $L^2_{\rho_1}(\RN, \Rb) \times L^2_{\rho_\mu}(\RN, \Rb)$, which is the set of all $\binom{f}{g} \in L^2_{loc}(\RN, \Rb) \times L^2_{loc}(\RN,\Rb)$ such that 
$$\left<\binom{f}{g}, \binom{f}{g}\right> < +\infty,$$
where 
$$\left<\binom{f_1}{g_1}, \binom{f_2}{g_2}\right>:= \big<f_1, f_2 \big>_{\rho_1} + \big<g_1, g_2 \big>_{\rho_\mu}.$$
If we introduce for each $\alpha \in \mathbb{N}^N$,
\begin{equation}\label{def:hkalpha}
h_\alpha(y) = a_\alpha\prod_{i=1}^NH_{\alpha_i}\left(\frac{y_i}{\sqrt 2}\right) \quad \text{and} \quad \hat{h}_\alpha(y) = \hat{a}_\alpha \prod_{i=1}^NH_{\alpha_i}\left(\frac{y_i}{2\sqrt{\mu}}\right),
\end{equation}
where $H_n$ is defined by \eqref{def:Hermite}, and $a_\alpha$ and $\hat a_\alpha$ are constants chosen so that the terms of highest degree in $h_\alpha$ and $\hat{h}_\alpha$ is $\prod_{i = 1}^N y^{\alpha_i}$, then 
\begin{equation}\label{eq:spectrumH}
\Hc \binom{h_\alpha}{0} = -\frac{|\alpha|}{2}\binom{h_\alpha}{0} \quad \text{and} \quad \Hc \binom{0}{\hat h_\alpha} = -\frac{|\alpha|}{2}\binom{0}{\hat h_\alpha}.
\end{equation}
Moreover, for each $\binom{f}{g}$ in $L^2_{\rho_1}(\RN, \Rb) \times L^2_{\rho_\mu}(\RN, \Rb)$, we can write it in the form
\begin{align*}
\binom{f}{g}= \sum_{\alpha \in \mathbb{N}^N}\big<f, h_\alpha\big>_{\rho_1}\binom{h_\alpha}{0} + \big<g,\hat{h}_\alpha \big>_{\rho_\mu}\binom{0}{\hat{h}_\alpha}.
\end{align*}

\medskip

\noindent $\bullet$ \textbf{Spectral properties of $\Hc + \Mc$}: As announced in the beginning of Section \ref{sec:formu}, we switch back to the case $N = 1$ for simplicity. Of course, our proof remains valid in the case $N \geq 2$, though with some complications in the notation. We want to find a basis where $\Hc + \Mc$ is diagonal or at least in Jordan blocks' form. More precisely, we have the following:

\begin{lemma}[Diagonalization of $\Hc + \Mc$ in the one dimensional case] \label{lemm:diagonal}  For all $n \in \mathbb{N}$, there exist polynomials $f_n, g_n, \tilde{f}_n$ and $\tilde{g}_n$ of degree $n$ such that

\begin{equation}\label{eq:HMspec1}
\Big(\Hc+ \Mc\Big)\binom{f_n}{g_n} = \left(1 - \frac{n}{2}\right)\binom{f_n}{g_n}, 
\end{equation}
and 
\begin{equation}\label{eq:HMspectilde}
 \Big(\Hc+ \Mc\Big)\binom{\tilde{f}_n}{\tilde{g}_n} = -\left(\frac{n}{2} + \frac{(p+1)(q+1)}{pq-1}\right)\binom{\tilde{f}_n}{\tilde{g}_n},
\end{equation}
where
\begin{equation}\label{def:fngn}
\binom{f_n}{g_n} =  \sum_{j = 0}^{\left[\frac{n}{2}\right]}d_{n, n- 2j} \binom{h_{n - 2j}}{0} + e_{n, n - 2j}\binom{0}{\hat{h}_{n - 2j}},
\end{equation}
\begin{equation}\label{def:fngntilde}
\binom{\tilde{f}_n}{\tilde{g}_n} = \sum_{j = 0}^{\left[\frac{n}{2}\right]}\tilde d_{n, n - 2j} \binom{h_{n - 2j}}{0} + \tilde e_{n, n - 2j}\binom{0}{\hat{h}_{n - 2j}},
\end{equation}
and the coefficients $d_{n, n-2j}, e_{n, n-2j}$, $\tilde{d}_{n, n-2j}$, $\tilde{e}_{n, n-2j}$ depend on the parameters $p, q$ and $\mu$. In particular, we have

\begin{equation}\label{eq:dnen2}\begin{array}{ll}
d_{n,n} = (p+1)\Gamma,\quad &d_{n, n - 2} = n(n-1)p \Gamma (1 - \mu),\\
\\
e_{n,n} = (q+1)\gamma,\quad &e_{n, n - 2} = n(n-1)q\gamma (\mu -1),
\end{array}
\end{equation}
and 
\begin{equation}\label{eq:dnen2tilde}
\begin{array}{ll}
\tilde{d}_{n,n} = p\Gamma, \quad & \tilde{d}_{n,n-2} = n(n-1)\dfrac{pq\Gamma (p+1)(1-\mu)}{3pq + p + q - 1},\\
\\
\tilde{e}_{n,n} = -q\gamma, \quad & \tilde{e}_{n,n-2} = n(n-1)\dfrac{pq\gamma (q+1)(1-\mu)}{3pq + p + q - 1}.
\end{array}
\end{equation}

\end{lemma}
\begin{remark}\label{rema:012} The spectrum of $\Hc + \Mc$ has two positive eigenvalues $\lambda_0 = 1$ and $\lambda_1 = \frac{1}{2}$ corresponding to eigenvectors $\binom{f_0}{g_0}$ and $\binom{f_1}{g_1}$; a zero eigenvalue $\lambda_2 = 0$ corresponding to eigenvector $\binom{f_2}{g_2}$. Note that in the case when $N \geq 2$, we have 
$$\binom{f_0(y)}{g_0(y)} = \binom{(p+1)\Gamma}{(q + 1)\gamma}, \quad \binom{f_1(y)}{g_1(y)} = \binom{f_{1,i}(y)}{g_{1,i}(y)}_{1 \leq i \leq N
},$$
and 
$$\binom{f_2(y)}{g_2(y)} = \binom{f_{2,ij}(y)}{g_{2,ij}(y)}_{1 \leq i,j \leq N},$$
where
$$ \binom{f_{1,i}(y)}{g_{1,i}(y)} = \binom{(p+1)\Gamma y_i}{(q+1)\gamma y_i} \quad \text{for}\; 1\leq i \leq N,$$
and 
$$\binom{f_{2,ij}(y)}{g_{2,ij}(y)} = \binom{f_{2,ji}(y)}{g_{2,ji}(y)} = \binom{(p+1)\Gamma y_i y_j}{(q+1)\gamma y_iy_j} \quad \text{for}\; 1 \leq i \neq j \leq N,$$
and 
\begin{equation}\label{def:f2g2Ndim}
\binom{f_{2,ii}(y)}{g_{2,ii}(y)} = \binom{(p+1)\Gamma y_i^2}{(q+1)\gamma y_i^2} + \binom{2p\Gamma(1 - \mu)}{2q\gamma(\mu - 1)} \quad \text{for}\; 1 \leq i \leq N.
\end{equation}
\end{remark}

\begin{proof} For each $n \in \mathbb{N}$, we want to find $\binom{F_n}{G_n}$ in the form of polynomials of degree $n$ such that 
\begin{equation}\label{eq:conHM}
\Big(\Hc + \Mc\Big)\binom{F_n}{G_n} = \lambda\binom{F_n}{G_n} \quad \text{for some }\; \lambda \in \Rb.
\end{equation}
Let us assume that 
$$\binom{F_n}{G_n} = \sum_{i=0}^n\binom{a_{n,n - i}}{b_{n, n-i}} y^{n-i}, \quad a_{n,n} \neq 0, \; b_{n,n} \neq 0.$$
Plugging this form into \eqref{eq:conHM} and comparing elements of the order $y^{n-i}$ with $i \in \{0, 1, \cdots, n\}$, we have for $i = 0$,
\begin{equation}\label{eq:cdnn}
\Big(\Mc - (\lambda + \frac{n}{2})\, \textbf{Id}\Big)\binom{a_{n,n}}{b_{n,n}} = 0, 
\end{equation}
and for $i = 1$,
\begin{equation}\label{eq:cdn1}
\Big(\Mc - (\lambda + \frac{n- 1}{2})\, \textbf{Id}\Big)\binom{a_{n,n-1}}{b_{n,n-1}} = 0,
\end{equation}
and for $i = 2, 3, \cdots, n$, 
\begin{equation}\label{eq:cdni}
\Big(\Mc - (\lambda + \frac{n - i}{2})\, \textbf{Id}\Big)\binom{a_{n,n - i}}{b_{n,n- i}} + (n - i + 2)(n - i + 1)\binom{a_{n, n - i + 2}}{\mu \,b_{n, n - i +2}} = 0.
\end{equation}
Since $(a_{n,n}, b_{n,n}) \neq (0,0)$, we deduce from \eqref{eq:cdnn} that $\det\Big(\Mc - (\lambda + \frac{n}{2})\, \textbf{Id}\Big) = 0$, which means that $\lambda$ satisfy
$$\left(\lambda + \frac{n}{2}\right)^2 + \left(\lambda + \frac{n}{2}\right) \left(\frac{p + q + 2}{pq - 1}\right) - \frac{(p+1)(q + 1)}{pq-1} = 0.$$
Hence, either
\begin{equation}\label{eq:lambdaEig}
\lambda = \lambda_+ = 1 - \frac{n}{2} \quad \text{or} \quad \lambda = \lambda_-= -\frac{(p+1)(q + 1)}{pq-1} - \frac{n}{2}.
\end{equation}
Substituting these values of $\lambda$ into \eqref{eq:cdnn} yields
$$\binom{a_{n,n}}{b_{n,n}} = \binom{(p+1)\Gamma}{(q+1)\gamma}\;\; \text{if} \; \lambda = \lambda_+,\quad \text{and} \quad \binom{a_{n,n}}{b_{n,n}} = \binom{p\Gamma}{-q\gamma} \;\; \text{if} \; \lambda = \lambda_-.$$
Note that for these values of $\lambda$ given in \eqref{eq:lambdaEig}, we have by a direct computation for $i = 1, 2, \cdots, n$,
$$\text{det}\,\Big(\Mc - (\lambda + \frac{n- i}{2})\, \textbf{Id}\Big) = \frac{i^2}{4} - \frac{i}{2}\left(\frac{p + q +2}{pq - 1} + 2\lambda + n\right) \neq 0,$$
whence, we obtain from equation \eqref{eq:cdn1},
$$\binom{a_{n, n-1}}{b_{n, n-1}} = \binom{0}{0}.$$
and from equation \eqref{eq:cdni} by induction,
$$\binom{a_{n, n-3}}{b_{n, n-3}} = \binom{a_{n, n-5}}{b_{n, n-5}}= \cdots = \binom{0}{0}.$$
The couples $\binom{a_{n, n - 2}}{b_{n,n-2}}$, $\binom{a_{n, n - 4}}{b_{n,n-4}}$, $\cdots$ are respectively determined from $\binom{a_{n, n}}{b_{n,n}}$, $\binom{a_{n, n - 2}}{b_{n,n-2}}$, $\cdots$ through equation \eqref{eq:cdni}.

Since the terms of highest degree of $h_m$ and $\hat{h}_m$ are $y^m$, and $h_m$ and $\hat{h}_m$ are even (or odd respectively) if $m$ is even integer (or odd), we can rewrite the expression of $\binom{F_n}{G_n}$ in terms of $\binom{h_j}{0}$ and $\binom{0}{\hat{h}_j}$ for $j = 0, 1, \cdots, n$ as stated in \eqref{def:fngn} and \eqref{def:fngntilde}. 

In order to precise the values of $\binom{d_{n, n - 2}}{e_{n, n - 2}}$ and $\binom{\tilde d_{n, n - 2}}{\tilde e_{n, n - 2}}$, let us compute $\binom{a_{n, n - 2}}{b_{n,n-2}}$.

- For $\lambda = \lambda_+$, we use \eqref{eq:cdni} with $i = 2$ to get
\begin{equation*}
\binom{a_{n, n-2}}{b_{n, n-2}} = -n(n-1)\Mc^{-1}\binom{(p+1)\Gamma}{\mu(q + 1)\gamma} = -n(n-1)\binom{\Gamma(p\mu + 1) }{\gamma(q +\mu)}.
\end{equation*}
Recalling from the definition  \eqref{def:hkalpha} that  $h_{n}(y) = y^n - n(n-1)y^{n-2} + \cdots,$ and $\hat h_n(y) = y^n - n(n-1)\mu y^{n-2} + \cdots$, we deduce from \eqref{def:fngn} that

$$\binom{d_{n,n-2}}{e_{n,n-2}} = n(n-1)\binom{d_{n,n}}{\mu\, e_{n,n}} + \binom{a_{n, n-2}}{b_{n, n-2}} = n(n-1)\binom{p\Gamma(1 - \mu)}{q\gamma(\mu - 1)},$$
which is \eqref{eq:dnen2}.

- For $\lambda = \lambda_-$, we similarly have
\begin{align*}
\binom{\tilde a_{n, n-2}}{\tilde b_{n, n-2}} &= -n(n - 1)\left[\Mc + \left(\frac{2pq + p + q}{pq-1}\right)\textbf{Id} \right]^{-1}\binom{p\Gamma}{-\mu q \gamma}\nonumber\\
&  = -\frac{n(n-1)}{3pq + p + q - 1}\binom{p\Gamma\big[pq(1 + \mu) + (p + \mu q) + (pq - 1)\big]}{-q\gamma\big[pq(1 + \mu) + (p + \mu q) + \mu(pq-1)\big]}.
\end{align*}
Noticing from \eqref{def:fngntilde} that 
$$\binom{\tilde d_{n,n-2}}{\tilde e_{n,n-2}} = n(n-1)\binom{\tilde{d}_{n,n}}{\mu\, \tilde e_{n,n}} + \binom{\tilde a_{n, n-2}}{\tilde b_{n, n-2}},$$
hence, \eqref{eq:dnen2tilde} follows after a straightforward calculation. This concludes the proof of Lemma \ref{lemm:diagonal}.
\end{proof}

\bigskip

For the sake of controlling $\binom{\Lambda}{\Upsilon}$ in the region $|y| \leq 2K\sqrt s$, we will expand  $\binom{\Lambda}{\Upsilon}$ with respect to the family $\left\{\binom{h_n}{0}, \binom{0}{\hat{h}_n}\right\}_{n \geq 0}$ and then with respect to the family $\left\{\binom{f_n}{g_n}, \binom{\tilde{f}_n}{\tilde{g}_n}\right\}_{n \geq 0}$. We start by writing
\begin{align}\label{eq:expandLAmbUpb}
\binom{\Lambda(y,s)}{\Upsilon(y,s)} &= \sum_{n \leq M} Q_n(s) \binom{h_n(y)}{0} + \hat{Q}_n(s) \binom{0}{\hat{h}_n(y)} + \binom{\Lambda_-(y,s)}{\Upsilon_-(y,s)},
\end{align}
(note that this identity is precisely the definition of $\binom{\Lambda_-}{\Upsilon_-}$) where $M$ is a fixed even integer satisfying 
\begin{equation}\label{eq:Mfixed}
M \geq 4\left[1 + \|\Mc\|_{\infty} + 2\max_{y \in \Rb, s \geq 1, i=1,2}|V_i(y,s)|\big)
 \right],
\end{equation}
with $\|\Mc\|_\infty = \max\left\{q\Gamma^{q-1} + \frac{p+1}{pq-1}, p\gamma^{p-1} + \frac{q+1}{pq-1}\right\}$ (in view of the definition \eqref{def:matM} of $\Mc$, this is indeed a suitable norm for $(2\times 2)$ matrices). As we will show in  Section \ref{sec:negetivepart}, the choice of $M$ is crucial and allows us to successfully use a Gronwall's inequality in the control of the infinite-dimensional part $\binom{\Lambda_-}{\Upsilon_-}$, and

\noindent $\bullet\;$ $Q_n(s)$ and $\hat Q_n(s)$ are the projections of $\binom{\Lambda}{\Upsilon}$ on $\binom{h_n}{0}$ and $\binom{0}{\hat h_n}$ respectively, defined by
\begin{equation}\label{def:Qn}
Q_n(s) = \frac{\left< \binom{\Lambda}{\Upsilon}, \binom{h_n}{0} \right>}{\left< \binom{h_n}{0}, \binom{h_n}{0} \right>} = \frac{\big<\Lambda,h_n \big>_{\rho_1}}{\big< h_n, h_n\big>^2_{\rho_1}}\equiv \Pi_n\binom{\Lambda}{\Upsilon},
\end{equation}
\begin{equation}\label{def:Qnhat}
\hat Q_n(s) = \frac{\left< \binom{\Lambda}{\Upsilon}, \binom{0}{\hat h_n} \right>}{\left< \binom{0}{\hat h_n}, \binom{0}{\hat h_n} \right>} = \frac{\big<\Upsilon,\hat{h}_n \big>_{\rho_\mu}}{\big< \hat{h}_n, \hat{h}_n\big>^2_{\rho_\mu}} \equiv \hat \Pi_n\binom{\Lambda}{\Upsilon},
\end{equation}

\noindent $\bullet\;$ $\binom{\Lambda_-(y,s)}{\Upsilon_-(y,s)} = \Pi_{-,M}\binom{\Lambda}{\Upsilon}$ is called the infinite-dimensional part of $\binom{\Lambda}{\Upsilon}$, where $\Pi_{-,M}$ is the projector on the subspace of $L_{\rho_1} \times L_{\rho_\mu}$ where the spectrum of $\Hc$ is lower than $\frac{1-M}{2}$. Note that for all $n \leq M$,
\begin{equation}\label{eq:partnegVW}
\left< \binom{\Lambda_-}{\Upsilon_-}, \binom{h_n}{0} \right>= \big<\Lambda_-, h_n\big>_{\rho_1} = 0 \;\; \text{and}\;\; \left< \binom{\Lambda_-}{\Upsilon_-}, \binom{0}{\hat h_n} \right> = \big<\Upsilon_-,\hat h_n\big>_{\rho_\mu} = 0.
\end{equation}

\noindent $\bullet\;$ We also introduce $\Pi_{+,M} = \textbf{Id} - \Pi_{-,M}$, and the complementary part 
$$\binom{\Lambda_+}{\Upsilon_+} = \Pi_{+,M}\binom{\Lambda}{\Upsilon} = \binom{\Lambda}{\Upsilon} - \binom{\Lambda_-}{\Upsilon_-}$$
which is called the finite-dimensional part of $\binom{\Lambda}{\Upsilon}$, and which satisfies for all $s$,
\begin{equation}
\left<\binom{\Lambda_+(y,s)}{\Upsilon_+(y,s)}, \binom{\Lambda_-(y,s)}{\Upsilon_-(y,s)}\right> = 0.
\end{equation}
We will expand it with respect to the basis of eigenfunctions of $\Hc + \Mc$ computed in Lemma \ref{lemm:diagonal}, namely the family $\left\{\binom{f_n}{g_n}, \binom{\tilde f_n}{\tilde g_n}\right\}_{n \leq M}$, as follows:
\begin{align}
\binom{\Lambda_+(y,s)}{\Upsilon_+(y,s)} &= \sum_{n \leq M}Q_n(s)\binom{h_n(y)}{0} + \hat Q_n(s)\binom{0}{\hat h_n(y)}\nonumber\\
&= \sum_{n \leq M}\theta_n(s)\binom{f_n(y)}{g_n(y)} + \tilde \theta_n(s)\binom{\tilde f_n(y)}{\tilde g_n(y)},\label{eq:LplusUplusexpand}
\end{align}
where $\theta_n(s)= P_{n,M}\binom{\Lambda}{\Upsilon}$ and $\tilde{\theta}_n(s) = \tilde P_{n,M}\binom{\Lambda}{\Upsilon}$ are projections of $\binom{\Lambda}{\Upsilon}$ on $\binom{f_n}{g_n}$ and $\binom{\tilde f_n}{\tilde g_n}$ respectively. This is possible, since from Lemma \ref{lemm:diagonal}, we can express $\theta_n(s)$ and $\tilde{\theta}_n(s)$ in terms of $Q_n(s)$ and $\hat Q_n(s)$ as follows:

\begin{lemma}[Definition of the projection on the modes $\binom{f_n}{g_n}$ and $\binom{\tilde f_n}{\tilde g_n}$]\label{lemm:DefProjection} We have 
\begin{equation}\label{def:Pn}
\theta_n =  \sum_{j = 0}^{\left[\frac{M - n}{2}\right]} A_{n + 2j, n}\, Q_{n + 2j} + B_{n + 2j, n}\, \hat Q_{n + 2j} \equiv P_{n,M}\binom{\Lambda}{\Upsilon},
\end{equation}
and 
\begin{equation}\label{def:Pntilde}
\tilde{\theta}_n =  \sum_{j = 0}^{\left[\frac{M - n}{2}\right]} \tilde A_{n + 2j, n}\, Q_{n + 2j} + \tilde B_{n + 2j,n}\, \hat Q_{n + 2j} \equiv \tilde P_{n,M}\binom{\Lambda}{\Upsilon}
\end{equation}
where the coefficients $A_{n + 2j, n}$, $B_{n + 2j,n}$, $\tilde A_{n + 2j, n}$ and $\tilde B_{n + 2j,n}$ for $j = 0, 1, 2, \cdots$ depend on $p, q$ and $\mu$. In particular, we have 
\begin{equation}\label{eq:AnBn}
A_{n,n} = \frac{q}{\Gamma(2pq+p+q)}, \quad B_{n,n} = \frac{p}{\gamma(2pq+p+q)},
\end{equation}
and
\begin{equation}\label{eq:Ann2Bnn2}
A_{n+2,n} = -\frac{\tilde e_{n+2,n}}{\Gamma \gamma(2pq+p+q)} \quad \text{and} \quad B_{n+2,n} = \left(\frac {p+1}{q+1}\right) \frac{ \tilde e_{n+2,n}}{\gamma^2(2pq+p+q)}.
\end{equation}
\end{lemma}

\begin{remark} \label{rema:pro} From Lemma \ref{lemm:DefProjection}, we obviously see that when a function is of the form $ \sum_{n = 0}^M \omega_n \binom{f_n}{g_n} + \tilde{\omega}_n \binom{\tilde f_n}{\tilde g_n}$, its projections on $\binom{f_n}{g_n}$ and $\binom{\tilde f_n}{\tilde g_n}$ are respectively $\omega_n$ and $\tilde \omega_n$.
\end{remark}
\begin{remark} The precise values given in \eqref{eq:AnBn} and \eqref{eq:Ann2Bnn2} are crucial in deriving a refined ODE satisfied by the null mode, that is the ODE given in part $(iii)$ of Proposition \ref{prop:dyn} (see Lemma \ref{lemm:Pro3rdtermf2g2} also).
\end{remark}

\begin{proof} We first note that the matrix of $\left\{\binom{f_n}{g_n}, \binom{\tilde{f}_n}{\tilde{g}_n}\right\}_{n \in \mathbb{N}}$ in the basis $\left\{\binom{h_n}{0}, \binom{0}{\hat{h}_n}\right\}_{n \in \mathbb{N}}$ is "lower triangular" in the sense that we can express the matrix in terms of $(2 \times 2)$ blocks (see \eqref{def:fngn} and \eqref{def:fngntilde}) as follows:

\begin{equation}\label{sy:1}
\begin{pmatrix}
X_0 \\ X_1 \\ X_2 \\ X_3 \\ X_4 \\ \vdots \\ X_{M}
\end{pmatrix}
 = \begin{pmatrix}
\Fc & \mathbf {0} & \cdots & \cdots & \cdots & \cdots & \mathbf {0}\\
\mathbf {0} & \Fc & \mathbf {0} & \cdots & \cdots & \cdots & \mathbf {0}\\
\Dc_{2,0} & \mathbf {0} & \Fc & \mathbf {0} & \cdots & \cdots & \mathbf {0}\\
\mathbf {0} & \Dc_{3,1} & \mathbf {0} & \Fc & \mathbf {0} & \cdots & \mathbf {0}\\
\Dc_{4,0} & \mathbf {0} & \Dc_{4,2} & \mathbf {0} & \Fc & \cdots & \mathbf {0}\\
\vdots & \ddots & \ddots & \ddots & \ddots & \ddots & \vdots\\
\cdots & \cdots & \Dc_{M, M - 4} &  \mathbf {0} & \Dc_{M, M - 2} & \mathbf {0} & \Fc
\end{pmatrix} \begin{pmatrix}
Y_0 \\ Y_1\\ Y_2\\ Y_3 \\ Y_4 \\ \vdots \\ Y_{M}
\end{pmatrix},
\end{equation}
where $\mathbf {0}$ is the $(2\times 2)$ zero matrix,
$$X_n = \begin{pmatrix}\binom{f_n}{g_n} \\ \\
\binom{\tilde f_n}{\tilde g_n} \end{pmatrix}, \quad Y_n = \begin{pmatrix}\binom{h_n}{0} \\ \\ \binom{0}{\hat{h}_n}\end{pmatrix},$$
$$\Fc = \begin{pmatrix}
(p+1)\Gamma &\quad (q+1)\gamma\\
p \Gamma &\quad -q \gamma
\end{pmatrix},$$

$$\Dc_{n, n-2j} = \begin{pmatrix}
d_{n, n - 2j}&\quad e_{n, n - 2j}\\
\tilde d_{n, n-2j}&\quad \tilde e_{n, n-2j}
\end{pmatrix}.$$

Thus, we can express $\left\{\binom{h_n}{0}, \binom{0}{\hat{h}_n}\right\}_{n \in \mathbb{N}}$ in terms of $\left\{\binom{f_n}{g_n}, \binom{\tilde{f}_n}{\tilde{g}_n}\right\}_{n \in \mathbb{N}}$ by inverting the matrix associated to system \eqref{sy:1} resulting in the following:

\begin{equation}\label{sy:2}
\begin{pmatrix}
Y_0 \\ Y_1 \\ Y_2 \\ Y_3 \\ Y_4 \\ \vdots \\ Y_{M}
\end{pmatrix}
  = \begin{pmatrix}
\Tc & \mathbf {0} & \cdots & \cdots & \cdots & \cdots & \mathbf {0}\\
\mathbf {0} & \Tc & \mathbf {0} & \cdots & \cdots & \cdots & \mathbf {0}\\
\Gc_{2,0} & \mathbf {0} & \Tc & \mathbf {0} & \cdots & \cdots & \mathbf {0}\\
\mathbf {0} & \Gc_{3,1} & \mathbf {0} & \Tc & \mathbf {0} & \cdots & \mathbf {0}\\
\Gc_{4,0} & \mathbf {0} & \Gc_{4,2} & \mathbf {0} & \Tc & \cdots & \mathbf {0}\\
\vdots & \ddots & \ddots & \ddots & \ddots & \ddots & \vdots\\
\cdots & \cdots & \Gc_{M, M - 4} &  \mathbf {0} & \Gc_{M, M - 2} & \mathbf {0} & \Tc
\end{pmatrix}  \begin{pmatrix}
X_0 \\ X_1\\ X_2\\ X_3 \\ X_4 \\ \vdots \\ X_{M}
\end{pmatrix},
\end{equation}
where 
$$\Tc = \Fc^{-1} = \frac{1}{\Gamma \gamma (2pq + p + q)} \begin{pmatrix}
q \gamma &\quad (q + 1)\gamma \\p\Gamma &\quad - (p+1)\Gamma
\end{pmatrix},$$
for some $(2\times2)$ matrices $\Gc_{i,j}$.

By extracting the $(2\times 2)$ blocks in \eqref{sy:2}, we derive from \eqref{eq:expandLAmbUpb} the expressions \eqref{def:Pn} and \eqref{def:Pntilde}. It remains to compute \eqref{eq:AnBn} and \eqref{eq:Ann2Bnn2} in order to complete the proof of Lemma \ref{lemm:DefProjection}. To this end, we note from \eqref{sy:2} that 
$$\begin{pmatrix}
A_{n,n} & \tilde{A}_{n,n}\\
B_{n,n} & \tilde{B}_{n,n}
\end{pmatrix} = \Tc, \quad \text{and} \quad \begin{pmatrix}
A_{n+2,n} & \tilde{A}_{n+2,n}\\
B_{n+2,n} & \tilde{B}_{n+2,n}
\end{pmatrix} = \Gc_{n+2,n} = - \Tc \Dc_{n+2,n} \Tc.$$

This gives \eqref{eq:AnBn} and the following formulas for $A_{n+2,n}$ and $B_{n+2,2}$:
\begin{align*}
A_{n+2,n} &= -\frac{1}{\Gamma^2 \gamma^2(2pq+p+q)^2}\big[q\gamma^2 (q d_{n+2,n} + (q+1)\tilde d_{n+2,n}) + \\& \qquad \qquad \qquad \qquad \qquad \qquad + p\Gamma \gamma(q e_{n+2,n} + (q+1)\tilde e_{n+2,n}) \big],\\
B_{n+2,n} &= -\frac{1}{\Gamma^2 \gamma^2(2pq+p+q)^2}\big[ q\Gamma \gamma(p d_{n+2,n} - (p+1)\tilde d_{n+2,n}) + \\
&\qquad \qquad \qquad \qquad \qquad \qquad + p\Gamma^2 (p e_{n+2,n} - (p+1)\tilde{e}_{n+2,n}) \big],
\end{align*}
where the coefficients $d_{n+2,n}$, $e_{n+2,n}$, $\tilde{d}_{n+2,n}$ and $\tilde e_{n+2,n}$ are given in \eqref{eq:dnen2} and \eqref{eq:dnen2tilde}.\\
Note that
$$d_{n+2,n} = -\frac{p\Gamma}{q\gamma}e_{n+2,n} \quad \text{and}\quad \tilde{d}_{n+2,n} = \frac{(p+1)\Gamma}{(q+1)\gamma}\tilde{e}_{n+2,n},$$
we then simplify the expressions of $A_{n+2,n}$ and $B_{n+2,n}$ resulting in \eqref{eq:Ann2Bnn2}. This concludes the proof of Lemma \ref{lemm:DefProjection}.
\end{proof}

\bigskip

From \eqref{eq:expandLAmbUpb} and \eqref{eq:LplusUplusexpand}, it holds that
\begin{equation}\label{decomoposeLamUp}
\binom{\Lambda(y,s)}{\Upsilon(y,s)} = \sum_{n \leq M}\theta_n(s)\binom{f_n(y)}{g_n(y)} + \tilde{\theta}_n(s)\binom{\tilde{f}_n(y)}{\tilde{g}_n(y)} + \binom{\Lambda_-(y,s)}{\Upsilon_-(y,s)}.
\end{equation}
Note that the decomposition \eqref{decomoposeLamUp} is unique.

\section{Proof of the existence result assuming some technical lemmas.}\label{sec:existence}
This section is devoted to the proof of Theorem \ref{theo1}. We will first show the existence of a solution $\binom{\Lambda}{\Upsilon}$ of system \eqref{eq:LamUp} satisfying 
\begin{equation}\label{eq:goalconstruct}
\|\Lambda(s)\|_{L^\infty(\Rb)} + \|\Upsilon(s)\|_{L^\infty(\Rb)} \to 0 \quad \text{as} \quad s \to+\infty,
\end{equation}
which concludes part $(ii)$ of Theorem \ref{theo1} (though with no estimate of the error). The proof of parts $(i)$ and $(iii)$ then follows from part $(ii)$. We will give all the arguments of the proof without technical details which are left to the next section. 
Hereafter, we denote by $C$ a generic positive constant depending only on $p,q, \mu$ and $K$ introduced in \eqref{def:chi}. \\

Given $A \geq 1$ and $s_0 \geq e$, we consider initial data for system \eqref{eq:LamUp}, depending on two real parameters $d_0$ and $d_1$ of the form:
\begin{equation}\label{eq:intialdata}
\binom{\Lambda_0(y)}{\Upsilon_0(y)}_{d_0,d_1, s_0, A} = \frac{A}{s_0^2}\left(d_0\binom{f_0(y)}{g_0(y)} + d_1\binom{f_1(y)}{g_1(y)}\right)\chi(2y, s_0),
\end{equation}
where $\binom{f_i}{g_i}, i = 0,1$ are defined by \eqref{def:fngn} and $\chi$ is introduced in \eqref{def:chi}.  The solution of system \eqref{eq:LamUp} with initial data \eqref{eq:intialdata} will be denoted by 
$\binom{\Lambda(y,s)}{\Upsilon(y,s)}_{d_0, d_1, s_0, A}$, or by $\binom{\Lambda(y,s)}{\Upsilon(y,s)}$
when there is no ambiguity. Our aim is to show that if $A$ is fixed large enough, then $s_0$ is fixed large enough depending on $A$, we can also fix the parameters $(d_0,d_1) \in [-2,2]^2$ so that the solution $\binom{\Lambda(y,s)}{\Upsilon(y,s)}_{d_0, d_1, s_0, A}$ will be defined for all $s \geq s_0$ and converges to $\binom{0}{0}$  as $s \to +\infty$ in $L^\infty(\Rb)$, meaning that \eqref{eq:goalconstruct} holds. According to the decomposition \eqref{decomoposeLamUp} and the definition \eqref{def:LeUe}, it is enough to control the solution in a shrinking set defined as follows:

\begin{definition}[Definition of a shrinking set for the components of $\binom{\Lambda}{\Upsilon}$] \label{def:VA} For all $A \geq 1$ and $s \geq e$, we defined $\Vc_A(s)$ as the set of all $\binom{\Lambda}{\Upsilon} \in L^\infty(\Rb) \times L^\infty(\Rb)$ such that 
$$\|\Lambda_e(s)\|_{L^\infty(\Rb)} \leq \frac{A^{M+2}}{\sqrt{s}},\quad  \|\Upsilon_e(s)\|_{L^\infty(\Rb)} \leq \frac{A^{M+2}}{\sqrt{s}},$$
$$ \left\|\frac{\Lambda_-(y,s)}{1 + |y|^{M+1}} \right\|_{L^\infty(\mathbb{R})}\leq \frac{A^{M+1}}{s^{\frac{M+2}{2}}},\quad  \left\|\frac{\Upsilon_-(y,s)}{1 + |y|^{M+1}} \right\|_{L^\infty(\mathbb{R})} \leq \frac{A^{M+1}}{s^{\frac{M+2}{2}}},$$

$$|\theta_j(s)|\leq \frac{A^j}{s^\frac{j+1}{2}},\quad |\tilde{\theta}_j(s)| \leq \frac{A^j}{s^\frac{j+1}{2}} \;\; \text{for}\;\; 3\leq j\leq M,$$

$$ |\tilde \theta_i(s)| \leq \frac{A^{2}}{s^2}\;\; \text{for}\;\; i = 0, 1,2,$$

$$|\theta_2(s)| \leq \frac{A^4 \log s}{s^2},$$

$$|\theta_0(s)| \leq \frac{A}{s^2}, \quad |\theta_1(s)| \leq \frac{A}{s^2},$$
where $\Lambda_e, \Upsilon_e$ are defined by \eqref{def:LeUe},  $\Lambda_-, \Upsilon_-$, $\theta_n$, $\tilde \theta_n$ are defined as in \eqref{decomoposeLamUp}.
\end{definition}

As a mater of fact, one can check that if $\binom{\Lambda}{\Upsilon} \in \Vc_A(s)$ for $s \geq e$, then
\begin{equation}\label{eq:LUinVAest}
\|\Lambda(s)\|_{L^\infty(\Rb)} + \|\Upsilon(s)\|_{L^\infty(\Rb)} \leq \frac{CA^{M+2}}{\sqrt{s}},
\end{equation}
for some positive constant $C$ (see Proposition \ref{prop:properVA} below for the proof). Thus, if a solution $\binom{\Lambda}{\Upsilon}$ stays in $\Vc_A(s)$ for all $s \geq s_0$, then it converges to zero in $L^\infty(\Rb)\times L^\infty(\Rb)$. Our aim is then reduced to proving the following proposition:

\begin{proposition}[Existence of a solution of \eqref{eq:LamUp} trapped in $\Vc_A(s)$] \label{prop:goalVA} There exists $A_1$ such that for all $A \geq A_1$, there exists $s_{0,1}(A)$ such that for all $s_0 \geq s_{0,1}$, there exists $(d_0,d_1)$ such that if $\binom{\Lambda}{\Upsilon}$ is the solution of \eqref{eq:LamUp} with initial data at $s_0$ given by \eqref{eq:intialdata}, then $\binom{\Lambda(s)}{\Upsilon(s)} \in \Vc_A(s)$ for all $s \geq s_0$.
\end{proposition}

Let us first make sure that initial data \eqref{eq:intialdata} belongs to $\Vc_A(s_0)$.  In particular, we claim the following:

\begin{proposition}[Properties of initial data \eqref{eq:intialdata}] \label{prop:properinti} For each $A \geq 1$, there exists $s_{0,2}(A) \geq e$ such that for all $s_0 \geq s_{0,2}$, we have the following properties:

\noindent $(i)\;$ There exists a rectangle 
$$\Dc_{s_0} \subset [-2,2]^2$$
such that the mapping 
\begin{align*}
\Theta: \Rb^2 &\to \Rb^2\\
(d_0,d_1) &\mapsto (\theta_{0,0}, \theta_{0,1})
\end{align*}
(where $\theta_{0,i} = P_{i,M}\binom{\Lambda_0}{\Upsilon_0}$ for $i = 0, 1$, and $\binom{\Lambda_0}{\Upsilon_0}$ stands for $\binom{\Lambda_0}{\Upsilon_0}_{d_0,d_1,s_0,A}$) is linear, one to one from $\Dc_{s_0}$ onto $\left[-\frac{A}{s_0^2}, \frac{A}{s_0^2}\right]^2$ and maps $\partial \Dc_{s_0}$ into $\partial\left(\left[-\frac{A}{s_0^2}, \frac{A}{s_0^2}\right]^2\right)$. Moreover, it has degree one on the boundary. \\

\noindent $(ii)\;$ For all $(d_0,d_1) \in \Dc_{s_0}$, $\binom{\Lambda_0}{\Upsilon_0} \in \Vc_{A}(s_0)$ with strict inequalities except for $\theta_{0,0}$ and $\theta_{0,1}$, in the sense that

$$\Lambda_{0,e} = \Upsilon_{0,e} = 0,$$

$$ \left\|\frac{\Lambda_{0,-}(y)}{1 + |y|^{M+1}} \right\| < \frac{1}{s_0^{\frac{M+2}{2}}},\quad  \left\|\frac{\Upsilon_{0,-}(y)}{1 + |y|^{M+1}} \right\| < \frac{1}{s_0^{\frac{M+2}{2}}},$$

$$|\theta_{0,j}|< \frac{1}{s_0^\frac{j+1}{2}},\quad |\tilde{\theta}_{0,j}| < \frac{1}{s_0^\frac{j+1}{2}} \;\; \text{for}\;\; 3\leq j\leq M,$$

$$ |\tilde \theta_{0,i}| < \frac{1}{s_0^2}\;\; \text{for}\;\; i = 0, 1,2,$$

$$|\theta_{0,2}| < \frac{\log s_0}{s_0^2},$$

$$|\theta_{0,0}| \leq \frac{A}{s_0^2}, \quad |\theta_{0,1}| \leq \frac{A}{s_0^2}.$$
\end{proposition} 

\begin{remark} In some sense, $\binom{\Lambda_0}{\Upsilon_0}_{d_0,d_1, s_0, A}$ is reduced to the sum of its components on $\binom{f_0}{g_0}$ and $\binom{f_1}{g_1}$, the only eigenfunctions corresponding to the positive eigenvalues of the linear operator $\Hc + \Mc$ ($\lambda_0 = 1$ and $\lambda_1 = \frac{1}{2}$; see Lemma \ref{lemm:diagonal}). In $N$ dimensions, one has to take $d_0 \in \Rb$ and $d_1\in \RN$ because of the definition of $\binom{f_0}{g_0}$ and $\binom{f_1}{g_1}$ given in Remark \ref{rema:012}.
\end{remark}

The proof of Proposition \ref{prop:properinti} is postponed to Subsection \ref{sec:intialdata} (see Lemma \ref{lemm:properiniti}). Let us now give the proof of Proposition \ref{prop:goalVA}.
\begin{proof}[Proof of Proposition \ref{prop:goalVA}] Let us consider $A \geq 1$ and $s_0 \geq s_{0,2}$, $(d_0,d_1) \in \Dc_{s_0}$, where $s_{0,2}$ is introduced in Proposition \ref{prop:properinti}. From the local Cauchy problem for system \eqref{PS} in $L^\infty(\Rb) \times L^\infty(\Rb)$, we note that for each initial data $\binom{\Lambda_0}{\Upsilon_0}_{d_0,d_1, s_0,A}$, system \eqref{eq:LamUp} has a unique solution which stays in $\Vc_{A}(s)$ until some maximum time $s_* = s_*(d_0,d_1)$. If $ s_*(d_0,d_1)= +\infty$ for some $(d_0,d_1) \in \Dc_{s_0}$, then the proof is complete. Otherwise, we argue by contradiction and suppose that $s_*(d_0,d_1) < +\infty$ for any $(d_0,d_1) \in \Dc_{s_0}$. By continuity and the definition of $s_*$, we note that the solution at time $s_*$ is on the boundary of $\Vc_{A}(s_*)$.
Thus, at least one of the inequalities in the definition of $\Vc_A(s_*)$ is an equality. In the following proposition, we show that this can happen only for the two components $\theta_0(s_*)$ and $\theta_1(s_*)$. Precisely, we have the following result:

\begin{proposition}[Control of $\binom{\Lambda(s)}{\Upsilon(s)}$ by $(\theta_0(s), \theta_1(s))$ in $\Vc_A(s)$] \label{prop:redu} There exists $A_3 \geq 1$ such that for each $A \geq A_3$, there exists $s_{0,3}(A) \geq e$ such that for all $s_0 \geq s_{0,3}(A)$, the following holds:

If $\binom{\Lambda}{\Upsilon}$ is a solution of  \eqref{eq:LamUp} with initial data at $s = s_0$ given by \eqref{eq:intialdata} with $(d_0,d_1) \in \Dc_{s_0}$, and $\binom{\Lambda(s)}{\Upsilon(s)} \in \Vc_{A}(s)$ for all $s \in [s_0, s_1]$ for some $s_1 \geq s_0$ and $\binom{\Lambda(s_1)}{\Upsilon(s_1)} \in \partial \Vc_A(s_1)$, then:\\

\noindent $(i)\;$ (Reduction to a finite-dimensional problem) We have 
$$\big(\theta_0(s_1), \theta_1(s_1)\big) \in \partial \left( \left[-\frac{A}{s_1^2}, \frac{A}{s_1^2}\right]\right)^2.$$

\noindent $(ii)\;$ (Transverse outgoing crossing) There exists $\delta_0 > 0$ such that 
$$\forall \delta \in (0, \delta_0), \quad \binom{\Lambda(s_1 + \delta)}{\Upsilon(s_1 + \delta)} \not \in \Vc_A(s_1 + \delta). $$
\end{proposition}

\begin{remark} In $N$ dimensions, $\theta_0 \in \Rb$ and $\theta_1 \in \RN$. In particular, the finite-dimensional problem is of dimension $N + 1$. This is why in initial data \eqref{eq:intialdata}, one has to take $d_0 \in \Rb$ and $d_1 \in \RN$.

\end{remark}

The proof of Proposition \ref{prop:redu} is a direct consequence of the dynamics of system \eqref{eq:LamUp}. The idea is to project system \eqref{eq:LamUp} on the different components of the decomposition \eqref{decomoposeLamUp} and \eqref{def:LeUe}. However, because of the number of parameters in our problem ($p, q$ and $\mu$) and the coordinates in \eqref{decomoposeLamUp}, the computations become too long. That is why a whole section (Section \ref{sec:reduction}) is devoted to the proof of Proposition \ref{prop:redu}.  Let us now assume Proposition \ref{prop:redu} and continue the proof of Proposition \ref{prop:goalVA}. \\

Let $A \geq A_3$ and $s_0 \geq \max\{s_{0,2}, s_{0,3}\}$. From part $(i)$ of Proposition \ref{prop:redu}, it follows that 
$$\big(\theta_0(s_*), \theta_1(s_*)\big) \in \partial \left( \left[-\frac{A}{s_*^2}, \frac{A}{s_*^2}\right]\right)^2.$$
Hence, we may define the rescaled flow $\Theta$ at $s = s_*$ as follows:
\begin{align*}
\Theta: \Dc_{s_0} &\to \partial \big([-1,1]^2\big)\\
(d_0,d_1) & \mapsto \frac{s_*^2}{A}\big(\theta_0, \theta_1\big)_{d_0,d_1}(s_*).
\end{align*}
From Proposition \ref{prop:redu}$(ii)$, we see that $\Theta$ is continuous. On the other hand, from Proposition \ref{prop:properinti}, we see that when $(d_0,d_1) \in \partial \Dc_{s_0}$, we have the strict inequalities for the other components. Applying the transverse crossing property given in  Proposition \ref{prop:redu}$(ii)$, we see that $\binom{\Lambda(s)}{\Upsilon(s)}$ must leave $\Vc_A(s)$ at $s = s_0$, hence, $s_*(d_0,d_1) = s_0$. From  Proposition \ref{prop:properinti}$(i)$, the restriction of $\Theta$ to the boundary is of degree one. A contradiction then follows from the index theory. This means that there exists $(d_0,d_1) \in \Dc_{s_0}$ such that for all $s \geq s_0$, $\binom{\Lambda(s)}{\Upsilon(s)}_{d_0,d_1,s_0,A} \in \Vc_A(s)$. This concludes the proof of Proposition \ref{prop:goalVA} assuming that Propositions \ref{prop:redu} and \ref{prop:properinti} hold.
\end{proof}

\bigskip

Let us now use the result of Proposition \ref{prop:goalVA} to derive Theorem \ref{theo1}. 
\begin{proof}[Proof of Theorem \ref{theo1}] Recall that we have already chosen $a = 0$ at the beginning of Section \ref{sec:existence}, thanks to translation invariance of equation \eqref{PS}. We have already showed in  Proposition \ref{prop:goalVA} that there exist initial data of the form \eqref{eq:intialdata} such that the corresponding solution  $\binom{\Lambda(s)}{\Upsilon(s)}$ of system \eqref{eq:LamUp} satisfies $\binom{\Lambda(s)}{\Upsilon(s)} \in \Vc_A(s)$ for all $s \geq s_0$. This means that \eqref{eq:LUinVAest} holds for all $s \geq s_0$. From \eqref{def:LamUps} and \eqref{def:simiVars}, we then derive part $(ii)$ of Theorem \ref{theo1}. 

If $x_0 = 0$, then we see from \eqref{eq:asyTh1} that 
$$|u(0,t)| \sim \Gamma(T-t)^{-\frac{p+1}{pq-1}} \quad \text{and} \quad |v(0,t)| \sim \gamma(T-t)^{-\frac{q+1}{pq-1}} \quad \text{as}\; t \to T.$$
Hence, $u$ and $v$ both blow up at time $T$ at $x_0 = 0$.  It remains to show that if $x_0 \neq 0$, then $x_0$ is not a blowup point. The following result from Giga and Kohn \cite{GKcpam89} allows us to conclude:

\begin{proposition}[No blowup under some threshold] \label{prop:Noblowup} For all $C_0 > 0$, there is $\eta_0 > 0$ such that if $(u(x,t), v(x,t))$ solves
\begin{equation*}
\big |\partial_t u - \Delta u\big| \leq C_0\big(1+|v|^p\big), \quad \big| \partial_t v - \mu \Delta v\big| \leq C_0 \big( 1 + |u|^q\big)
\end{equation*}
and satisfies
$$|u(x,t)| \leq \eta_0(T-t)^{-\frac{p+1}{pq-1}}, \quad |v(x,t)| \leq \eta_0(T-t)^{-\frac{q+1}{pq-1}}$$
for all $(x,t) \in B(x_0, r) \times [T-r^2,T)$ for some $x_0 \in \Rb$ and $r > 0$, then $(u,v)$ does not blow up at $(x_0,T)$.
\end{proposition}
\begin{proof} Although Giga and Kohn give in \cite{GKcpam89} the proof only for the scalar case \eqref{eq:Scalar} (see Theorem 2.1, page 850 in \cite{GKcpam89}), their argument remains valid for system \eqref{PS} because of the following scaling invariant property of system \eqref{PS}: If $(u,v)$ solves \eqref{PS}, then 
$$\forall \lambda > 0, \quad (u_\lambda(x,t), v_{\lambda}(x,t)) := \left(\lambda^\frac{2(p+1)}{pq-1}u(\lambda x, \lambda^2 t), \lambda^\frac{2(q+1)}{pq-1}v(\lambda x, \lambda^2 t) \right),$$
does the same; and because the semigroup and the fundamental solution generated by $\eta \Delta$ with $\eta \in \{1, \mu\}$ have the same regularizing effect independently from $\eta$.
\end{proof}

\medskip

Indeed, we see from \eqref{eq:asyTh1} that 
$$
\sup_{|x| < \frac{|x_0|}{2}} (T-t)^\frac{p+1}{pq-1}|u(x,t)| \leq \Phi^*\left(\frac{|x_0|/2}{\sqrt{(T-t)\log(T-t)}} \right) + \frac{C}{\sqrt{\log(T-t)}} \to 0,
$$
and 
$$
\sup_{|x| < \frac{|x_0|}{2}} (T-t)^\frac{q+1}{pq-1}|v(x,t)| \leq \Psi^*\left(\frac{|x_0|/2}{\sqrt{(T-t)\log(T-t)}} \right) + \frac{C}{\sqrt{\log(T-t)}} \to 0,
$$
as $t \to T$, hence, $x_0$ is not a blowup point of $(u,v)$ from Proposition \ref{prop:Noblowup}. This concludes the proof of part $(i)$ of Theorem \ref{theo1}.\\

We now give the proof of part $(iii)$ of Thereom \ref{theo1}. Using the technique of Merle \cite{Mercpam92}, we derive the existence of a blowup profile $(u^*, v^*) \in \Cc^2(\Rb^*) \times \Cc^2(\Rb^*)$ such that 
$$(u(x,t),v(x,t)) \to (u^*(x), v^*(x)) \quad \text{as} \quad t \to T.$$
The profile $(u^*, v^*)$ is singular at the origin, as we will see shortly, after deriving its equivalent as $x \to 0$. Since our argument is exactly the same as in Zaag \cite{ZAAihn98} used for equation \eqref{eq:GL} with $\beta = 0$ (no new idea is needed), we just give the key arguments and kindly refer the reader to Section 4 in \cite{ZAAihn98} for more details. Consider $K_0 > 0$ to be fixed large enough later. If $x_0 \neq 0$ and $|x_0|$ is small enough, we introduce for all $(\xi, \tau) \in \Rb \times \left[-\frac{t_0(x_0)}{T - t_0(x_0)},1 \right)$,
$$g(x_0, \xi, \tau) = (T-t_0(x_0))^\frac{p+1}{pq-1}u(x,t), \quad h(x_0, \xi, \tau) = (T-t_0(x_0))^\frac{q+1}{pq-1}v(x,t),$$
where 
\begin{equation}\label{eq:defx0t_0}
x = x_0 + \xi\sqrt{T - t_0(x_0)}, \quad t = t_0(x_0) + \tau(T - t_0(x_0)),
\end{equation}\label{eq:deft0x0unique}
 and $t_0(x_0)$ is uniquely determined by 
 \begin{equation}\label{eq:relx0at0}
|x_0| = K_0\sqrt{(T - t_0(x_0))|\log (T - t_0(x_0))|}.
 \end{equation}
From the invariance of system \eqref{PS} under dilation, $(g(x_0, \xi, \tau), h(x_0, \xi, \tau))$ is also a solution of \eqref{PS} on its domain. From  \eqref{eq:deft0x0unique}, \eqref{eq:defx0t_0} and \eqref{eq:asyTh1}, we have
$$\sup_{|\xi| \leq 2 |\log (T-t_0(x_0))|^{1/4}} \left|g(x_0, \xi, 0) - \Phi^*(K_0)\right| \leq \frac{C}{|\log (T - t_0(x_0))|^{1/4}} \to 0,$$
and 
$$\sup_{|\xi| \leq 2 |\log (T-t_0(x_0))|^{1/4}} \left|h(x_0, \xi, 0) - \Psi^*(K_0)\right| \leq \frac{C}{|\log (T - t_0(x_0))|^{1/4}} \to 0,$$
as $x_0 \to 0$. Using the continuity with respect to initial data for system \eqref{PS} associated to a space-localization in the ball $B(0, |\xi| < |\log (T-t_0(x_0))|^{1/4})$, we show as in Section 4 of \cite{ZAAihn98} that 
$$\sup_{|\xi| \leq 2 |\log (T-t_0(x_0))|^{1/4}, 0 \leq \tau < 1} \left|g(x_0, \xi, 0) - \hat{g}_{K_0}(\tau)\right| \leq \epsilon(x_0) \to 0,$$
and 
$$\sup_{|\xi| \leq 2 |\log (T-t_0(x_0))|^{1/4}, 0 \leq \tau < 1} \left|h(x_0, \xi, 0) - \hat{h}_{K_0}(\tau)\right| \leq \epsilon(x_0) \to 0,$$
as $x_0 \to 0$, where 
$$\hat g_{K_0}(\tau) = \Gamma(1 - \tau + bK_0^2)^{-\frac{p+1}{pq-1}}, \quad \hat h_{K_0}(\tau) = \gamma(1 - \tau + bK_0^2)^{-\frac{q+1}{pq-1}},$$
is the solution of system \eqref{PS} with constant initial data $(\Phi^*(K_0), \Psi^*(K_0))$.

Making $\tau \to 1$ and using \eqref{eq:defx0t_0}, we see that 
\begin{align*}
u^*(x_0) &= \lim_{t \to T}u(x,t) = (T-t_0(x_0))^{-\frac{p+1}{pq-1}}\lim_{\tau \to 1}g(x_0, 0, \tau)\\
& \qquad \qquad\qquad \qquad \qquad \qquad\sim (T-t_0(x_0))^{-\frac{p+1}{pq-1}}\hat g_{K_0}(1),\\
v^*(x_0) &= \lim_{t \to T}v(x,t) = (T-t_0(x_0))^{-\frac{q+1}{pq-1}}\lim_{\tau \to 1}h(x_0, 0, \tau)\\
&\qquad \qquad\qquad \qquad \qquad \qquad \sim (T-t_0(x_0))^{-\frac{q+1}{pq-1}}\hat h_{K_0}(1),
\end{align*}
as $x_0 \to 0$. From \eqref{eq:relx0at0}, we have
$$|\log(T - t_0(x_0))| \sim 2\log |x_0|, \quad T -t_0(x_0) \sim \frac{|x_0|^2}{2K_0^2 |\log |x_0||} \quad \text{as}\quad x_0 \to 0,$$
hence,
$$u^*(x_0) \sim \Gamma \left(\frac{b|x_0|^2}{2|\log |x_0||} \right)^{-\frac{p+1}{pq-1}}, \quad v^*(x_0) \sim \gamma \left(\frac{b|x_0|^2}{2|\log |x_0||} \right)^{-\frac{q+1}{pq-1}},$$
as $x_0 \to 0$, which concludes the proof of part $(iii)$ of Theorem \ref{theo1}, assuming that Propositions \ref{prop:redu} and \ref{prop:properinti} hold.

\end{proof}

\section{Proof of the technical results.}\label{sec:dyn}
 In this section, we prove all the technical results used for the proof of the existence of a solution of system \eqref{eq:LamUp} satisfying \eqref{eq:goalconstruct}. In particular, we give the proofs of Propositions \ref{prop:properinti} and \ref{prop:redu}, each  in a separate subsection.

\subsection{Preparation of the initial data.}\label{sec:intialdata}
In this subsection, we give the proof of Proposition \ref{prop:properinti}. Let us start with some properties of the set $\Vc_A(s)$ introduced in Definition \ref{def:VA}:

\begin{proposition}[Properties of elements of $\Vc_A(s)$] \label{prop:properVA} For all $A \geq 1$, there exists $s_1(A) \geq 1$ such that for all $s \geq s_1$, if $\binom{\Lambda(s)}{\Upsilon(s)} \in \Vc_A(s)$, then\\

\noindent $(i)\;$ $\|\Lambda(s)\|_{L^\infty(|y| \leq 2K\sqrt{s})} + \|\Upsilon(s)\|_{L^\infty(|y| \leq 2K\sqrt{s})} \leq C\frac{A^{M+1}}{\sqrt{s}}.$\\

\noindent $(ii) \;$ $\|\Lambda(s)\|_{L^\infty(\Rb)} + \|\Upsilon(s)\|_{L^\infty(\Rb)} \leq C\frac{A^{M+2}}{\sqrt{s}}.$\\

\noindent $(iii)\;$ For all $y \in \Rb$, $|\Lambda(y,s)| + |\Upsilon(y,s)| \leq CA^{M+1}\frac{\log s}{s^2}(1 + |y|^{M+1}).$
\end{proposition}

\begin{proof} Take $A \geq 1$, $s \geq 1$ and assume that $\binom{\Lambda(s)}{\Upsilon(s)} \in \Vc_A(s)$. 

$(i)\;$ If $|y| \leq 2K\sqrt s$, since we have for all $0 \leq n \leq M$, $|\theta_n| + |\tilde{\theta}_n| \leq C\frac{A^{M+1}}{s^{\frac{n+1}{2}}}$ by Definition \ref{def:VA}, we write from \eqref{decomoposeLamUp},
\begin{align*}
|\Lambda(y,s)| &\leq \left(\sum_{n \leq M}|\theta_n||f_n| + |\tilde{\theta}_n||\tilde{f}_n|\right) + |\Lambda_-(y,s)|\\
& \leq C\sum_{n \leq M} \frac{A^{M+1}}{s^{\frac{n+1}{2}}}(1 + |y|)^n + \frac{A^{M+1}}{s^{\frac{M+2}{2}}}(1 + |y|)^{M+1} \leq \frac{C(K)A^{M+1}}{\sqrt s}, 
\end{align*}
(remember from Lemma \ref{lemm:diagonal} that $f_n, g_n, \tilde{f}_n$ and $\tilde g_n$ are polynomials of degree $n$).
The same estimate holds for $|\Upsilon(y,s)|$, which concludes the proof of $(i)$.

$(ii)$ From the definition \eqref{def:LeUe}, we have for $|y| \geq 2K\sqrt s$, $\left|\binom{\Lambda(y,s)}{\Upsilon(y,s)}\right| = \left|\binom{\Lambda_e(y,s)}{\Upsilon_e(y,s)}\right| \leq \frac{CA^{M+2}}{\sqrt s}$. Together with $(i)$, this yields the conclusion.

$(iii)$ We just use \eqref{decomoposeLamUp} and the fact that for all $0 \leq n \leq M$, $|\theta_n| + |\tilde{\theta}_n| \leq C\frac{A^{M+1} \log s}{s^2}$ by Definition \ref{def:VA} (use again the information on the polynomials' degree from Lemma \ref{lemm:diagonal}).  This finishes the proof of Proposition \ref{prop:properVA}.
\end{proof}

Clearly, Proposition \ref{prop:properinti} directly follows from the following result:

\begin{lemma}\label{lemm:properiniti} For all $A \geq 1$, there exists $s_2(A) \geq 1$ such that for all $s_0 \geq s_2$, if initial data for equation \eqref{eq:LamUp} are given by \eqref{eq:intialdata} (write $\binom{\Lambda_0}{\Upsilon_0}:= \binom{\Lambda_0}{\Upsilon_0}_{d_0,d_1,s_0,A}$ for simplicity), then 

$$\Lambda_{0,e} = \Upsilon_{0,e} = 0, \quad \left\|\frac{\Lambda_{0,-}(y)}{1 + |y|^{M+1}}\right\|_{L^\infty(\Rb)} + \left\|\frac{\Upsilon_{0,-}(y)}{1 + |y|^{M+1}}\right\|_{L^\infty(\Rb)} \leq \frac{CA(|d_0| + |d_1|)}{s_0^{\frac{M + 4}{2}}},$$
and all $|\theta_{0,n}|, |\tilde \theta_{0,n}|$ are less than $CA(|d_0| + |d_1|)e^{-s_0}$, except:
$$|\theta_{0,i} - \frac{A d_i}{s^2_0}| \leq CA|d_i|e^{-s_0}, \quad i = 0,1.$$
\end{lemma}
\begin{proof} The proof mainly relies on the projections of $\binom{\Lambda_0}{\Upsilon_0}$ on $\binom{f_n}{g_n}$ and $\binom{\tilde f_n}{\tilde{g}_n}$ defined in Lemma \ref{lemm:DefProjection}. Let us start by estimating the outer part. Note from the definition of $\chi$ given in \eqref{def:chi} that we have $\chi(2y,s_0)(1 - \chi(y,s_0)) = 0$. Thus, from \eqref{def:LeUe}, $\Lambda_{0,e} = \Upsilon_{0,e} = 0$. For the other components, let us rewrite \eqref{eq:intialdata} as follows:

$$\binom{ \Lambda_0(y)}{\Upsilon_0(y)} = \binom{\hat \Lambda_0(y)}{\hat{\Upsilon}_0(y)} + \binom{\hat \Lambda_0(y)}{\hat{\Upsilon}_0(y)} (\chi(2y,s_0) - 1),$$
where
$$\binom{\hat \Lambda_0(y)}{\hat{\Upsilon}_0(y)} = \frac{A}{s_0^2}\left(d_0\binom{f_0(y)}{g_0(y)} + d_1\binom{f_1(y)}{g_1(y)}\right);$$
the result will then follow by linearity.

From Remark \ref{rema:pro}, we see that all $P_{n,M}\binom{\hat \Lambda_0}{\hat{\Upsilon}_0}$ and $\tilde P_{n,M}\binom{\hat \Lambda_0}{\hat{\Upsilon}_0}$ are zero, except
$$P_{i,M}\binom{\hat \Lambda_0}{\hat{\Upsilon}_0} = \frac{Ad_i}{s_0^2}, \quad i = 0,1.$$
Using \eqref{decomoposeLamUp}, we see that
$$\binom{\hat \Lambda_{0,-}(y)}{\hat{\Upsilon}_{0,-}(y)} = \Pi_{-,M}\binom{\hat \Lambda_0}{\hat{\Upsilon}_0} \equiv 0.$$
It remains to handle $\binom{\hat \Lambda_0(y)}{\hat{\Upsilon}_0(y)} (\chi(2y,s_0) - 1)$. Since $(\chi(2y,s_0) - 1) = 0$ for $|y| \leq \frac{K}{2}\sqrt {s_0}$, we see that
$$\rho_\eta(\chi(2y,s_0) - 1) \leq Ce^{-s_0}\sqrt{\rho_\eta(y)}, \quad \eta = \{1, \mu\},$$
if $K \geq \sqrt{8\eta}$. Therefore, we derive from Lemma \ref{lemm:DefProjection} and symmetry that
\begin{equation}\label{est:Pnms0}
\left|P_{n,M}\left[\binom{\hat \Lambda_0}{\hat{\Upsilon}_0}(\chi(2y,s_0) - 1)\right] \right|  \leq C(n)A(|d_0| + |d_1|)e^{-s_0}, \;\; \text{for all $n \leq M$}.
\end{equation}
Similarly, \eqref{est:Pnms0} holds with $P_{n,M}$ replaced by $\tilde{P}_{n,M}$.

Furthermore, we have 
\begin{align*}
\left(|\Lambda_0(y)| + |\Upsilon_0(y)| \right) |\chi(2y,s_0) - 1| &\leq \frac{A(|d_0| + |d_1|)}{s_0^2}(1 + |y|) \frac{1 + |y|^M}{(\frac K2 \sqrt{s_0})^{M}} \\
&\leq \frac{CA(|d_0| + |d_1|)}{s_0^{\frac{M + 4}{2}}}(1 + |y|^{M+1}).
\end{align*}
Hence, by a straightforward estimate, we have

$$\left\|\frac{\Lambda_{-,0}(y) (\chi(2y,s_0) - 1)}{1 + |y|^{M+1}}\right\|_{L^\infty(\Rb)} \leq \frac{CA(|d_0| + |d_1|)}{s_0^{\frac{M + 4}{2}}},$$
and 
$$\left\|\frac{\Upsilon_{-,0}(y) (\chi(2y,s_0) - 1)}{1 + |y|^{M+1}}\right\|_{L^\infty(\Rb)} \leq \frac{CA(|d_0| + |d_1|)}{s_0^{\frac{M + 4}{2}}}.$$
This concludes the proof of Lemma \ref{lemm:properiniti}. Since Proposition \ref{prop:properinti} clearly follows from Lemma \ref{lemm:properiniti}, this concludes the proof of  Proposition \ref{prop:properinti} as well.
\end{proof}

\subsection{Reduction to a finite-dimensional problem.} \label{sec:reduction}
In this subsection, we give the proof of Proposition \ref{prop:redu}, which is the crucial part in our analysis. The idea of the proof is to project system \eqref{eq:LamUp} on the different components defined by \eqref{def:LeUe} and the decomposition \eqref{decomoposeLamUp}. More precisely, we claim that Proposition \ref{prop:redu} is a direct consequence of the following:

\begin{proposition}[Dynamics of system \eqref{eq:LamUp}]  \label{prop:dyn} There exists $A_3 \geq 1$ such that for all $A \geq A_3$, there exists $s_3(A) \geq 1$ such that the following holds for all $s_0 \geq s_3(A)$:\\
Assume that for all $s \in [\tau, \tau_1]$ for some $\tau_1 \geq \tau \geq s_0$, $\binom{\Lambda(s)}{\Upsilon(s)} \in \Vc_{A}(s)$, then the following holds for all $s \in [ \tau, \tau_1]$:\\

\noindent $(i)\;$ (ODEs satisfied by the positive modes) For $n = 0, 1$, we have
$$\left|\theta_n'(s) - \left(1 - \frac{n}{2}\right)\theta_n(s)\right| \leq \frac{C}{s^2}.$$

\noindent $(ii)\;$ (ODE satisfied by the null mode) 
$$\left|\theta_2'(s) - \frac 2s\theta_2(s)\right| \leq \frac{CA^3}{s^3}.$$

\noindent $(iii)\;$ (Control of the finite-dimensional part)
\begin{align*}
|\theta_j(s)| &\leq e^{-\left(\frac{j}{2} - 1 \right)(s -\tau)}|\theta_j(\tau)| + \frac{CA^{j-1}}{s^\frac{j+1}{2}},\quad 3\leq j\leq M,\\
|\tilde \theta_j(s)| &\leq e^{-\left(\frac{j}{2} + \frac{(p+1)(q+1)}{pq-1} \right)(s -\tau)}|\tilde \theta_j(\tau)| + \frac{CA^{j-1}}{s^\frac{j+1}{2}}, \quad 3\leq j \leq M,\\
|\tilde \theta_j(s)| &\leq e^{-\left(\frac{j}{2} + \frac{(p+1)(q+1)}{pq-1} \right)(s -\tau)}|\tilde \theta_j(\tau)| + \frac{C}{s^2}, \quad j = 0,1,2.
\end{align*}

\noindent $(iv)\;$ (Control of the infinite-dimensional part) 
\begin{align*}
&\left\|\frac{\Lambda_{-}(y,s)}{1 + |y|^{M+1}} \right\|_{L^\infty(\Rb)} + \left\|\frac{\Upsilon_{-}(y,s)}{1 + |y|^{M+1}} \right\|_{L^\infty(\Rb)}\\
& \leq Ce^{-\frac{(M+1)(s-\tau)}{4}} \left(\left\|\frac{\Lambda_{-}(y,\tau)}{1 + |y|^{M+1}} \right\|_{L^\infty(\Rb)} + \left\|\frac{\Upsilon_{-}(y,\tau)}{1 + |y|^{M+1}} \right\|_{L^\infty(\Rb)}
 \right) + \frac{CA^M}{s^\frac{M+2}{2}}.
\end{align*}

\noindent $(v)\;$ (Control of the outer part)
\begin{align*}
&\|\Lambda_e(s)\|_{L^\infty(\Rb)} + \|\Upsilon_e(s)\|_{L^\infty(\Rb)}\\
& \leq Ce^{-\frac{(r+1)(s - \tau)}{2(pq-1)}}\left(\|\Lambda_e(\tau)\|_{L^\infty(\Rb)} + \|\Upsilon_e(\tau)\|_{L^\infty(\Rb)}\right) +  \frac{CA^{M+1}}{\sqrt{s}}(1 + s - \tau),
\end{align*}
where $r = \min\{p,q\}$.
\end{proposition}

Because of the number of parameters in our problem ($p$, $q$ and $\mu$) and the coordinates in \eqref{decomoposeLamUp}, the proof of Proposition \ref{prop:dyn} is too long. For that reason, we will organize the rest of this subsection in 4 separate parts for the reader's convenience:\\

\noindent - \textbf{Part 1}: We  assume the result of Proposition \ref{prop:dyn} in order to  complete the proof of Proposition \ref{prop:redu}. The proof of Proposition \ref{prop:dyn} will be carried out in the next three parts. \\

\noindent - \textbf{Part 2}: We deal with system \eqref{eq:LamUp} to write ODEs satisfied by $\theta_n$ and $\tilde \theta_n$ for $n \leq M$. The definition of the projection of $\binom{\Lambda}{\Upsilon}$ on $\binom{f_n}{g_n}$ and $\binom{\tilde f_n}{g_n}$ given in Lemma \ref{lemm:DefProjection} will be the main tool to derive these ODEs. Then, we prove items $(i)$, $(ii)$ and $(iii)$ of Proposition \ref{prop:dyn}.\\

\noindent - \textbf{Part 3}: We derive from system \eqref{eq:LamUp} a system satisfied by $\binom{\Lambda_-}{\Upsilon_-}$ and prove item $(iv)$ of Proposition \ref{prop:dyn}.  Unlike the estimate on $\theta_n$ and $\tilde \theta_n$ where we use the properties of the linear operator $\Hc + \Mc$, here we use the operator $\Hc$. The fact that $M$ is large enough (as fixed in \eqref{eq:Mfixed}) is crucial in the proof, in the sense that this choice of $M$ allows us to successfully apply a Gronwall's inequality at the end for the control of the infinite-dimensional part.\\

\noindent - \textbf{Part 4}: In the shortest part, we project system \eqref{eq:LamUp} to write a system satisfied by $\binom{\Lambda_e}{\Upsilon_e}$ and prove item $(v)$ of Proposition \ref{prop:dyn}. As mentioned early, the linear operator of the equation satisfied by $\Lambda_e$ and $\Upsilon_e$ has a negative spectrum, which makes the control of $\|\Lambda_e(s)\|_{L^\infty(\Rb)}$ and $\|\Upsilon_e(s)\|_{L^\infty(\Rb)}$ easily.\\

\subsubsection{Proof of Proposition \ref{prop:redu} assuming Proposition \ref{prop:dyn}.} \label{sec:redufromdyn}
We give the proof of Proposition \ref{prop:redu} assuming that Proposition \ref{prop:dyn} holds. Consider $A \geq A_3$ and $s_0 = -\log T \geq s_3(A)$, where $A_3$ and $s_3$ are given in Proposition \ref{prop:dyn}. 

Since $\binom{\Lambda(s)}{\Upsilon(s)} \in \Vc_A(s)$ for all $[s_0,s_1]$ and $\binom{\Lambda(s_1)}{\Upsilon(s_1)} \in \partial\Vc_A(s_1)$, part $(i)$ will be proved if we show that for all $s \in [s_0, s_1]$ the following holds:
$$\|\Lambda_e(s)\|_{L^\infty(\Rb)} +  \|\Upsilon_e(s)\|_{L^\infty(\Rb)} \leq \frac{A^{M+2}}{2\sqrt{s}},$$
$$ \left\|\frac{\Lambda_-(y,s)}{1 + |y|^{M+1}} \right\|_{L^\infty(\Rb)} +  \left\|\frac{\Upsilon_-(y,s)}{1 + |y|^{M+1}} \right\|_{L^\infty(\Rb)} \leq \frac{A^{M+1}}{2s^{\frac{M+2}{2}}},$$

\begin{equation}\label{eq:allests}
 |\theta_j(s)|\leq \frac{A^j}{2s^\frac{j+1}{2}},\quad |\tilde{\theta}_j(s)| \leq \frac{A^j}{2s^\frac{j+1}{2}} \;\; \text{for}\;\; 3\leq j\leq M, 
\end{equation}

$$ |\tilde \theta_i(s)| \leq \frac{A^{2}}{2s^2}\;\; \text{for}\;\; i = 0, 1,2,$$

$$|\theta_2(s)| < \frac{A^4 \log s}{s^2}.$$

Define $\lambda = \log A$ and take $s_0 \geq \lambda$ so that for all $\tau \geq s_0$ and $s \in [\tau, \tau+\lambda]$, we have 
\begin{equation}\label{eq:staus0}
\tau \leq s \leq \tau + \lambda \leq \tau + s_0 \leq 2\tau, \quad \text{hence},\quad \frac{1}{2\tau} \leq \frac{1}{s} \leq \frac{1}{\tau} \leq \frac{2}{s}.
\end{equation}
We then consider the two following cases:\\
\paragraph{Case 1: $s \leq s_0 + \lambda$.} Using Proposition \ref{prop:dyn}$(iv)- (v)$ with $\tau = s_0$,  Proposition \ref{prop:properinti}$(ii)$ and \eqref{eq:staus0}, we deduce that
$$\|\Lambda_e(s)\|_{L^\infty(\Rb)} + \|\Upsilon_e(s)\|_{L^\infty(\Rb)} \leq \frac{CA^{M+1}}{\sqrt{s}}(1 + \log A) \leq \frac{A^{M+2}}{2\sqrt s},$$
$$\left\|\frac{\Lambda_-(y,s)}{1 + |y|^{M+1}} \right\|_{L^\infty(\Rb)} + \left\|\frac{\Upsilon_-(y,s)}{1 + |y|^{M+1}} \right\|_{L^\infty(\Rb)} \leq \frac{C}{s^\frac{M+2}{2}} + \frac{CA^M}{s^\frac{M+2}{2}}\leq \frac{A^{M+1}}{2s^\frac{M+2}{2}},$$
$$|\theta_j(s)| \leq \frac{C}{s^\frac{j + 1}{2}} + \frac{CA^{j-1}}{s^\frac{j+1}{2}} \leq \frac{A^j}{2 s^\frac{j+1}{2}}, \quad 3\leq j \leq M,$$
$$|\tilde \theta_j(s)| \leq \frac{C}{s^\frac{j + 1}{2}} + \frac{CA^{j-1}}{s^\frac{j+1}{2}} \leq \frac{A^j}{2 s^\frac{j+1}{2}}, \quad 3 \leq j \leq M,$$
$$|\tilde \theta_j(s)| \leq \frac{C}{s^2} + \frac{C}{s^2} \leq \frac{A^2}{2 s^2}, \quad j = 0, 1, 2,$$
provided that $A$ is large enough. 

To show that $|\theta_2(s)| < \frac{A^4\log s}{s^2}$ for all $s \in [s_0, s_0 + \lambda]$, since $\theta_2(s_0) < \frac{\log s_0}{s_0^2}$ from item $(ii)$ in Proposition \ref{prop:properinti}, we may argue by contradiction and assume that there is $s_* \in [s_0, s_0 + \lambda]$ such that 
$$ \text{for all}\; s \in [s_0, s^*),\;\; |\theta_2(s)| < \frac{A^4\log s}{s^2} \quad \text{and} \quad |\theta_2(s_*)| = \frac{A^4\log s_*}{s_*^2}.$$
Assuming that $\theta_2(s_*) > 0$ (the case $\theta_2(s_*) < 0$ is similar), we have 
$$\theta_2'(s_*) \geq \left.\frac{d}{ds}\left(\frac{A^4 \log s}{s^2}\right)\right \vert _{s=s_*} = \frac{A^4}{s_*^3} - \frac{2A^4\log s_*}{s_*^3},$$
on the one hand.

On the other hand, we have from $(ii)$ of Proposition \ref{prop:dyn}, 
$$\theta_2'(s_*) \leq -\frac{2A^4 \log s_*}{s_*^3} + \frac{CA^3}{s_*^3},$$
and a contradiction follows if $A \geq C+1$. Hence, the estimates given in \eqref{eq:allests} are proved for all $s \in [s_0, s_0 + \lambda]$.

\paragraph{Case 2: $s > s_0 + \lambda$.} Using parts $ (iv)- (v)$ of Proposition \ref{prop:dyn} with $\tau = s - \lambda > s_0$ and recalling that $\tau \geq \frac{s}{2}$ from \eqref{eq:staus0}, we write
$$\|\Lambda_e(s)\|_{L^\infty(\Rb)} + \|\Upsilon_e(s)\|_{L^\infty(\Rb)} \leq Ce^{-\frac{r + 1}{2(pq-1)}\lambda}\frac{A^{M+2}}{\sqrt{s/2}} + \frac{CA^{M+1}}{\sqrt{s}}(1 + \lambda),$$

$$\left\|\frac{\Lambda_-(y,s)}{1 + |y|^{M+1}} \right\|_{L^\infty(\Rb)} + \left\|\frac{\Upsilon_-(y,s)}{1 + |y|^{M+1}} \right\|_{L^\infty(\Rb)} \leq Ce^{-\frac{M+1}{4}\lambda}\frac{A^{M+1}}{(s/2)^\frac{M+2}{2}} + \frac{CA^M}{s^\frac{M+2}{2}}, $$

$$|\theta_j(s)| \leq Ce^{-\frac{j - 2}{2}\lambda}\frac{A^j}{(s/2)^\frac{j+1}{2}} + \frac{CA^{j-1}}{s^\frac{j+1}{2}}, \quad 3\leq j \leq M,$$

$$|\tilde \theta_j(s)| \leq Ce^{- \left(\frac{j}{2} + \frac{(p+1)(q+1)}{pq-1}\right)\lambda} \frac{A^j}{(s/2)^\frac{j+1}{2}} + \frac{CA^{j-1}}{s^\frac{j+1}{2}}, \quad 3 \leq j \leq M,$$

$$|\tilde \theta_j(s)| \leq Ce^{- \left(\frac{j}{2} + \frac{(p+1)(q+1)}{pq-1}\right)\lambda} \frac{A^2}{(s/2)^2} + \frac{C}{s^2} \leq \frac{A^2}{2 s^2}, \quad j = 0, 1, 2.$$
It is clear that if $A \geq A_5$ for some $A_5 \geq 1$ large enough, all the estimates in \eqref{eq:allests} hold, except for the strict inequality for $\theta_2(s)$ which is treated similarly as in the first case. This concludes the proof of part $(i)$ of Proposition \ref{prop:redu}. \\

The conclusion of part $(ii)$ directly follows from part $(i)$. Indeed, from item $(i)$, we know that for $n = 0$ or $1$ and $\omega = \pm 1$, we have  $\theta_n(s_1) = \omega \frac{A}{s_1^2}$. Therefore, using item $(i)$ of Proposition \ref{prop:dyn}, we see that 
$$\omega \theta_n'(s_1) \geq \left(1 - \frac{n}{2}\right)\omega\theta_n(s_1) - \frac{C}{s_1^2} \geq \frac{(1 - n/2)A - C}{s_1^2}.$$
Taking $A$ large enough gives $\omega \theta_n'(s_1) > 0$, which means that $\theta_n$ is traversal outgoing to the bounding curve $s \mapsto \omega As^{-2}$ at $s = s_1$. This concludes the proof of part $(ii)$ and finishes the proof of Proposition \ref{prop:redu} assuming that Proposition \ref{prop:dyn} holds.\\

Let us now give the proof of Proposition \ref{prop:dyn} in order to complete the proof of Proposition \ref{prop:redu}. The proof is given in the next three parts. 

\subsubsection{The finite-dimensional part.}
In this part, we give the proof of items $(i)$, $(ii)$ and $(iii)$ of Proposition \ref{prop:dyn}. We proceed in two steps:\\

\noindent - In the first step, we find the main contribution to the projections $P_{n, M}$ and $\tilde P_{n,M}$ of the various terms appearing in \eqref{eq:LamUp}.  \\
\noindent - In the second step, we gather all the estimates obtained in the first step to derive items $(i)$, $(ii)$ and $(iii)$ of Proposition \ref{prop:dyn}.\\

\paragraph{- Step 1: The projection of system \eqref{eq:LamUp} on the eigenfunctions of the operator $\Hc + \Mc$.} 

In the following, we will find the main contribution to the projections $P_{n, M}$ and $\tilde P_{n,M}$ of the five terms appearing in \eqref{eq:LamUp} (note that we handle $(\Hc + \Mc)\binom{\Lambda}{\Upsilon}$ as one term).

\subparagraph{$\bullet\;$ First term: $\partial_s\binom{\Lambda}{\Upsilon}$.} From the decomposition \eqref{decomoposeLamUp} and Lemma \ref{lemm:DefProjection}, its projection on $\binom{f_n}{g_n}$ and $\binom{\tilde f_n}{\tilde g_n}$ is $\theta_n'(s)$ and $\tilde{\theta}_n'(s)$, respectively:
\begin{equation}
P_{n,M}\left[\partial_s\binom{\Lambda}{\Upsilon}\right] = \theta_n' \quad \text{and} \quad \tilde P_{n,M}\left[\partial_s\binom{\Lambda}{\Upsilon}\right] = \tilde \theta_n'.
\end{equation}

\subparagraph{$\bullet\;$ Second term: $(\Hc + \Mc)\binom{\Lambda}{\Upsilon}$.} We claim the following:
\begin{lemma}[Projections of $(\Hc + \Mc)\binom{\Lambda}{\Upsilon}$ on $\binom{f_n}{g_n}$ and $\binom{\tilde f_n}{\tilde g_n}$ for $n \leq M$] \label{lemm:Pro2ndterm} For all $n \leq M$,\\

\noindent $(i)\;$ It holds that
\begin{align}
&\left|P_{n,M}\left[(\Hc + \Mc)\binom{\Lambda}{\Upsilon} \right] - \left(1 - \frac n2\right)\theta_n(s) \right| \nonumber \\
& +\left|\tilde P_{n,M}\left[(\Hc + \Mc)\binom{\Lambda}{\Upsilon} \right] - \left(\frac{(p+1)(q+1)}{pq-1} + \frac n2\right)\tilde \theta_n(s) \right| \nonumber \\
&\leq C\left\|\frac{\Lambda_-(y,s)}{1 + |y|^{M+1}} \right\|_{L^\infty(\Rb)} + C\left\|\frac{\Upsilon_-(y,s)}{1 + |y|^{M+1}} \right\|_{L^\infty(\Rb)}.\label{est:2ndterm}
\end{align}

\noindent $(ii)\;$ For all $A \geq 1$, there exists $s_4(A) \geq 1$ such that for all $s \geq s_4(A)$, if $\binom{\Lambda(s)}{\Upsilon(s)} \in \Vc_A(s)$, then:

\begin{align}
&\left|P_{n,M}\left[(\Hc + \Mc)\binom{\Lambda}{\Upsilon} \right] - \left(1 - \frac n2\right)\theta_n(s) \right| \nonumber \\
& +\left|\tilde P_{n,M}\left[(\Hc + \Mc)\binom{\Lambda}{\Upsilon} \right] - \left(\frac{(p+1)(q+1)}{pq-1} + \frac n2\right)\tilde \theta_n(s) \right| \leq C\frac{A^{M+1}}{s^\frac{M+2}{2}}.\label{est:2ndtermVA}
\end{align}
\end{lemma}

\begin{proof} Let us write from \eqref{decomoposeLamUp}
\begin{align*}
(\Hc + \Mc)\binom{\Lambda}{\Upsilon}  &= (\Hc + \Mc)\left(\sum_{n \leq M}\theta_n(s)\binom{f_n(y)}{g_n(y)} + \tilde{\theta}_n(s)\binom{\tilde f_n(y)}{\tilde g_n(y)}\right)\\
& + (\Hc + \Mc)\binom{\Lambda_{-,M}(y,s)}{\Upsilon_{-,M}(y,s)} := L_1 + L_2.
\end{align*}
Using Lemma \ref{lemm:diagonal}, we write 
$$L_1 = \sum_{n \leq M} \left(1 - \frac n2\right) \theta_n(s)\binom{f_n}{g_n} - \left(\frac{(p+1)(q+1)}{pq-1} + \frac{n}{2} \right)\tilde \theta_n(s)\binom{\tilde f_n}{\tilde g_n}.$$
From Remark \ref{rema:pro}, we see that 
$$P_{n,M}(L_1) = \left(1 - \frac n2\right) \theta_n(s), \quad \tilde P_{n,M}(L_1) = - \left(\frac{(p+1)(q+1)}{pq-1} + \frac{n}{2} \right)\tilde \theta_n(s).$$

We now deal with $L_2$. Let us write $L_2 = \Hc\binom{\Lambda_-}{\Upsilon_-} + \Mc \binom{\Lambda_-}{\Upsilon_-} := L_{2,1} + L_{2,2}$. Using Lemma \ref{lemm:DefProjection}, we have 
$$P_{n,M}(L_{2,1}) = \sum_{j = 0}^{\left[\frac{M-n}{2}\right]} A_{n + 2j,n}\Pi_{n+2j}(L_{2,1}) + B_{n + 2j,n}\hat \Pi_{n+2j}(L_{2,1}),$$
$$\tilde P_{n,M}(L_{2,1}) = \sum_{j = 0}^{\left[\frac{M-n}{2}\right]} \tilde A_{n + 2j,n}\Pi_{n+2j}(L_{2,1}) + \tilde B_{n + 2j,n}\hat \Pi_{n+2j}(L_{2,1}).$$
By the definitions of $\Hc$, $\Pi_m$ and $\hat \Pi_m$ given in \eqref{def:Qn} and \eqref{def:Qnhat} and the fact that $\Lc_\eta$ for $\eta = \{1, \mu\}$ is self-adjoint with respect to $\rho_\eta$ (see \eqref{eq:Ladjoint}) together with \eqref{eq:partnegVW}, we have for all $m = n + 2j \leq M$, 
\begin{align*}
\Pi_{m}(L_{2,1}) & = \Pi_{m}\binom{\Lc_1 \Lambda_-}{\Lc_\mu \Upsilon_-} = \|h_m\|_{{\rho_1}}^{-2}\int_\Rb \big(\Lc_1 \Lambda_-(y,s)\big)h_m(y)\rho_1 dy\\
& = \|h_m\|_{{\rho_1}}^{-2}\int_\Rb \Lambda_-(y,s)\big(\Lc_1 h_m(y)\big)\rho_1 dy\\
& = -\frac{m}{2}\|h_m\|_{{\rho_1}}^{-2}\int_\Rb \Lambda_-(y,s)h_m(y)\rho_1 dy = 0,
\end{align*}
and
\begin{align*}
\hat \Pi_{m}(L_{2,1}) &= \hat \Pi_{m}\binom{\Lc_1 \Lambda_-}{\Lc_\mu \Upsilon_-} = \|\hat h_m\|_{{\rho_\mu}}^{-2}\int_\Rb \big(\Lc_\mu \Upsilon_-(y,s)\big)\hat h_m(y)\rho_\mu dy\\
& = \|\hat h_m\|_{{\rho_\mu}}^{-2}\int_\Rb \Upsilon_-(y,s)\big(\Lc_\mu \hat h_m(y)\big)\rho_\mu dy\\
& = -\frac{m}{2}\|\hat h_m\|_{{\rho_\mu}}^{-2}\int_\Rb \Upsilon_-(y,s)\hat h_m(y)\rho_\mu dy = 0.
\end{align*}
Thus, $P_{n,M}(L_2) = P_{n,M}(L_{2,2})$ and $\tilde P_{n,M}(L_2) = \tilde P_{n,M}(L_{2,2})$. By straightforward computation, they are controlled by $C\left\|\frac{\Lambda_-(y,s)}{1 + |y|^{M+1}} \right\|_{L^\infty(\Rb)} + C\left\|\frac{\Upsilon_-(y,s)}{1 + |y|^{M+1}} \right\|_{L^\infty(\Rb)}$. This concludes the proof of \eqref{est:2ndterm}. Since $\binom{\Lambda(s)}{\Upsilon(s)} \in \Vc_A(s)$ (see Definition \ref{def:VA}), the right hand side of \eqref{est:2ndterm} is bounded by $C\frac{A^{M+1}}{s^\frac{M+2}{2}}$, which yields \eqref{est:2ndtermVA}. This finishes the proof of Lemma \ref{lemm:Pro2ndterm}.
\end{proof}

\subparagraph{$\bullet\;$ Third term: $V\binom{\Lambda}{\Upsilon} = \binom{V_1 \Upsilon}{V_2 \Lambda}$.} We claim the following:

\begin{lemma}[Power series of $V_1$ and $V_2$ as $s \to +\infty$] \label{lemm:estVys} The functions $V_1(y,s)$ and $V_2(y,s)$ given in \eqref{def:Vys} satisfy
\begin{equation}\label{est:Viorder1}
|V_i(y,s)| \leq \frac{C(1 + |y|^2)}{s} \quad \forall y \in \Rb, \; s \geq 1,
\end{equation}
and for all $k \in \mathbb{N}^*$, 
\begin{equation}\label{est:Viorderk}
V_i(y,s) = \sum_{j = 1}^k \frac{1}{s^j}W_{i,j}(y) + \tilde{W}_{i,k}(y,s),
\end{equation}
where $W_{i,j}(y)$ is an even polynomial of degree $2j$ and $\tilde{W}_{i,k}(y,s)$ satisfies
$$|\tilde{W}_{i,k}(y,s)| \leq \frac{C(1 + |y|^{2k + 2})}{s^{k+1}}, \quad \forall |y| \leq \sqrt{s}, \; s \geq 1.$$
Moreover, we have for all $|y| \leq \sqrt{s}$ and $s \geq 1$,
\begin{align}
\left|V_1(y,s) +\frac{p(p-1)\gamma^{p-2}b}{(pq-1)s}g_2(y) \right|  &\leq \frac{C(1 + |y|^4)}{s^2},\label{est:V1g2}\\
\left|V_2(y,s) +\frac{q(q-1)\Gamma^{q-2}b}{(pq-1)s}f_2(y) \right|  &\leq \frac{C(1 + |y|^4)}{s^2}.\label{est:V2f2}
\end{align}
\end{lemma}
\begin{proof} Since the estimates of $V_1$ and $V_2$ are the same, we only deal with $V_1$. Let us introduce 
$$F(w) = p(w^{p-1} - \gamma^{p-1})$$
and consider $z = \frac{|y|^2}{s}$, we see from \eqref{def:Vys} that
$$V_1(y,s) = F\left(\Psi^*(z) + \frac{E}{s}\right),$$
where 
$$\Psi^*(z) = \gamma(1 + bz)^{-\frac{q+1}{pq-1}}, \quad E = \frac{2b\gamma(q + \mu)}{pq-1}, \quad b = \frac{(pq-1)(2pq +p+q)}{4pq(p+1)(q+1)(\mu + 1)}.$$
Note that there exist positive constants $c_0$ and $s_0$ such that $|\Psi^*(z)|$ and $\left|\Psi^*(z) + \frac{E}{s}\right|$ are both larger than $\frac{1}{c_0}$ and smaller than $c_0$, uniformly in $|z| < 1$ and $s \geq s_0$. Since $F(w)$ is $\Cc^\infty$ for $\frac{1}{c_0} \leq |w| \leq c_0$, we Taylor-expand it around $w = \Psi^*(z)$ as follows: for all $s \geq s_0$ and $|z| < 1$,
$$\left|F\left(\Psi^*(z) + \frac{E}{s}\right) - F\left(\Psi^*(z)\right) \right| \leq \frac{C}{s},$$
$$\left|F\left(\Psi^*(z) + \frac{E}{s}\right) - F\left(\Psi^*(z)\right) - \sum_{j = 1}^k\frac{1}{s^j}F_j(\Psi^*(z)) \right| \leq \frac{C}{s^{k+1}},$$
where $F_j = F^{(j)}(\Psi^*(z))$ are $\Cc^\infty$. Furthermore, we Taylor-expand $F(w)$ and $F_j(w)$ around $w = \Psi^*(0)$ as follows: for all $s \geq s_0$ and $|z| < 1$,
$$\left|F\left(\Psi^*(z) + \frac{E}{s}\right) - F\left(\Psi^*(0)\right) \right| \leq C|z| + \frac{C}{s},$$
\begin{align}
&\left|F\left(\Psi^*(z) + \frac{E}{s}\right) - F\left(\Psi^*(0)\right) - \sum_{j = 1}^k c_{0,j}z^j - \sum_{j = 1}^k \sum_{l = 0}^{k - j} \frac{c_{j,l}}{s^j}z^l \right|\nonumber \\
& \leq  C|z|^{k + 1} + \sum_{j=1}^k\frac{C}{s^j}|z|^{k - j + 1}  + \frac{C}{s^{k+1}}.\label{eq:Ftaylork}
\end{align}
Since $F(\Psi^*(0)) = F(\gamma) = 0$, this yields estimates \eqref{est:Viorder1} and \eqref{est:Viorderk} for $V_1$, when $|z| < 1$ and $s \geq s_0$. Since $V_1$ is bounded, \eqref{est:Viorder1} is also valid for $|z| \geq 1$, that is for $|y| \geq \sqrt{s}$ and for $s \geq 1$. Estimate \eqref{est:V1g2} directly follows from \eqref{eq:Ftaylork} with $k = 1$ and the definition of $g_2$ given in \eqref{def:f2g2Ndim}. This concludes the proof of Lemma \ref{lemm:estVys}.
\end{proof}

We now use Lemma \ref{lemm:estVys} to derive the projections of $\binom{V_1 \Upsilon}{V_2 \Lambda}$ on $\binom{f_n}{g_n}$ and $\binom{\tilde{f}_n}{\tilde{g}_n}$. More precisely, we have the following:

\begin{lemma}[Projections of $\binom{V_1 \Upsilon}{V_2 \Lambda}$ on $\binom{f_n}{g_n}$ and $\binom{\tilde{f}_n}{\tilde{g}_n}$] \label{lemm:Pro3rdterm} $\quad$\\

\noindent $(i)\;$ For all $n \leq M$ and for all $s \geq 1$, we have

\begin{align*}
&\left|P_{n,M}\binom{V_1 \Upsilon}{V_2 \Lambda} \right| + \left|\tilde P_{n,M}\binom{V_1 \Upsilon}{V_2 \Lambda} \right|\nonumber\\
& \leq \frac{C}{s}\sum_{i = n-2}^M\big(|\theta_i(s)| + |\tilde \theta_i(s)|\big) + \sum_{i = 0}^{n-3}\frac{C}{s^{\frac{n-i}{2}}}\big(|\theta_i(s)| + |\tilde \theta_i(s)|\big)\\
& + \frac{C}{s}\left(\left\|\frac{\Lambda_-(y,s)}{1 + |y|^{M+1}} \right\|_{L^\infty(\Rb)} + \left\|\frac{\Upsilon_-(y,s)}{1 + |y|^{M+1}} \right\|_{L^\infty(\Rb)}\right).
\end{align*}

\noindent $(ii)\;$  For all $A \geq 1$, there exists $s_5(A) \geq 1$ such that for all $s \geq s_5(A)$, if $\binom{\Lambda(s)}{\Upsilon(s)} \in \Vc_A(s)$, then:\\
- for $3 \leq n \leq M$,
\begin{equation*}
\left|P_{n,M}\binom{V_1 \Upsilon}{V_2 \Lambda} \right| + \left|\tilde P_{n,M}\binom{V_1 \Upsilon}{V_2 \Lambda} \right|\leq \frac{CA^{n - 2}}{s^\frac{n + 1}{2}}.
\end{equation*}
- for $n = 0, 1, 2$,
\begin{equation*}
\left|P_{n,M}\binom{V_1 \Upsilon}{V_2 \Lambda} \right| + \left|\tilde P_{n,M}\binom{V_1 \Upsilon}{V_2 \Lambda} \right|\leq \frac{C}{s^2}.
\end{equation*}
\end{lemma}
\begin{proof} From Lemma \ref{lemm:DefProjection}, let us write 
\begin{align}
P_{n,M}\binom{V_1 \Upsilon}{V_2 \Lambda} &= \sum_{j = 0}^{\left[\frac{M-n}{2}\right]}A_{n+2j,n}\Pi_{n + 2j}\binom{V_1 \Upsilon}{V_2 \Lambda} + B_{n+2j,n}\hat \Pi_{n + 2j}\binom{V_1 \Upsilon}{V_2 \Lambda},\label{exp:PnM}\\
\tilde P_{n,M}\binom{V_1 \Upsilon}{V_2 \Lambda} &= \sum_{j = 0}^{\left[\frac{M-n}{2}\right]}\tilde A_{n+2j,n}\Pi_{n + 2j}\binom{V_1 \Upsilon}{V_2 \Lambda} + \tilde B_{n+2j,n}\hat \Pi_{n + 2j}\binom{V_1 \Upsilon}{V_2 \Lambda}.\label{exp:PnMti}
\end{align}
Thus, it is enough to estimate $\Pi_m\binom{V_1 \Upsilon}{V_2 \Lambda}$ and $\hat \Pi_m\binom{V_1 \Upsilon}{V_2 \Lambda}$ for $m = n + 2j \leq M$. By definition \eqref{def:Qn} and decomposition \eqref{decomoposeLamUp}, we write

\begin{align*}
\|h_m\|^2_{\rho_1}\Pi_m\binom{V_1 \Upsilon}{V_2 \Lambda} &= \int_{\Rb}V_1 \Upsilon h_m \rho_1 = \int_{\Rb} V_1 \Upsilon_- h_m \rho_1 dy \\
&+ \sum_{i = 0}^M \theta_i(s)\int_\Rb V_1 g_i h_m \rho_1 dy\\
& +  \sum_{i = 0}^M \tilde \theta_i(s)\int_\Rb V_1 \tilde g_i h_m \rho_1 dy := I_1 + I_2 + I_3.
\end{align*}
Using \eqref{est:Viorder1}, the first term can be bounded by 
\begin{align*}
|I_1| &\leq \frac{C}{s} \left\|\frac{\Upsilon_-(y,s)}{1 + |y|^{M+1}} \right\|_{L^\infty(\Rb)} \int_{\Rb}(1 + |y|^{3 + m + M}) \rho_1 dy \leq \frac{C}{s} \left\|\frac{\Upsilon_-(y,s)}{1 + |y|^{M+1}} \right\|_{L^\infty(\Rb)}.
\end{align*}
Since $I_2$ and $I_3$ are estimated in the same way, we only focus on the estimate for $I_2$.\\
- If $i \geq m - 2$, we use \eqref{est:Viorder1} to bound $\left|\int_{\Rb}V_1 g_i h_m \rho_1 dy\right| \leq \frac{C}{s}$.\\
- If $i \leq m - 3$, we use \eqref{est:Viorderk} to show that 
\begin{equation}\label{est:V1fihm}
\left|\int_{\Rb}V_1 g_i h_m \rho_1 dy\right| \leq \frac{C}{s^{\frac{m - i}{2}}}.
\end{equation}
Let us prove \eqref{est:V1fihm}. We use \eqref{est:Viorderk} to write
\begin{align*}
\int_{\Rb}V_1 g_i h_m \rho_1 dy &= \int_{|y| > \sqrt s} V_1 g_i h_m \rho_1 dy + \sum_{l = 1}^k \frac{1}{s^l} \int_{|y| < \sqrt s} W_{1,l}g_i h_m \rho_1 dy\\
& + \Oc \left(\frac{1}{s^{k + 1}} \int_{|y| < \sqrt s}(1 + |y|^{2k + 2})|g_i||h_m||\rho_1|dy \right),
\end{align*}
where we take $k$ to be the largest integer such that $i + 2k < m$, that is $k = \left[\frac{m - i - 1}{2}\right]$. \\
Since $|\rho_1(y)| \leq Ce^{-cs}$ when $|y| > \sqrt s$, the first term can be bounded by $Ce^{-cs}$. The last term is bounded by $\frac{C}{s^{k + 1}} \leq \frac{C}{s^{\frac{m - i}{2}}}$. For the second term, we note that $\text{deg}(g_iW_{1,l}) = i + 2l \leq i + 2k < m$, hence, we have by the orthogonality \eqref{eq:orthohnhm},
$$\int_\Rb W_{1,l}g_i h_m \rho_1 dy = 0.$$
This directly follows that the second term is bounded by $Ce^{-cs}$ and concludes the proof of \eqref{est:V1fihm}.  Hence, we have just proved that
\begin{align}
\left|\Pi_{m}\binom{V_1 \Upsilon}{V_2 \Lambda}\right| &\leq  \frac{C}{s} \left\|\frac{\Upsilon_-(y,s)}{1 + |y|^{M+1}} \right\|_{L^\infty(\Rb)} \nonumber\\
& + \frac{C}{s}\sum_{i = m -2}^M \big(|\theta_i(s)| + |\tilde \theta_i(s)|\big) + \sum_{i = 0}^{m - 3} \frac{1}{s^{\frac{m-i}{2}}}\big(|\theta_i(s)| + |\tilde \theta_i(s)|\big).\label{est:Pim}
\end{align}
Similarly, it holds that
\begin{align}
\left|\hat \Pi_{m}\binom{V_1 \Upsilon}{V_2 \Lambda}\right| &\leq  \frac{C}{s} \left\|\frac{\Lambda_-(y,s)}{1 + |y|^{M+1}} \right\|_{L^\infty(\Rb)} \nonumber \\
& + \frac{C}{s}\sum_{i = m -2}^M \big(|\theta_i(s)| + |\tilde \theta_i(s)|\big) + \sum_{i = 0}^{m - 3} \frac{1}{s^{\frac{m-i}{2}}}\big(|\theta_i(s)| + |\tilde \theta_i(s)|\big).\label{est:Pimti}
\end{align}
Injecting \eqref{est:Pim} and \eqref{est:Pimti} into \eqref{exp:PnM} and \eqref{exp:PnMti} and making the change of index $m = n + 2j$, we obtain 
\begin{align*}
&\left|P_{n,M}\binom{V_1 \Upsilon}{V_2 \Lambda} \right| + \left|\tilde P_{n,M}\binom{V_1 \Upsilon}{V_2 \Lambda} \right|\nonumber\\
& \leq \frac{C}{s}\left(\left\|\frac{\Lambda_-(y,s)}{1 + |y|^{M+1}} \right\|_{L^\infty(\Rb)} + \left\|\frac{\Upsilon_-(y,s)}{1 + |y|^{M+1}} \right\|_{L^\infty(\Rb)}\right)\\
& \quad + \frac{C}{s}\sum_{i = n-2}^M\big(|\theta_i(s)| + |\tilde \theta_i(s)|\big) + \sum_{i = 0}^{n-3}\frac{C}{s^{\frac{n-i}{2}}}\big(|\theta_i(s)| + |\tilde \theta_i(s)|\big)\\
&\qquad + \sum_{m = n + 1}^M \left\{\frac{C}{s}\sum_{i = m-2}^M\big(|\theta_i(s)| + |\tilde \theta_i(s)|\big) + \sum_{i = 0}^{m-3}\frac{C}{s^{\frac{m-i}{2}}}\big(|\theta_i(s)| + |\tilde \theta_i(s)|\big)\right\}
\end{align*}
We rewrite the last term as follows:
\begin{align*}
&\sum_{m = n + 1}^M \left\{\frac{C}{s}\sum_{i = m-2}^M\big(|\theta_i(s)| + |\tilde \theta_i(s)|\big) + \sum_{i = 0}^{m-3}\frac{C}{s^{\frac{m-i}{2}}}\big(|\theta_i(s)| + |\tilde \theta_i(s)|\big)\right\}\\
&=\sum_{m = n + 1}^M \left\{\frac{C}{s}\sum_{i = m-2}^M\big(|\theta_i(s)| + |\tilde \theta_i(s)|\big) + \sum_{i = 0}^{n-3}\frac{C}{s^{\frac{m - i}{2}}}\big(|\theta_i(s)| + |\tilde \theta_i(s)|\big)\right.\\
&\qquad \qquad \left.+ \sum_{i = n-2}^{m-3}\frac{C}{s^{\frac{m - i}{2}}}\big(|\theta_{i}| + |\tilde{\theta}_{i}| \big)\right\}\\
&\leq \frac{C}{s}\sum_{i = n-2}^M\big(|\theta_i(s)| + |\tilde \theta_i(s)|\big) + \sum_{i = 0}^{n-3}\frac{C}{s^{\frac{n -i}{2}}}\big(|\theta_i(s)| + |\tilde \theta_i(s)|\big).
\end{align*}
This concludes the proof of item $(i)$. Using the definition \ref{def:VA} of $\Vc_A(s)$, item $(ii)$ simply follows from item $(i)$. This finished the proof of Lemma \ref{lemm:Pro3rdterm}.

\end{proof}

Using estimate \eqref{est:V1g2} and \eqref{est:V2f2}, we further refine the estimate concerning the projection of $\binom{V_1 \Upsilon}{V_2 \Lambda}$ on $\binom{f_2}{g_2}$ as follows:

\begin{lemma}[Projection of $\binom{V_1 \Upsilon}{V_2 \Lambda}$ on $\binom{f_2}{g_2}$]\label{lemm:Pro3rdtermf2g2} $\quad $\\
$(i)\;$ It holds that 
\begin{align*}
\left|P_{2,M}\binom{V_1 \Upsilon}{V_2 \Lambda} + \frac{2}{s}\theta_2(s)\right|& \leq \frac{C}{s} \left(\sum_{j = 0, j \ne 2}^M|\theta_j(s)| + \sum_{j = 0}^M|\tilde \theta_j(s)|\right)\\
& + \frac Cs\left(\left\|\frac{\Lambda_-(y,s)}{1 + |y|^{M+1}} \right\|_{L^\infty(\Rb)}+\left\|\frac{\Upsilon_-(y,s)}{1 + |y|^{M+1}} \right\|_{L^\infty(\Rb)} \right).
\end{align*}
$(ii)\;$ For all $A \geq 1$, there exists $s_6(A) \geq 1$ such that for all $s \geq s_6(A)$, if $\binom{\Lambda(s)}{\Upsilon(s)} \in \Vc_A(s)$, then:
\begin{equation*}
\left|P_{2,M}\binom{V_1 \Upsilon}{V_2 \Lambda} + \frac{2}{s}\theta_2(s)\right| \leq \frac{CA^3}{s^3}.
\end{equation*}
\end{lemma}
\begin{proof} Using \eqref{est:V1g2}, \eqref{est:V2f2} and decomposition \eqref{decomoposeLamUp}, let us write for all $|y| < \sqrt s$,
\begin{align*}
\binom{V_1 \Upsilon}{V_2 \Lambda}& = \frac{\theta_2(s)}{s}\binom{W_{1} g_2}{W_{2} f_2} + \frac{1}{s}\sum_{j = 0, j \neq 2}^M \theta_j(s) \binom{W_1 g_j}{W_2f_j} + \frac{1}{s}\sum_{j = 0}^M \tilde \theta_j(s) \binom{W_1 \tilde g_j}{W_2 \tilde f_j}\\
&+ \frac{1}{s}\binom{W_1 \Upsilon_-}{W_2\Lambda_-} + \Oc \left(\frac{1 + |y|^4}{s^2}\binom{|\Upsilon|}{|\Lambda|} \right) = I_1 + I_2 + I_3 + I_4 + I_5,
\end{align*}
where 
$$W_1(y) = -\frac{p(p-1)\gamma^{p - 2}b}{pq-1}g_2(y), \quad W_2(y) = -\frac{q(q - 1)\Gamma^{q - 2} b}{pq-1}f_2(y).$$
We first note that 
\begin{align*}
|P_{2,M}(I_2 + I_3 + I_4 )|& \leq \frac{C}{s}\left(\sum_{j = 0, j \neq 2}^M |\theta_j(s)| + \sum_{j = 0}^M |\tilde \theta_j(s)|\right)\\
& + \frac{C}{s}\left(\left\|\frac{\Lambda_-(y,s)}{1 + |y|^{M+1}} \right\|_{L^\infty(\Rb)} + \left\|\frac{\Upsilon_-(y,s)}{1 + |y|^{M+1}} \right\|_{L^\infty(\Rb)} \right),
\end{align*}
and 
\begin{align*}
|P_{2,M}(I_5)|& \leq \frac{C}{s^2}\left(\sum_{j = 0, j \neq 2}^M |\theta_j(s)| + \sum_{j = 0}^M |\tilde \theta_j(s)|\right)\\
& + \frac{C}{s^2}\left(\left\|\frac{\Lambda_-(y,s)}{1 + |y|^{M+1}} \right\|_{L^\infty(\Rb)} + \left\|\frac{\Upsilon_-(y,s)}{1 + |y|^{M+1}} \right\|_{L^\infty(\Rb)} \right).
\end{align*}
Therefore, the problem is reduced to prove that
\begin{equation}\label{est:P2Mf2g2}
P_{2,M}\binom{W_1 g_2}{W_2 f_2}= -2.
\end{equation}
To do so, let us write 
$$W_1(y)g_2(y) = -\frac{p(p-1)\gamma^{p - 2}b}{pq-1}g^2_2(y) = \alpha_4 h_4(y) + \alpha_2 h_2(y) + \alpha_0h_0,$$
and 
$$W_2(y)f_2(y) = -\frac{q(q - 1)\Gamma^{q - 2} b}{pq-1}f_2^2(y)= \beta_4 \hat h_4(y) + \beta_2 \hat h_2 + \beta_0 \hat h_0,$$
where $h_0$, $h_2$, $h_4$ and $\hat h_0$, $\hat h_2$, $\hat h_4$ are defined as in \eqref{eq:hntildeN1} with $\eta = 1$ and $\eta = \mu$ respectively, and
$$\alpha_4 = -\frac{b \gamma^p p(p-1) (q+1)^2}{pq-1},$$
$$\beta_4 = - \frac{b \Gamma^q q(q-1) (p+1)^2}{pq-1},$$
$$\alpha_2 = - \frac{4 b\gamma^p p(p-1) (q+1) [2(q + \mu) + 3(1 - \mu)]}{pq-1},$$
$$\beta_2 = -\frac{4b\Gamma^q q(q-1)(p+1)[2(p\mu + 1) - 3(1 - \mu)]}{pq-1},$$
$$\alpha_0 = - \frac{b\gamma^p p(p-1)[8(q+1)^2 + 4(1-\mu)^2]}{pq-1},$$
$$\beta_0 = - \frac{b\Gamma^q q(q - 1)[8\mu^2(p+1)^2 + 4(1 - \mu)^2]}{pq-1}.$$
Using the definition of $P_{2,M}$ given in \eqref{def:Pn} and the orthogonality \eqref{eq:orthohnhm}, we see that 
\begin{align*}
P_{2,M}\binom{W_1 g_2}{W_2 f_2} &= A_{2,2}\Pi_{2}\binom{W_1 g_2}{W_2 f_2} + B_{2,2}\hat \Pi_{2}\binom{W_1 g_2}{W_2 f_2}\\
& + A_{4,2}\Pi_{4}\binom{W_1 g_2}{W_2 f_2} + B_{4,2}\hat \Pi_{4}\binom{W_1 g_2}{W_2 f_2}\\
&= A_{2,2}\alpha_2 + B_{2,2}\beta_2 + A_{4,2}\alpha_4 + B_{4,2}\beta_4,
\end{align*}
where the values of $A_{2,2}$, $B_{2,2}$, $A_{4,2}$ and $B_{4,2}$ are explicitly given by \eqref{eq:AnBn} and \eqref{eq:Ann2Bnn2}, that is
$$A_{2,2} = \frac{q}{\Gamma(2pq + p + q)}, \quad B_{2,2} = \frac{p}{\gamma(2pq+p+q)}, $$
$$A_{4,2} = -\frac{\tilde e_{4,2}}{\Gamma \gamma(2pq + p + q)}, \quad B_{4,2} = \left(\frac{p+1}{q+1}\right)\frac{\tilde e_{4,2}}{\gamma^2(2pq + p + q)}, $$
$$\tilde e_{4,2} = \frac{12pq\gamma(q+1)(1 - \mu)}{3pq + p + q - 1}.$$
A straightforward calculation yields
$$A_{2,2}\alpha_2 + B_{2,2}\beta_2 + A_{4,2}\alpha_4 + B_{4,2}\beta_4 = -2,$$
from which \eqref{est:P2Mf2g2} is proved and part $(i)$ follows. Part $(ii)$ simply follows from part $(i)$ and Definition \ref{def:VA} of $\Vc_A(s)$. This concludes the proof of Lemma \ref{lemm:Pro3rdtermf2g2}.
\end{proof}

\subparagraph{$\bullet\;$ Fourth term: $\binom{F_1(\Upsilon, y,s)}{F_2(\Lambda, y,s)}$.} 

We first claim the following:
\begin{lemma}[Decompositions of $F_1$ and $F_2$] \label{lemm:decomB} The functions $F_1(\Upsilon, y,s)$ and $F_2(\Lambda, y,s)$ given in \eqref{def:Bys} can be decomposed for all  $|\Lambda| \leq 1$, $|\Upsilon| \leq 1$ as follows: for all $s \geq 1$ and $|y| < \sqrt s$,
\begin{align*}
\left|F_1(\Upsilon, y,s) - \sum_{k = 2}^{M+1} \Upsilon^k \sum_{l = 0}^M \frac{1}{s^l}\left[F_{1,k}^l\left( \frac{y}{\sqrt s}\right) + \tilde F_{1,k}^l(y,s)\right] \right| \leq C|\Upsilon|^{M+2} + \frac{C}{s^{M+1}}, 
\end{align*}
and 
\begin{align*}
\left|F_2(\Lambda, y,s) - \sum_{k = 2}^{M+1} \Lambda^k \sum_{l = 0}^M \frac{1}{s^l}\left[F_{2,k}^l\left( \frac{y}{\sqrt s}\right) + \tilde F_{2,k}^l(y,s)\right] \right| \leq C|\Lambda|^{M+2} + \frac{C}{s^{M+1}}, 
\end{align*}
where $F_{i,k}^l$ is an even polynomials of degree less or equal to $M$ and $\tilde F_{i,k}^l(y,s)$ satisfies
$$ |\tilde{F}^l_{i,k}(y,s)| \leq \frac{C(1 + |y|^{M+1})}{s^\frac{M+1}{2}}.
$$
On the other hand, we have for all $y \in \Rb$ and $s \geq 1$,
\begin{equation}\label{eq:F1pF2q}
|F_1(\Upsilon,y,s)| \leq C|\Upsilon|^{\bar p}, \quad |F_2(\Lambda, y,s)| \leq C|\Lambda|^{\bar q},
\end{equation}
where $\bar{p} = \min\{2,p\}$ and $\bar{q} = \min\{q,2\}$.
\end{lemma}
\begin{proof} We only deal with $F_1(\Upsilon,y,s)$ since the same proof holds for $F_2(\Lambda,y,s)$. We first note that in the region $\{|y| < \sqrt s\}$ and for $s \geq s_0$ for some $s_0 \geq 1$, $\psi(y,s)$ is bounded from above and from below. Thus, we Taylor expand $F_1$ in term of $\Upsilon$ and write
$$\left|F_1(\Upsilon, y,s) - \sum_{k = 2}^{M+1}E_{1,k}(\psi)\Upsilon^k \right| \leq C|\Upsilon|^{M+2}.$$
Now, we expand $E_{1,k}(\psi)$ in terms of the variable $\frac{1}{s}$, and write
$$\left|E_{1,k}(\psi) - \sum_{l = 0}^M \frac{1}{s^l}E^l_{1,k}(\Psi^*) \right| \leq \frac{C}{s^{M+1}}.$$
Then, we expand $E_{1,k}^l(\Psi^*)$ in terms of $z = \frac{y}{\sqrt s}$ as follows:
$$\left|E^l_{1,k}(\Psi^*) - \sum_{i = 0}^{M/2}e_{1,k}^{l,i} |z|^{2i} \right| \leq C|z|^{M+2}.$$
Finally, we set 
\begin{equation}\label{eq:519}
F_{1,k}^l(z) = \sum_{i = 0}^{M/2}e_{1,k}^{l,i} |z|^{2i} \quad \text{and} \quad \tilde{F}_{1,k}^l(y,s) =  E^l_{1,k}(\Psi^*) - F_{1,k}^l\left(\frac{y}{\sqrt s}\right),
\end{equation}

which yields the desired result. This concludes the proof of Lemma \ref{lemm:decomB}.

\end{proof}

Using Lemma \ref{lemm:decomB}, let us now find estimates on the projection of $\binom{F_1}{F_2}$ on $\binom{f_n}{g_n}$ and $\binom{\tilde f_n}{\tilde{g}_n}$. In particular, we claim the following:

\begin{lemma}[Projections of $\binom{F_1}{F_2}$ on $\binom{f_n}{g_n}$ and $\binom{\tilde f_n}{\tilde{g}_n}$] \label{lemm:Pro4thterm} For all $A \geq 1$, there exists $s_7(A) \geq 1$ such that for all $s \geq s_7(A)$, if $\binom{\Lambda(s)}{\Upsilon(s)} \in \Vc_A(s)$, then:\\
- for $3 \leq n \leq M$, 
\begin{equation*}
\left|P_{n,M}\binom{F_1}{F_2}\right| + \left|\tilde P_{n,M}\binom{F_1}{F_2}\right| \leq \frac{CA^n}{s^\frac{n + 2}{2}},
\end{equation*}
- for $n = 0, 1, 2$,
\begin{equation*}
\left|P_{n,M}\binom{F_1}{F_2}\right| + \left|\tilde P_{n,M}\binom{F_1}{F_2}\right| \leq \frac{C}{s^3}.
\end{equation*}
\end{lemma}
\begin{proof} Let us write from Lemma \ref{lemm:DefProjection} the projections of $\binom{F_1}{F_2}$ on $\binom{f_n}{g_n}$ and $\binom{\tilde f_n}{\tilde{g}_n}$ for $n \leq M$ as follows:
\begin{align*}
P_{n,M}\binom{F_1}{F_2} &= \sum_{j = 0}^{\left[\frac{M - n}{2}\right]}A_{n + 2j,n}\Pi_{n + 2j}\binom{F_1}{F_2} + B_{n + 2j,n}\hat \Pi_n \binom{F_1}{F_2},\\
\tilde P_{n,M}\binom{F_1}{F_2} &= \sum_{j = 0}^{\left[\frac{M - n}{2}\right]}\tilde A_{n + 2j,n}\Pi_{n + 2j}\binom{F_1}{F_2} + \tilde B_{n + 2j,n}\hat \Pi_n \binom{F_1}{F_2}.
\end{align*}
We see that it is enough to estimate $\Pi_m\binom{F_1}{F_2}$ and $\hat \Pi_m\binom{F_1}{F_2}$ with $m = n + 2j \leq M$, since it implies the same estimate for $P_{n,M}$ and $\tilde{P}_{n,M}$. Since the estimates for $\Pi_m$ and $\hat \Pi_m$ are similar, we only deal with $\Pi_m\binom{F_1}{F_2}$ which is defined as follows:
$$\|h_m\|^2_{\rho_1}\Pi_m\binom{F_1}{F_2} = \int_{\Rb}F_1 h_m \rho_1 dy.$$
Using Lemma \ref{lemm:decomB}, let us write
\begin{align*}
\int_{\Rb}F_1 h_m \rho_1 dy & =  \int_{|y| < \sqrt s} \sum_{k = 2}^M \Upsilon^k \sum_{l = 0}^M \frac{1}{s^l}\left[F_{1,k}^l\left( \frac{y}{\sqrt s}\right) + \tilde F_{1,k}^l(y,s)\right] h_m \rho_1 dy\\
&+ \int_{|y| > \sqrt s} F_1 h_m \rho_1 dy \\
&+ \Oc \left(\int_{|y| < \sqrt s} |h_m| \left( |\Upsilon|^{M+2} + \frac{1}{s^{M+1}} \right) \rho_1 dy \right) = I_1 + I_2 + I_3.
\end{align*}
We use part $(iii)$ of Proposition \ref{prop:properVA} to get the estimate
$$|I_3| \leq CA^{(M+1)^2}\int_{|y| < \sqrt s}(1 + |y|^{m + (M+1)^2})\left(\left(\frac{\log s}{s^2}\right)^{M+1} +   \frac 1 {s^{M+1}}\right)\rho_1 dy \leq \frac{C}{s^\frac{M+2}{2}},$$
for all $s \geq A^\frac{2(M+1)^2}{M}$.
From part $(ii)$ of Proposition \ref{prop:properVA} and \eqref{eq:F1pF2q}, we see that 
$$|F_1(\Upsilon, y,s)| \leq C|\Upsilon|^{\bar{p}} \leq \frac{CA^{(M+2)\bar{p}}}{s^{\frac{\bar{p}}{2}}} \leq C,$$
for all $y \in \Rb$ and $s \geq A^{2(M+2)}$. Since $\sqrt{\rho_1(y)}\leq Ce^{-cs}$ for $|y| > \sqrt s$, we then get 
$$|I_2| \leq Ce^{-cs}.$$ 
Let us now estimate $I_1$. We write 
$$\Upsilon^k = \left(\sum_{j = 0}^M \big(\theta_j g_j + \tilde \theta_j \tilde g_j\big) + \Upsilon_- \right)^k, \quad F^l_{1,k}\left(\frac{y}{\sqrt s}\right) = \sum_{i = 0 }^{M/2}\frac{e_{1,k}^{l,i}}{s^i}y^{2i},$$
where $e_{1,k}^{l,i}$ are the coefficients of the polynomial $F^l_{1,k}$ defined in \eqref{eq:519}. We note from part $(ii)$ of Proposition \ref{prop:properVA} that $\|\Upsilon(s)\|_{L^\infty} \leq C$ for all $s \geq A^{2(M+2)}$, from which we derive
$$|\Upsilon^k - \Upsilon_+^k| \leq C\left(|\Upsilon_-|^k + |\Upsilon_+|^{k -1}|\Upsilon_-|\right),$$
where $k \geq 2$, and $\Upsilon_+ = \sum_{j = 0}^M \big(\theta_j g_j + \tilde \theta_j \tilde g_j\big)$. From Definition \ref{def:VA} of $\Vc_A(s)$, we have 
$$|\Upsilon_-| \leq \frac{A^{M+1}}{s^\frac{M+2}{2}}(1 + |y|^{M+1}), \quad |\Upsilon_+| \leq \frac{CA^M\log s}{s^2}(1 + |y|^M),$$
which yields 
$$\left|\Upsilon^k - \Upsilon_+^k \right|\leq \frac{CA^{k(M+1)}}{s^\frac{M+4}{2}}(1 + |y|^{k(M+1)}).$$
Hence, the contribution coming from $\Upsilon_-$ to the estimate of $I_1$ is controlled by $\frac{CA^{k(M+1)}}{s^\frac{M+4}{2}} \leq \frac{CA^n}{s^\frac{n+2}{2}}$ for $n \leq M$ and $s$ large enough. On the other hand, we notice that $F^l_{1,k}\left(\frac{y}{\sqrt{s}}\right)\Upsilon_+^k$ is a polynomial function in $y$ where the coefficient of the term of degree $m$ is bounded by $\frac{A^m}{s^\frac{m + 2}{2}} \leq \frac{A^n}{s^\frac{n + 2}{2}}$ for $n \geq 3$, and by $\frac{A^{M^2}\log ^2s}{s^4} \leq \frac{C}{s^3}$ for $n = 0 , 1, 2$. Note also from the orthogonality \eqref{eq:orthohnhm} that for all polynomial functions $f$ of degree $n < m$, we have $\int_{\mathbb{R}}f h_m \rho_1 dy = 0$.
This implies that $\int_{|y| < \sqrt s} \Upsilon_+^k \tilde{F}^l_{1,k}(y,s) h_m \rho_1 dy$  is bounded by $\frac{A^n}{s^\frac{n + 2}{2}}$ for $n \geq 3$, and by $\frac{A^{M^2}\log ^2s}{s^4} \leq \frac{C}{s^3}$ for $n = 0 , 1, 2$. From part $(i)$ of Proposition \ref{prop:properVA} and the definition of $\tilde{F}^{l}_{1,k}$ given in Lemma \ref{lemm:decomB}, we deduce that for all $l$ and $k$,
$$\left|\int_{|y| < \sqrt s} \frac{1}{s^l} \Upsilon^k \tilde{F}^l_{1,k}(y,s) h_m \rho_1 dy\right| \leq \frac{C}{s^{l + k + \frac{M+1}{2}}} \leq \frac{C}{s^\frac{M+2}{2}}.$$
This concludes the proof of Lemma \ref{lemm:Pro4thterm}.
\end{proof}

\subparagraph{$\bullet\;$ Fifth term: $\binom{R_1}{R_2}$.}
We first expand $R_1(y,s)$ and $R_2(y,s)$ as a power series of $\frac{1}{s}$ as $s \to +\infty$, uniformly for $|y| < \sqrt s$. More precisely, we claim the following:
\begin{lemma}[Power series of $R_1$ and $R_2$ as $s \to +\infty$] \label{lemm:expandR1R2}
For all $m \in \mathbb{N}$, the functions $R_1(y,s)$ and $R_2(y,s)$ given in \eqref{def:Rys} can be expanded as follows: for all $|y| < \sqrt s$ and $s \geq 1$,
\begin{equation}\label{eq:Riexpand}
\left|R_{i}(y,s) - \sum_{k = 1}^{m-1} \frac{1}{s^{k+1}}R_{i,k}(y)\right| \leq \frac{C(1 + |y|^{2m})}{s^{m + 1}},
\end{equation}
where $R_{i,k}$ is a polynomial of degree $2k$. In particular,
\begin{align}
R_{1,1} = \frac{b\Gamma(p+1)}{pq-1} \left(-1 + \frac{6bp(q+1)}{pq-1} - \frac{2bp(q+1)(p-1)(q+\mu)}{(pq-1)^2} \right)y^2&\nonumber\\
+ \frac{2b\Gamma(p\mu+1)}{pq-1} - \frac{4b^2q(q-1)\gamma^{p}(q+\mu)^2}{(pq-1)^3},&\label{eq:R1exps3}\\
R_{2,1} =\frac{b\gamma(q+1)}{pq-1} \left(-1 + \frac{6b\mu q(p+1)}{pq-1} - \frac{2b q(p+1)(q-1)(p\mu+1)}{(pq-1)^2} \right)y^2&\nonumber\\
+ \frac{2b\gamma(q + \mu)}{pq-1} - \frac{4b^2p(p-1)\Gamma^{q}(p\mu + 1)^2}{(pq-1)^3}.&\label{eq:R2exps3}
\end{align}
\end{lemma}
\begin{proof} Let us consider $z = \frac{y}{\sqrt s}$ and write from  \eqref{def:varphiys} and \eqref{def:psiys},
\begin{align*}
&\varphi(y,s) = \Phi^*(z) + \frac{D}{s}, \quad D = \frac{2b\Gamma(p\mu + 1)}{pq-1},\\
&\psi(y,s) = \Psi^*(z) + \frac{E}{s}, \quad E = \frac{2b\gamma(q + \mu)}{pq-1},
\end{align*}
where $\Phi^*,\Psi^*$ are defined as in \eqref{def:fgpro}, and $b$ is given by \eqref{def:val_b}.

Using the fact that $(\Phi^*, \Psi^*) \equiv (\Phi_0, \Psi_0)$ satisfies \eqref{eq:Phi0Psi0}, we can write
from \eqref{def:Rys},
\begin{align*}
R_1(y,s)&= \frac{z}{2s}\cdot \nabla_z\Phi^* + \frac{D}{s^2} + \frac{1}{s}\Delta_z \Phi^* - \frac{(p+1)D}{(pq-1)s} + F\left(\Psi^* + \frac{E}{s} \right) - F(\Psi^*),\\
R_2(y,s) &= \frac{z}{2s}\cdot \nabla_z\Psi^* + \frac{E}{s^2} + \frac{\mu}{s}\Delta_z \Psi^* - \frac{(q+1)E}{(pq-1)s} + G\left(\Phi^* + \frac{D}{s} \right) - G(\Phi^*),
\end{align*}
where $F(\xi) = \xi^p$ and $G(\xi) = \xi^q$.

We only deal with $R_1$ because the estimate for $R_2$ follows similarly. For $|z| < 1$, there exist positive constants $c_0$ and $s_0$ such that $|\Phi^*(z)|$, $|\Psi^*(z)|$, $\left|\Phi^*(z) + \frac{D}{E} \right|$ and $\left|\Psi^*(z) + \frac{E}{s} \right|$ are lager than $\frac{1}{c_0}$ and smaller than $c_0$, uniformly in $|z| < 1$ and $s \geq s_0$. Since $F(\xi)$ is $\Cc^\infty$ for $\frac{1}{c_0} \leq |\xi| \leq c_0$, we expand it around $\xi = \Psi^*(z)$ as follows:
$$\left|F\left(\Psi^* + \frac{D}{s} \right) - F(\Psi^*) - \sum_{j = 1}^m\frac{1}{s^m}F_j(\Psi^*(z))\right| \leq \frac{C}{s^{m+1}},$$
where $F_j(\xi)$ are $\Cc^\infty$. Hence, we can expand $F_j(\xi)$ around $\xi = \Psi^*(0)$ and write
 $$\left|F\left(\Psi^* + \frac{D}{s} \right) - F(\Psi^*) - \sum_{j = 1}^m\frac{1}{s^m} \sum_{l = 0}^{m - j}c_{j,l}|z|^{2l}\right| \leq  \sum_{j = 1}^m \frac{C}{s^j}z^{2(m-j) + 2} + \frac{C}{s^{m+1}}.$$
Similarly, we have
$$\left|\frac{z}{2s}\cdot \nabla_z \Psi^*(z) - \frac{z^2}{s}\sum_{j = 0}^{m-2}d_jz^{2j}\right| \leq \frac{C}{s}|z|^{2m},$$
$$\left|\frac{1}{s}\Delta_z\Psi^*(z) - \frac{1}{s}\sum_{j =0}^{m-1}b_j z^{2j} \right| \leq \frac{C}{s}|z|^{2m},$$
$$\left|\Psi^*(z) - \sum_{j = 0}^{m - 1}e_jz^{2j} \right| \leq C|z|^{2m}.$$
Gathering all the above expansion to the expression of $R_1(y,s)$, we find that the term of order $\frac{1}{s}$ is given by 
$$ - \frac{2b\Gamma(p+1)}{pq-1} - \frac{D(p+1)}{pq-1} + pE\gamma^{p-1} = 0,$$
$$(\text{note that for $R_2$, it is} \; - \frac{2b\mu\gamma(q+1)}{pq-1} - \frac{E(q+1)}{pq-1} + qD\Gamma^{q-1} = 0)$$
hence, \eqref{eq:Riexpand} follows. The formulas \eqref{eq:R1exps3} and \eqref{eq:R2exps3} are obtained by explicit calculations. This concludes the proof of Lemma \ref{lemm:expandR1R2}.
\end{proof}

We now use Lemma \ref{lemm:expandR1R2} to estimate the projections of $\binom{R_1}{R_2}$ on $\binom{f_n}{g_n}$ and $\binom{\tilde f_n}{\tilde g_n}$ as follows:
\begin{lemma}[Projections of $\binom{R_1}{R_2}$ on $\binom{f_n}{g_n}$ and $\binom{\tilde f_n}{\tilde g_n}$] \label{lemm:Pro5thterm} For all $s \geq 1$ and $n \leq M$, we have\\
- if $n$ is odd, then
\begin{equation}\label{est:PnMR1R1nodd}
P_{n,M}\binom{R_1(y,s)}{R_2(y,s)} = \tilde{P}_{n,M}\binom{R_1(y,s)}{R_2(y,s)} = 0,
\end{equation}
- if $n \geq 4$ is even, then
\begin{equation}\label{est:PnMR1R1n4}
\left|P_{n,M}\binom{R_1(y,s)}{R_2(y,s)} \right| + \left|\tilde P_{n,M}\binom{R_1(y,s)}{R_2(y,s)} \right| \leq \frac{C}{s^{\frac{n + 2}{2}}}.
\end{equation}

- if $n = 0$ and $n = 2$, then
\begin{equation}\label{est:PnMR1R1n02}
\left|P_{0,M}\binom{R_1(y,s)}{R_2(y,s)} \right| + \left|\tilde P_{0,M}\binom{R_1(y,s)}{R_2(y,s)} \right| + \left|\tilde P_{2,M}\binom{R_1(y,s)}{R_2(y,s)} \right| \leq \frac{C}{s^2},
\end{equation}
and
\begin{equation}\label{est:PnMR1R1n2}
\left|P_{2,M}\binom{R_1(y,s)}{R_2(y,s)} \right| \leq \frac{C}{s^3}.
\end{equation}

\end{lemma}
\begin{proof} let us write from Lemma \ref{lemm:DefProjection} for all $n \leq M$,
\begin{align*}
P_{n,M}\binom{R_1}{R_2} &= \sum_{j = 0}^{\left[\frac{M-n}{2}\right]} A_{n + 2j,n}\Pi_{n+2j}\binom{R_1}{R_2} + B_{n+2j,n}\hat \Pi_{n+2j}\binom{R_1}{R_2},\\
\tilde P_{n,M}\binom{R_1}{R_2} &= \sum_{j = 0}^{\left[\frac{M-n}{2}\right]} \tilde A_{n + 2j,n}\Pi_{n+2j}\binom{R_1}{R_2} + \tilde B_{n+2j,n}\hat \Pi_{n+2j}\binom{R_1}{R_2}.
\end{align*}
Since $R_1(y,s)$ and $R_2(y,s)$ are even functions in $y$, we deduce that 
 $$\Pi_j\binom{R_1(y,s)}{R_2(y,s)} = \hat \Pi_j\binom{R_1(y,s)}{R_2(y,s)}  = 0 \quad \text{if $j$ is odd},$$
which follows \eqref{est:PnMR1R1nodd}. Now when $n \geq 4$ is even, we use \eqref{eq:Riexpand} with $m = \left[\frac{n}{2}\right]$ and write for $i = 1,2$,
$$R_i(y,s) = \tilde{R}_{i,\frac{n}{2}}(y,s) + \Oc\left(\frac{1 + |y|^n}{s^{\frac{n}{2} + 1}}\right), \quad \text{for all}\; |y| < \sqrt s,\; s\geq 1,$$
where $\tilde R_i$ is polynomial in $y$ of degree less then $n - 1$. It is enough to estimate $\Pi_{k}\binom{R_1}{R_2}$ and $\hat \Pi_k \binom{R_1}{R_2}$ with $n \leq k = n + 2j \leq M$ since the same bound holds for $P_{n,M}$ and $\tilde P_{n,M}$. We only estimate $\Pi_{k}\binom{R_1}{R_2}$ because the same  proof holds for $\hat \Pi_k \binom{R_1}{R_2}$. From definition \eqref{def:Qn}, we write
\begin{align*}
\|h_k\|^2_{\rho_1}|\Pi_{k}\binom{R_1}{R_2}|& = \left|\int_{\Rb} R_1h_k \rho_1 dy\right|\\
&\leq \left| \int_{|y| \leq \sqrt s}\tilde{R}_{1, \frac{n}{2}}h_k \rho_1 dy\right| + \Oc\left(\frac{1}{s^{\frac{n}{2} + 1}}\int_{|y| < \sqrt s}(1 + |y|^n)|h_k|\rho_1dy  \right)\\
& + \left|\int_{|y| > \sqrt s}R_1 h_k \rho_1dy\right|\leq 0 + \frac{C}{s^{\frac n2 +1}} + Ce^{-cs},
\end{align*}
where we used the fact that $\text{deg}(\tilde{R}_{1, \frac n2}) \leq n - 1 < k$ and the orthogonality \eqref{eq:orthohnhm} resulting in $\int_{\Rb}\tilde{R}_{1, \frac{n}{2}} h_k \rho_1 dy = 0$, and that the integral over the domain $|y| > \sqrt s$ is controlled by $Ce^{-cs}$. We have proved \eqref{est:PnMR1R1n4}. When $n = 0$ and $n = 2$, estimate \eqref{est:PnMR1R1n02} directly follows from \eqref{eq:Riexpand} with $m = 1$, that is
$$|R_i(y,s)| \leq \frac{C(1 + |y|^2)}{s^2}.$$
It remains to prove \eqref{est:PnMR1R1n2}. To this end, let us write from \eqref{eq:Riexpand}
$$R_i(y,s) = \frac{1}{s^2}R_{i,1}(y) + \Oc\left(\frac{1 + |y|^4}{s^3}\right),$$
where $R_{1,1}$ and $R_{2,1}$ are given by \eqref{eq:R1exps3} and \eqref{eq:R2exps3}. Estimate \eqref{est:PnMR1R1n2} will follow if we show that 
$$P_{2,M}\binom{R_{1,1}(y)}{R_{2,1}(y)} = 0.$$
Using Lemma \ref{lemm:DefProjection} and the orthogonality \eqref{eq:orthohnhm} (note that $\text{deg}(R_{i,1}) = 2$, $i = 1,2$), we obtain
\begin{align*}
P_{2,M}\binom{R_{1,1}(y)}{R_{2,1}(y)} &= \frac{1}{2pq+p+q} \left(\frac{q}{\Gamma}\Pi_2\binom{R_{1,1}(y)}{R_{2,1}(y)} + \frac{p}{\gamma} \hat\Pi_2\binom{R_{1,1}(y)}{R_{2,1}(y)}\right)\\
& = \frac{1}{2pq+p+q} \left(\frac{q}{\Gamma}\|h_2\|^{-2}_{\rho_1} \int_{\Rb}R_{1,1}h_2 \rho_1 dy + \frac{p}{\gamma} \|\hat h_2\|^{-2}_{\rho_\mu}\int_{\Rb}R_{2,1}\hat h_2 \rho_\mu dy \right)\\
&= \frac{bq(p+1)}{(pq-1)(2pq + p + q)} \left(-1 + \frac{6bp(q+1)}{pq-1} - \frac{2bp(q+1)(p-1)(q+\mu)}{(pq-1)^2} \right)\\
& + \frac{bp(q+1)}{(pq-1)(2pq + p +q)} \left(-1 + \frac{6b\mu q(p+1)}{pq-1} - \frac{2b q(p+1)(q-1)(p\mu+1)}{(pq-1)^2} \right)\\
& = 0,
\end{align*}
after a straightforward simplification. This concludes the proof of Lemma \ref{lemm:Pro5thterm}.
\end{proof}

\paragraph{- Step 2: Proof of items $(i)$, $(ii)$ and $(iii)$ of Proposition \ref{prop:dyn}.}
In Step 1, we have obtained all the contribution in the projections $P_{n,M}$ and $\tilde P_{n,M}$ for the terms appearing in system \eqref{eq:LamUp}. More precisely, taking the projection of \eqref{eq:LamUp} on $\binom{f_n}{g_n}$ and $\binom{\tilde f_n}{\tilde g_n}$ for $n \leq M$, we see that for all $s \in [\tau, \tau_1]$:\\
- if $n = 0$ and $n = 1$, then
$$\left|\theta_n'(s) - \left(1 - \frac n2 \right)\theta_n(s) \right| \leq \frac{C}{s^2},$$
which is the conclusion of part $(i)$ of Proposition \ref{prop:dyn},\\
- if $n = 2$, then 
$$\left|\theta_2'(s) - \frac{2}{s}\theta_2(s) \right| \leq \frac{CA^3}{s^3},$$
which is the conclusion of part $(ii)$ of Proposition \ref{prop:dyn},\\
- if $3 \leq n \leq M$, then 
$$\left|\theta_n'(s) + \left(\frac{n - 2}{2}\right)\theta_n(s) \right| \leq \frac{CA^{n-1}}{s^{\frac{n+1}{2}}},$$
$$\left|\tilde \theta_n'(s) + \left(\frac n2 + \frac{(p+1)(q+1)}{pq-1}\right)\tilde \theta_n(s) \right| \leq \frac{CA^{n-1}}{s^{\frac{n+1}{2}}},$$
and $n = 0, 1, 2$,
$$\left|\tilde \theta_n'(s) + \left(\frac n2 + \frac{(p+1)(q+1)}{pq-1}\right)\tilde \theta_n(s) \right| \leq \frac{C}{s^2}.$$
Integrating these differential equations between $\tau$ and $s$ gives the conclusion of part $(iii)$ of Proposition \ref{prop:dyn}.

\subsubsection{The infinite-dimensional part.} \label{sec:negetivepart}
We prove item $(iv)$ of Proposition \ref{prop:dyn} in this part. We proceed in two steps:\\

\noindent - Firstly, we project \eqref{eq:LamUp} using the projector $\Pi_{-,M}$. Recall that $\Pi_{-,M}$ is the projector on the subspace of the spectrum of $\Hc$ which is smaller than $\frac{1-M}{2}$. Unlike as in the previous part where we used the spectrum of $\Hc + \Mc$.\\
\noindent - Secondly, from the main contribution in the projection $\Pi_{-,M}$ of the all terms appearing in \eqref{eq:LamUp}, we write  a system satisfied by $\binom{\Lambda_-}{\Upsilon_-}$, then use a Gronwall's inequality to get the conclusion.\\

\paragraph{Step 1: Projection $\Pi_{-,M}$ of the all terms appearing in \eqref{eq:LamUp}.} In this step, we will find the main contribution in the projection $\Pi_{-,M}$ of various terms appearing in \eqref{eq:LamUp}.

\subparagraph{First term: $\partial_s\binom{\Lambda}{\Upsilon}$.} From \eqref{decomoposeLamUp} and \eqref{def:Qn} and \eqref{def:Qnhat}, its projection is 
\begin{equation*}
\Pi_{-,M}\left[\partial_s\binom{\Lambda}{\Upsilon}\right] = \partial_s\binom{\Lambda_-}{\Upsilon_-}.
\end{equation*}
\subparagraph{Second term: $(\Hc + \Mc)\binom{\Lambda}{\Upsilon}$.} We have the following:
\begin{equation*}
\Pi_{-,M}\left[(\Hc + \Mc)\binom{\Lambda}{\Upsilon}\right] = \Hc\binom{\Lambda_-}{\Upsilon_-} + \Mc\binom{\Lambda_-}{\Upsilon_-},
\end{equation*}
where we used the fact that $\Pi_{-,M}\binom{f_n}{g_n} + \Pi_{-,M}\binom{\tilde f_n}{\tilde g_n} = 0$ for all $n \leq M$.

\subparagraph{Third term: $V\binom{\Lambda}{\Upsilon} = \binom{V_1 \Upsilon}{V_2 \Lambda}$.} We claim the following:
\begin{lemma}[Projection of $\binom{V_1 \Upsilon}{V_2 \Lambda}$ using $\Pi_{-,M}$] \label{lemm:Pineg3rdterm}$\;$\\
$(i)\;$ For all $s \geq 1$, we have 
\begin{align*}
\left\|\frac{\Pi_{-,M}(V_1 \Upsilon)}{1 + |y|^{M+1}}\right\|_{L^\infty(\Rb)} \leq \left(\|V_1\|_{L^\infty(\Rb)} + \frac C s \right)&\left\|\frac{\Upsilon_-}{1 + |y|^{M+1}}\right\|_{L^\infty(\Rb)}\\
& + \sum_{n = 0}^M \frac{C}{s^\frac{M+1 - n}{2}}(|\theta_n(s)| + |\tilde \theta_n(s)|),\\
\left\|\frac{\Pi_{-,M}(V_2 \Lambda)}{1 + |y|^{M+1}}\right\|_{L^\infty(\Rb)} \leq \left(\|V_2\|_{L^\infty(\Rb)} + \frac Cs \right)&\left\|\frac{\Lambda-}{1 + |y|^{M+1}}\right\|_{L^\infty(\Rb)}\\
&+ \sum_{n = 0}^M \frac{C}{s^\frac{M+1 - n}{2}}(|\theta_n(s)| + |\tilde \theta_n(s)|).
\end{align*}
$(ii)\;$ For all $A \geq 1$, there exists $s_8(A) \geq 1$ such that for all $s \geq s_8(A)$, if $\binom{\Lambda(s)}{\Upsilon(s)} \in \Vc_A(s)$, then
\begin{align*}
\left\|\frac{\Pi_{-,M}(V_1 \Upsilon)}{1 + |y|^{M+1}}\right\|_{L^\infty(\Rb)} &\leq \|V_1\|_{L^\infty(\Rb)} \left\|\frac{\Upsilon_-}{1 + |y|^{M+1}}\right\|_{L^\infty(\Rb)} + \frac{CA^M}{s^\frac{M+2}{2}},\\
\left\|\frac{\Pi_{-,M}(V_2 \Lambda)}{1 + |y|^{M+1}}\right\|_{L^\infty(\Rb)} &\leq \|V_2\|_{L^\infty(\Rb)} \left\|\frac{\Lambda_-}{1 + |y|^{M+1}}\right\|_{L^\infty(\Rb)} + \frac{CA^M}{s^\frac{M+2}{2}}.
\end{align*}

\end{lemma}
\begin{proof} We only deal with $V_1\Upsilon$ because the proof for $V_2\Lambda$ is similar.  Let us write $\Upsilon = \Upsilon_+ + \Upsilon_-$, where $\Upsilon_+ = \Pi_{+,M}\Upsilon = (\textbf{Id} - \Pi_{-,M})\Upsilon$ and
$$\Pi_{-,M}(V_1 \Upsilon) =   V_1 \Upsilon_- - \Pi_{+,M}(V_1\Upsilon_-) + \Pi_{-,M}(V_1 \Upsilon_+).$$
The first term is obviously bounded by
$$
\left\|\frac{V_1 \Upsilon_-}{1 + |y|^{M+1}}\right\|_{L^\infty(\Rb)} \leq \|V_1\|_{L^\infty(\Rb)}\left\|\frac{\Upsilon_-}{1 + |y|^{M+1}}\right\|_{L^\infty(\Rb)}.
$$
Note that if $|f(y)| \leq \sigma (1 + |y|^k)$ for $k \in \mathbb{N}$, then $|\Pi_{+,M}f(y)| \leq C\sigma$.  Using this property and \eqref{est:Viorder1},  we obtain the bound for the second term   
$$\left\|\frac{\Pi_{+,M}(V_1 \Upsilon_-)}{1 + |y|^{M+1}}\right\|_{L^\infty(\Rb)} \leq \frac Cs \left\|\frac{\Upsilon_-}{1 + |y|^{M+1}}\right\|_{L^\infty(\Rb)}.$$
For the last term, let us write  from \eqref{eq:LplusUplusexpand},
$$\Pi_{-,M}(V_1 \Upsilon_+) = \sum_{n = 0}^M \Pi_{-,M}\big[V_1(\theta_n g_n + \tilde \theta_n \tilde{g}_n) \big].$$
If $M - n$ is odd, we use \eqref{est:Viorderk} with $k = \frac{M - n - 1}{2}$, hence, 
\begin{align*}
\Pi_{-,M}\big[V_1(\theta_n g_n + \tilde \theta_n \tilde{g}_n) \big] & = \sum_{j = 1}^k\frac{1}{s^j}\Big[\Pi_{-,M}\big[W_{1,j}(\theta_n g_n + \tilde \theta_n \tilde{g}_n) \big]\Big]\\
&  + \Pi_{-,M}\big[\tilde{W}_{1,k}(\theta_n g_n + \tilde \theta_n \tilde{g}_n) \big] = I_1 + I_2.
\end{align*}
Since $\text{deg}(g_n) = \text{deg}(\tilde g_n) = n$ and $n + 2k = M - 1 < M$, we deduce that $I_1 = 0$. Moreover, since $|\tilde{W}_{1,k}(y,s)| \leq \frac{C}{s^{k+1}}(1 + |y|^{2k + 2})$, we deduce from $(iv)$ of Lemma \ref{lemm:prokernalL} that 
$$\left\|\frac{\Pi_{-,M}\big[V_1(\theta_n g_n + \tilde \theta_n \tilde{g}_n) \big]}{1 + |y|^{M+1}}\right\|_{L^\infty(\Rb)} \leq \frac{C(|\theta_n(s)| + |\tilde\theta_n(s)|)}{s^\frac{M - n + 1}{2}}.$$
Similarly, when $M - n$ is even, we use \eqref{est:Viorderk} with $k = \frac{M - n}{2}$ and argue as above to obtain the same estimate. This concludes the proof of part $(i)$. Part $(ii)$ simply follows from part $(i)$ and Definition \ref{def:VA} of $\Vc_A(s)$. This finishes the proof of Lemma \ref{lemm:Pineg3rdterm}.
\end{proof}

\subparagraph{Fourth term: $\binom{F_1}{F_2}$.} We claim the following:
\begin{lemma}[Projection of $\binom{F_1}{F_2}$ using $\Pi_{-,M}$] \label{lemm:Pineg4thterm} Let $\binom{\Lambda(s)}{\Upsilon(s)} \in \Vc_A(s)$. Then for all $A \geq 1$ and $K \geq 1$ introduced in \eqref{def:chi}, there exists $s_9(A,K) \geq 1$ such that for all $s \geq s_9(A,K)$,  the functions $F_1(\Upsilon,y,s)$ and $F_2(\Lambda,y,s)$ defined by \eqref{def:Bys} satisfy: 
\begin{equation*}
\left\|\frac{\Pi_{-,M}[F_1(\Upsilon, y,s)]}{1 + |y|^{M+1}} \right\|_{L^\infty(\Rb)} \leq \frac{CA^{(M+2)^2}}{s^\frac{M + 1 +\bar p}{2}},
\end{equation*}
and 
\begin{equation*}
\left\|\frac{\Pi_{-,M}[F_2(\Lambda, y,s)]}{1 + |y|^{M+1}} \right\|_{L^\infty(\Rb)} \leq \frac{CA^{(M+2)^2}}{s^\frac{M + 1 + \bar q}{2}}, 
\end{equation*}
where $\bar p = \min\{2,p\}$ and $\bar q = \min\{2,q\}$.
\end{lemma}
\begin{proof} We only deal with $F_1(\Upsilon, y,s)$ because the similar estimate holds for $F_2(\Lambda, y,s)$. Since the proof is similar to the proof of Lemma \ref{lemm:Pineg3rdterm}, we just give the key estimate. We first notice that for all polynomial functions $f(y)$ of degree $M$, we have $\Pi_{-,M}f(y) = 0$. Hence, the conclusion follows once we show that there exists a polynomial function $F_{1,M}$ of degree $M$ in $y$ such that  for all $y \in \Rb$ and $s \geq 1$,
\begin{equation}\label{est:F1F1M}
|F_1 - F_{1,M}| \leq \frac{CA^{(M+2)^2}}{s^\frac{M + 1 + \bar p}{2}}(1 + |y|^{M+1}),
\end{equation}
where $\bar p = \min\{2,p\}$. In particular, we take
$$F_{1,M} = \Pi_{+,M}\left[ \sum_{k = 2}^{M+1} \Upsilon^k \sum_{l = 0}^M \frac{1}{s^l}F_{1,k}^l\left( \frac{y}{\sqrt s}\right)\right].$$
To prove \eqref{est:F1F1M}, we recall from Lemma \ref{lemm:decomB} that 
$$\left|F_1 - \sum_{k = 2}^{M+1} \Upsilon^k \sum_{l = 0}^M \frac{1}{s^l}\left[F_{1,k}^l\left( \frac{y}{\sqrt s}\right) + \tilde{F}_{1,k}^l(y,s) \right]\right| \leq C|\Upsilon|^{M+2} + \frac{C}{s^{M+1}}.$$

We first consider the region $|y| \geq \sqrt s$. From \eqref{eq:F1pF2q} and part $(ii)$ of Proposition \ref{prop:properVA}$(ii)$, we have
$$|F_1| \leq C|\Upsilon|^{\bar{p}} \leq C\left(\frac{A^{M+2}}{\sqrt s}\right)^{\bar p} \frac{1}{s^\frac{M+1}{2}}(1 + |y|^{M+1}).$$
From Lemma \ref{lemm:Pro4thterm}, we know that for all $n \leq M$,
$$\left|P_{n,M}\binom{F_1}{F_2} \right| + \left|\tilde P_{n,M}\binom{F_1}{F_2} \right| \leq \frac{CA^n}{s^\frac{n + 2}{2}}.$$
In the region $|y| \leq \sqrt s$, we use the same argument as in the proof of Lemma \ref{lemm:Pro4thterm} to deduce that the coefficient of degree $k \geq M + 1$ of the polynomial 
$$\sum_{k = 2}^{M+1} \Upsilon^k \sum_{l = 0}^M \frac{1}{s^l}F_{1,k}^l\left( \frac{y}{\sqrt s}\right) - F_{1,M}$$
is controlled by $\frac{CA^k}{s^\frac{k + 2}{2}}$, hence, 
$$\left|\sum_{k = 2}^{M+1} \Upsilon^k \sum_{l = 0}^M \frac{1}{s^l}F_{1,k}^l\left( \frac{y}{\sqrt s}\right) - F_{1,M}\right| \leq \frac{CA^{2M + 2}}{s^\frac{M+3}{2}}(1 + |y|^{M+1}).$$
Using part $(i)$ of Proposition \ref{prop:properVA} yields
$$\left|\sum_{k = 2}^{M+1} \Upsilon^k \sum_{l = 0}^M \frac{1}{s^l}\tilde F_{1,k}^l(y,s)\right| \leq \frac{CA^{2M + 2}}{s^\frac{M+3}{2}}(1 + |y|^{M+1}).$$
To control the term $|\Upsilon|^{M+2}$, we use parts $i$ and $(iii)$ of Proposition \ref{prop:properVA} to get 
$$|\Upsilon|^{M+2} \leq C\left(\frac{A^{M+1}}{\sqrt s} \right)^{M+1} \frac{A^{M+1}\log s}{s^2}(1 + |y|^{M+1}).$$
A collection of all the above estimates yields \eqref{est:F1F1M}. The conclusion of Lemma \ref{lemm:Pineg4thterm} follows from \eqref{est:F1F1M} by using the same argument as in the proof of Lemma \ref{lemm:Pineg3rdterm}.

\end{proof}

\subparagraph{Fifth term: $\binom{R_1}{R_2}$.} From Lemma \ref{lemm:expandR1R2}, we have the following:
\begin{lemma}[Projection of $\binom{R_1}{R_2}$ using $\Pi_{-,M}$.] \label{lemm:Pineg5th} The functions $R_1(y,s)$ and $R_2(y,s)$ defined by \eqref{def:Rys} satisfy 
\begin{equation*}
\left\|\frac{\Pi_{-,M}\big[R_i(y,s)\big]}{1 + |y|^{M+1}} \right\|_{L^\infty(\Rb)} \leq \frac{C}{s^\frac{M+3}{2}}.
\end{equation*}

\end{lemma}
\begin{proof} Applying Lemma \ref{lemm:expandR1R2} with $m = \frac{M+2}{2}$, we write for all $|y| \leq \sqrt s$ and $s \geq 1$,
$$\left|R_i(y,s)- \sum_{k = 1}^{M/2}\frac{1}{s^{k + 1}}R_{i,k}(y) \right| \leq \frac{C(1 + |y|^{M + 2})}{s^\frac{M+4}{2}} \leq  \frac{C(1 + |y|^{M + 1})}{s^\frac{M+3}{2}}.$$
Since $\text{deg}(R_{i,k}) = 2k \leq M$, we have $\Pi_{-,M}R_{i,k} = 0$. The conclusion simply follows by using $(iv)$ of Lemma \ref{lemm:prokernalL}. This ends the proof of Lemma \eqref{lemm:Pineg5th}.
\end{proof}

We are ready to prove part $(iv)$ of Proposition \ref{prop:dyn}. 
\paragraph{Step 2: Proof of item $(iv)$ of Proposition \ref{prop:dyn}.} Applying the projection $\Pi_{-,M}$ to system \eqref{eq:LamUp} and using the various estimates given in the first step, we see that $\Lambda_-$ and $\Upsilon_-$ satisfy the following system:
\begin{align*}
\partial_s \Lambda_- = \Lc_1 \Lambda_- - \frac{p+1}{pq-1}\Lambda_- + p\gamma^{p-1}\Upsilon_- + G_{1,-}(y,s)\\
\partial_s \Upsilon_- = \Lc_\mu \Upsilon_- - \frac{q+1}{pq-1}\Upsilon_- + q\Gamma^{p-1}\Lambda_-  + G_{2,-}(y,s),
\end{align*}
where $G_{1,-}$ and $G_{2,-}$ satisfy 
$$\left\|\frac{G_{1,-}(y,s)}{1 + |y|^{M+1}} \right\|_{L^\infty(\Rb)} \leq \|V_1(s)\|_{L^\infty(\Rb)}\left\|\frac{\Upsilon_-}{1 + |y|^{M+1}} \right\|_{L^\infty(\Rb)}  + \frac{CA^M}{s^{\frac{M+2}{2}}} + \frac{CA^{(M+2)^2}}{s^\frac{M+1 + \bar p}{2}},$$
and 
$$\left\|\frac{G_{2,-}(y,s)}{1 + |y|^{M+1}} \right\|_{L^\infty(\Rb)} \leq \|V_2(s)\|_{L^\infty(\Rb)}\left\|\frac{\Lambda_-}{1 + |y|^{M+1}} \right\|_{L^\infty(\Rb)}  + \frac{CA^M}{s^{\frac{M+2}{2}}} + \frac{CA^{(M+2)^2}}{s^\frac{M+1 + \bar q}{2}}.$$

Using the semigroup representation of $\Lc_\eta$ with $\eta \int \{1, \mu\}$, we write for all $s \in [\tau, \tau_1]$,
\begin{align*}
\Lambda_-(s) &= e^{(s - \tau)\Lc_1}\Lambda_-(\tau) + \int_{\tau}^s e^{(s - s')\Lc_1}\left(- \frac{p+1}{pq-1}\Lambda_-(s') + p\gamma^{p-1}\Upsilon_-(s') + G_{1,-}(s')\right)ds'\\
\Upsilon_-(s) &= e^{(s - \tau)\Lc_\mu}\Upsilon_-(\tau) + \int_{\tau}^s e^{(s - s')\Lc_\mu}\left(- \frac{q+1}{pq-1}\Upsilon_-(s')+ q\Gamma^{p-1}\Lambda_-(s')  + G_{2,-}(s') \right)ds'.
\end{align*}
Using part $(iii)$ of Lemma \ref{lemm:prokernalL}, we get
\begin{align*}
\left\|\frac{\Lambda_-(s)}{1 + |y|^{M+1}}\right\|_{L^\infty} &\leq e^{-\frac{M+1}{2}(s - \tau)}\left\|\frac{\Lambda_-(\tau)}{1 + |y|^{M+1}}\right\|_{L^\infty}\\
&+ \frac{p+1}{pq-1}\int_{\tau}^s e^{-\frac{M+1}{2}(s - s')}\left\|\frac{\Lambda_-(s')}{1 + |y|^{M+1}}\right\|_{L^\infty}ds'\\
&+ p\gamma^{p-1}\int_{\tau}^s e^{-\frac{M+1}{2}(s - s')}  \left\|\frac{\Upsilon_-(s')}{1 + |y|^{M+1}}\right\|_{L^\infty}ds'\\
& + \int_{\tau}^s e^{-\frac{M+1}{2}(s - s')}\left\|\frac{G_{1,-}(y,s)}{1 + |y|^{M+1}} \right\|_{L^\infty}ds',
\end{align*}
and
\begin{align*}
\left\|\frac{\Upsilon_-(s)}{1 + |y|^{M+1}}\right\|_{L^\infty} &\leq e^{-\frac{M+1}{2}(s - \tau)}\left\|\frac{\Upsilon_-(\tau)}{1 + |y|^{M+1}}\right\|_{L^\infty}\\
&+ \frac{q+1}{pq-1}\int_{\tau}^s e^{-\frac{M+1}{2}(s - s')}\left\|\frac{\Upsilon_-(s')}{1 + |y|^{M+1}}\right\|_{L^\infty}ds'\\
&+ q\Gamma^{q-1}\int_{\tau}^s e^{-\frac{M+1}{2}(s - s')}  \left\|\frac{\Lambda_-(s')}{1 + |y|^{M+1}}\right\|_{L^\infty}ds'\\
& + \int_{\tau}^s e^{-\frac{M+1}{2}(s - s')}\left\|\frac{G_{2,-}(y,s)}{1 + |y|^{M+1}} \right\|_{L^\infty}ds',
\end{align*} 
If we set $\lambda(s) = \left\|\frac{\Lambda_-(s)}{1 + |y|^{M+1}}\right\|_{L^\infty}+\left\|\frac{\Upsilon_-(s)}{1 + |y|^{M+1}}\right\|_{L^\infty}$, then we have 
\begin{align*}
\lambda(s) &\leq e^{-\frac{M+1}{2}(s -\tau)}\lambda(\tau) \\
& +\int_{\tau}^s e^{-\frac{M+1}{2}(s - s')} \left(\|\Mc\|_\infty + \|V_1\|_{L^\infty} + \|V_2\|_{L^\infty} \right)\lambda(s')ds'\\
&+C\int_{\tau}^s e^{-\frac{M+1}{2}(s - s')} \left(\frac{A^{(M+2)^2}}{{s'}^\frac{M+2}{2}} \left({s'}^{\frac{\bar p - 1}{2}} + {s'}^{\frac{\bar q - 1}{2}} \right) + \frac{A^M}{{s'}^\frac{M+2}{2}} \right)ds',
\end{align*}
where $\|\Mc\|_\infty = \max \left\{\frac{p+1}{pq-1} + p\gamma^{p-1}, \frac{q+1}{pq-1} + q\Gamma^{q-1}\right\}$. 

Since we have already fixed $M$ in \eqref{eq:Mfixed} such that 
$$M \geq 4 (\|\Mc\|_{\infty} + 1+ \|V_1\|_{L^\infty} + \|V_2\|_{L^\infty}),$$
and that $A^{(M+2)^2}\left({s'}^{\frac{\bar p - 1}{2}} + {s'}^{\frac{\bar q - 1}{2}} \right) \leq A^M$ for $s'$ large enough, we then apply Lemma \ref{lemm:Gronwall} to deduce that 
$$e^{\frac{M+1}{2}s} \lambda(s) \leq e^{\frac{M+1}{4}(s - \tau)}e^{\frac{M+1}{2}\tau} \lambda(\tau) + Ce^{\frac{M+1}{2}s}\frac{A^M}{s^\frac{M+2}{2}},$$
which concludes the proof of part $(iv)$ of Proposition \ref{prop:dyn}.

\subsubsection{The outer part.}\label{sec:outerpart} 
We prove part $(v)$ of Proposition \ref{prop:dyn} in this subsection. Let us write from \eqref{eq:LamUp} a system satisfied by $\tilde \Lambda_e = (1 - \chi(2y,s))\Lambda$ and $\tilde \Upsilon_e = (1 - \chi(2y,s))\Upsilon$:
\begin{align*}
\partial_s \tilde \Lambda_e &= \Lc_1 \tilde \Lambda_e - \frac{p+1}{pq-1} \tilde \Lambda_e + (1 - \chi(2y,s))\big(\tilde{F}_1(\Upsilon, y,s) + R_1(y,s) \big)\\
& - \Lambda(s)\left(\partial_s \chi(2y,s) + \Delta \chi(2y,s) + \frac{1}{2}y\cdot \nabla \chi(2y,s) \right) + 2\text{div}(\Lambda \nabla \chi(2y,s)),\\
\partial_s \tilde \Upsilon_e &= \Lc_\mu \tilde \Upsilon_e - \frac{q+1}{pq-1} \tilde \Upsilon_e + (1 - \chi(2y,s))\big(\tilde{F}_2(\Lambda, y,s) + R_2(y,s) \big)\\
& - \Upsilon(s)\left(\partial_s \chi(2y,s) + \mu\Delta \chi(2y,s) + \frac{1}{2}y\cdot \nabla \chi(2y,s) \right) + 2\mu\text{div}(\Upsilon \nabla \chi(2y,s)),
\end{align*}
where 
$$\tilde{F}_1(\Upsilon, y,s) = |\Upsilon + \psi|^{p-1}(\Upsilon + \psi) - \psi^p, \quad \tilde F_2(\Lambda, y,s) = |\Lambda + \varphi|^{q - 1}(\Lambda + \varphi) - \varphi^q.$$
Using the  semigroup representation of $\Lc_\eta$ with $\eta \in \{1, \mu\}$ and parts $(i) - (ii)$ of Lemma \ref{lemm:prokernalL}, we write for all $s \in [\tau, \tau_1]$,
\begin{align*}
\|\tilde \Lambda_e(s)\|_{L^\infty} &\leq e^{-\frac{p+1}{pq-1}(s - \tau)}\|\tilde \Lambda_e(\tau)\|_{L^\infty}\\
& + \int_{\tau}^s e^{-\frac{p+1}{pq-1}(s - s')}\left(\left\|(1 - \chi(2y,s'))\tilde{F}_1(s')\right\|_{L^\infty} + \left\|(1 - \chi(2y,s'))R_1(s')\right\|_{L^\infty} \right)ds'\\
& + \int_{\tau}^s e^{-\frac{p+1}{pq-1}(s - s')}\left\|\Lambda(s')\left(\partial_s \chi(2y,s') + \Delta \chi(2y,s') + \frac{1}{2}y\cdot \nabla \chi(2y,s') \right)  \right\|_{L^\infty}ds'\\
& + \int_{\tau}^s e^{-\frac{p+1}{pq-1}(s - s')} \frac{C}{\sqrt{1 - e^{-(s - s')}}}\|\Lambda(s')\nabla \chi(2y,s')\|_{L^\infty} ds',
\end{align*}
and 
\begin{align*}
\|\tilde \Upsilon_e(s)\|_{L^\infty} &\leq e^{-\frac{q+1}{pq-1}(s - \tau)}\|\tilde \Upsilon_e(\tau)\|_{L^\infty}\\
& + \int_{\tau}^s e^{-\frac{q+1}{pq-1}(s - s')}\left(\left\|(1 - \chi(2y,s'))\tilde{F}_2(s')\right\|_{L^\infty} + \left\|(1 - \chi(2y,s'))R_2(s')\right\|_{L^\infty} \right)ds'\\
& + \int_{\tau}^s e^{-\frac{q+1}{pq-1}(s - s')}\left\|\Upsilon(s')\left(\partial_s \chi(2y,s') + \mu\Delta \chi(2y,s') + \frac{1}{2}y\cdot \nabla \chi(2y,s') \right)  \right\|_{L^\infty}ds'\\
& + \int_{\tau}^s e^{-\frac{q+1}{pq-1}(s - s')} \frac{C}{\sqrt{1 - e^{-(s - s')}}}\|\Upsilon(s')\nabla \chi(2y,s')\|_{L^\infty} ds',
\end{align*}
From the definition \eqref{def:chi} of $\chi$ and  part $(i)$ of Proposition \ref{prop:properVA}, we have 
\begin{align*}
&\left\|\Lambda(s')\left(\partial_s \chi(2y,s') + \Delta \chi(2y,s') + \frac{1}{2}y\cdot \nabla \chi(2y,s') \right)  \right\|_{L^\infty}\\
& + \left\|\Upsilon(s')\left(\partial_s \chi(2y,s') + \mu\Delta \chi(2y,s') + \frac{1}{2}y\cdot \nabla \chi(2y,s') \right)  \right\|_{L^\infty}\\
&\leq C\left(1 + \frac{1}{K^2s'^2} \right)\left(\|\Lambda(s')\|_{L^\infty(|y| \leq K\sqrt {s'})} + \|\Upsilon(s')\|_{L^\infty(|y| \leq K\sqrt {s'})} \right) \leq \frac{CA^{M+1}}{\sqrt {s'}},
\end{align*}
and 
\begin{align*}
&\|\Lambda(s')\nabla \chi(2y,s')\|_{L^\infty} + \|\Upsilon(s')\nabla \chi(2y,s')\|_{L^\infty}\\
& \leq \frac{C}{K\sqrt {s'}}\left(\|\Lambda(s')\|_{L^\infty(|y| \leq K\sqrt {s'})} + \|\Upsilon(s')\|_{L^\infty(|y| \leq K\sqrt {s'})} \right) \leq \frac{CA^{M+1}}{s'}.
\end{align*}
Recalling from \eqref{eq:boundR1R2} the bound for $R_1$ and $R_2$, we have 
$$\left\|(1 - \chi(2y,s'))R_i(s')\right\|_{L^\infty} \leq \frac{C}{s'}, \quad i = 1,2.$$
We also have
\begin{align*}
\left\|(1 - \chi(2y,s'))\tilde{F}_1(s')\right\|_{L^\infty}& \leq C \left( \|\psi(s')\|^{p-1}_{L^\infty(|y| \geq K\sqrt {s'})} + \|\Upsilon(s')\|^{p-1}_{L^\infty(|y| \geq K\sqrt {s'})} \right)\|\tilde \Upsilon_e(s')\|_{L^\infty}\\
& \leq  \frac{r + 1}{2(pq - 1)}\|\tilde \Upsilon_e(s')\|_{L^\infty},
\end{align*}
for $K$ large enough, where 
$$r = \min\{p,q\}.$$
Similarly, 
$$\left\|(1 - \chi(2y,s'))\tilde{F}_2(s')\right\|_{L^\infty} \leq  \frac{r + 1}{2(pq - 1)}\|\tilde \Lambda_e(s')\|_{L^\infty}.$$
If we set $\lambda(s) = \|\tilde \Lambda_e(s)\|_{L^\infty} + \|\tilde \Upsilon_e(s)\|_{L^\infty}$, then we end up with 
\begin{align*}
\lambda(s) &\leq e^{-\frac{r + 1}{pq - 1}(s - \tau)}\lambda(\tau)\\
&+\int_{\tau}^s e^{-\frac{r + 1}{pq-1}(s - s')}\left(\frac{r + 1}{2(pq-1)} \lambda(s') + \frac{CA^{M+1}}{\sqrt{s'}} + \frac{CA^{M+1}}{s' \sqrt{1 - e^{-(s - s')}}}\right)ds'.
\end{align*}
Applying Lemma \ref{lemm:Gronwall}, we finally obtain
$$\lambda(s) \leq e^{-\frac{r + 1}{2(pq - 1)}(s - \tau)}\lambda(\tau) + \frac{CA^{M+1}}{\sqrt{s}}(s - \tau + \sqrt{s - \tau}).$$
Since $\text{supp}(1 - \chi(y,s)) \subset \text{supp}(1 - \chi(2y,s))$, we have $\|\Lambda_e\|_{L^\infty} \leq \|\tilde \Lambda_e\|_{L^\infty}$ and $\|\Upsilon_e\|_{L^\infty} \leq \|\tilde \Upsilon_e\|_{L^\infty}$. This concludes the proof of part $(v)$ of Proposition \ref{prop:dyn}.

\section{Stability of the constructed solution.}\label{sec:stab}
In this section we give the proof of Theorem \ref{theo2}. The proof strongly relies on the same ideas used in the proof of Theorem \ref{theo1}. That is  the use of finite-dimensional parameters, the reduction to a finite-dimensional problem and the continuity. As the proof of Theorem \ref{theo1}, we only give the proof of Theorem \ref{theo2} in the one dimensional case for simplicity, however, the same proof holds for higher dimensions. We claim the following which directly follows Theorem \ref{theo2}:

\begin{proposition} \label{prop:stab} Let $(\hat u_0, \hat v_0)$ be the initial data of system \eqref{PS} such that the corresponding solution $(\hat u, \hat v)$ blows up in finite time $\hat T$ at only one blowup point $\hat a$ and $(\hat u(x, t), \hat v(x,t))$ satisfies \eqref{eq:asyTh1} with $T = \hat T$ and $a = \hat a$. Then, there exist $B_0 \geq 1$, $s_0 \geq 1$,  a neighborhood $\Ec_{s_0}$ of $(\hat T, \hat a)$ in $\Rb^2$ and a neighborhood $\mathscr{W}_0$ of $(\hat u_0, \hat v_0)$ in $L^\infty(\Rb) \times L^\infty(\Rb)$ such that the following holds: for any $(u_0, v_0) \in \mathscr{W}_0$, there exists $(T,a) \in \Ec_{s_0}$ such that for all $s \geq s_0$, $\binom{\Lambda_{T,a}(s)}{\Upsilon_{T,a}(s)} \in \Vc_{B_0}(s)$, where
\begin{equation}\label{def:LUTastab}
\Lambda_{T,a}(y,s) = \Phi_{T,a}(y,s) - \varphi(y,s), \quad \Upsilon_{T,a}(y,s) = \Psi_{T,a}(y,s) - \psi
(y,s),
\end{equation}
where $(\Phi_{T,a},\Psi_{T,a})$ is defined as in \eqref{def:simiVars} with $(u,v)$ is the unique solution of \eqref{PS} with initial data $(u_0, v_0)$, and $\varphi$, $\psi$ are defined in \eqref{def:varphiys} and \eqref{def:psiys}.
\end{proposition}

\medskip 

Indeed, once  Proposition \ref{prop:stab} is proved, we deduce from part $(ii)$ of Proposition \ref{prop:properVA} and \eqref{def:simiVars} that \eqref{eq:asyTh1} holds for $(u,v)$. Then, Proposition \ref{prop:Noblowup} applied to $(u,v)$ shows that $(u,v)$ blows up at time $T$ at one single point $a$. Since part $(iii)$ of Theorem \ref{theo1} follows from part $(ii)$, we conclude the proof of Theorem \ref{theo2} assuming that Proposition \ref{prop:stab} holds. \\ 

Let us now give the proof of Proposition \ref{prop:stab}. The proof is anagolous to the case of   equation \eqref{eq:Scalar} treated in \cite{MZdm97} (see also \cite{TZpre15}). For the reader's convenience, we give here the main idea of the proof. The interested reader is kindly referred to the stability section in \cite{MZdm97} and \cite{TZpre15} for more details.   \\

We consider $(\hat u, \hat v)$ the constructed solution of system \eqref{PS} in Theorem \ref{theo1}, and call $(\hat u_0, \hat v_0)$ its initial data in $L^\infty(\Rb) \times L^\infty(\Rb)$, and $(\hat T, \hat a)$ its blowup time and blowup point. From the construction method given in Section \ref{sec:existence}, we consider $\hat A \geq 1$ such that 
\begin{equation*}
\binom{\hat \Lambda(s)}{\hat \Upsilon(s)} \in \Vc_{\hat A}(s) \quad \text{for all}\; s \geq - \log \hat T,
\end{equation*}
where 
\begin{align}
\hat{\Lambda}(y,s) &= \hat \Phi(y,s) - \varphi(y,s), \quad \hat{\Phi}(y,s) = e^{-\frac{(p+1)s}{pq-1}}\hat u \left(\hat a + ye^{-\frac{s}{2}}, \hat T - e^{-s}\right),\label{def:Lamhat}\\
\hat{\Upsilon}(y,s) &= \hat \Psi(y,s) - \psi(y,s), \quad \hat{\Psi}(y,s) = e^{-\frac{(q+1)s}{pq-1}}\hat v \left(\hat a + ye^{-\frac{s}{2}}, \hat T - e^{-s}\right),\label{def:Upshat}
\end{align}
and $\varphi, \psi$ are defined in \eqref{def:varphiys} and \eqref{def:psiys}. 

Let $\epsilon_0 > 0$, we consider $(u_0, v_0) \in  L^\infty(\Rb) \times L^\infty(\Rb)$ such that 
$$(h_0, g_0) = (u_0 - \hat u_0, v_0 - \hat v_0), \quad \|h_0\|_{L^\infty(\Rb)} + \|g_0\|_{L^\infty(\Rb)} \leq \epsilon_0.$$
We denote by $(u,v)_{u_0,v_0}$ the solution of system \eqref{PS} with the initial data $(u_0,v_0)$, and by $T(u_0, v_0) \leq +\infty$ the maximal time of existence from the Cauchy theory in $ L^\infty(\Rb) \times L^\infty(\Rb)$. 

Our aim is to show that, if $\epsilon_0$ is small enough, then $T(u_0,v_0) < +\infty$ and $(u,v)_{u_0,v_0}$ blows up in finite time $T(u_0,v_0)$ only at one blowup point $a(u_0,v_0)$ with
$$|T(u_0, v_0) - \hat T| + |a(u_0, v_0) - \hat a| \to 0 \quad \text{as}\quad \epsilon_0 \to 0.$$
Moreover, there exist $B \geq 1$ and $s_0 \geq - \log \hat T$ large enough such that 
$$\binom{\Lambda_{T,a}(s)}{\Upsilon_{T,a}(s)} \in \Vc_B(s) \quad \text{for all}\; s \geq s_0,$$
where $\Lambda_{T,a}$ and $\Upsilon_{T,a}$ are defined in \eqref{def:LUTastab}.

Introducing for all  $x \in \Rb$ and $t \in \big[0, \min\{T(u_0,v_0), \hat T\}\big)$,
$$(h(x,t), g(x,t)) = (u(x,t) - \hat u(x,t), v(x,t) - \hat v(x,t)),$$
we see from \eqref{def:LUTastab}, \eqref{def:Lamhat}, \eqref{def:Upshat} and \eqref{def:simiVars} that for any $\sigma_0 \in [-\log \hat T, - \log(\hat T - T(u_0,v_0))_+)$, 
\begin{align}
\Lambda_{T,a, u_0,v_0}(y,s_0) &\equiv \bar{\Lambda}_0( T, a, u_0, v_0,y, \sigma_0)\nonumber \\
& = (1 + \tau)^\frac{p+1}{pq-1}\left[\hat \Lambda(z,\sigma_0) + \varphi(z, \sigma_0) \right] - \varphi(y,s_0)\nonumber\\
& + (1 + \tau)^\frac{p+1}{pq-1}e^{-\frac{(p+1)\sigma_0}{pq-1}}\,h\left(ze^{-\frac{\sigma_0}{2}}, \hat T - e^{-\sigma_0}\right), \label{def:barLamstab}\\
\Upsilon_{T,a, u_0, v_0}(y,s_0) &\equiv  \bar{\Upsilon}_0( T, a, u_0, v_0, y, \sigma_0) \nonumber \\
& = (1 + \tau)^\frac{q+1}{pq-1}\left[\hat \Upsilon(z,\sigma_0) + \psi(z, \sigma_0)\right] - \psi(y,s_0),\nonumber\\
& + (1 + \tau)^\frac{q+1}{pq-1}e^{-\frac{(q+1)\sigma_0}{pq-1}}\,g\left(ze^{-\frac{\sigma_0}{2}}, \hat T - e^{-\sigma_0}\right), \label{def:barUpstab}
\end{align}
where 
\begin{equation}\label{def:taualpha}
\tau = (T - \hat T)e^{\sigma_0}, \quad z = y\sqrt{1 + \tau} + \alpha, \quad \alpha = (a - \hat a)e^\frac{\sigma_0}{2}, \quad s_0 = \sigma_0 - \log (1+\tau).
\end{equation}

In view of \eqref{def:barLamstab} and \eqref{def:barUpstab}, $\binom{\bar \Lambda_0}{\bar \Upsilon_0}(T,a, u_0, v_0, \sigma_0)$ appears as initial data for system \eqref{eq:LamUp} at time $s = s_0(\sigma_0,\tau)$ and our parameters is now $(T,a)$ replacing $(d_0,d_1)$ in \eqref{eq:intialdata}. In particular, we have the following property:
\begin{proposition}[Properties of initial data $\binom{\bar \Lambda_0}{\bar \Upsilon_0}(T,a, u_0, v_0, \sigma_0)$ given in \eqref{def:barLamstab} and \eqref{def:barUpstab}] \label{prop:initstab} There exists $B_0 = B_0(M, \hat A) \geq 1$ such that for any $B \geq B_0$, there exists $\sigma_0'(B) \geq 1$ large enough such that for any $\sigma_0 \geq \sigma_0'$, there exists $\epsilon_0(\sigma_0) > 0$ small enough such that 
$$\|u_0 - \hat u_0\|_{L^\infty(\RN)} + \|v_0 - \hat v_0\|_{L^\infty(\RN)} \leq \epsilon_0(\sigma_0),$$
and the following hold:\\
$(i)\;$ There exists a set 
\begin{equation*}
\bar \Dc_{B,\sigma_0, u_0, v_0} \subset \left\{(T,a) \Big \vert |T-\hat T| \leq \frac{2Be^{-\sigma_0}(pq-1)}{\sigma^2_0}, \; |a - \hat a| \leq \frac{Be^{-\frac{\sigma_0}{2}}(pq-1)}{b \sigma_0}  \right\},
\end{equation*}
whose boundary is a Jordan curve such that the mapping 
$$(T,a) \mapsto (\bar \theta_{0,0}, \bar \theta_{0,1})(T,a, u_0,v_0,\sigma_0), \quad s_0 = \sigma_0 - \log (1 + \tau),$$
(where $\bar \theta_{0,i} = P_{i,M}\binom{\bar{\Lambda}_0}{\bar \Upsilon_0}$ for $i=0,1$ and $\binom{\bar{\Lambda}_0}{\bar \Upsilon_0}$ stands for $\binom{\bar{\Lambda}_0}{\bar \Upsilon_0}(T,a, u_0,v_0, \sigma_0)$, is one to one from $\bar \Dc_{B,\sigma_0, u_0,v_0}$ onto $\left[-\frac{B}{s_0^2}, \frac{B}{s_0^2}\right]^{N+1}$. Moreover, it is of degree $-1$ on the boundary.\\
$(ii)\;$ For all $(T,a) \in \bar \Dc_{B, \sigma_0, u_0, v_0}$, $\binom{\bar{\Lambda}_0}{\bar \Upsilon_0}$ verifies 
$$\bar \Lambda_{0,e} < \frac{B^{M+2}}{\sqrt {s_0}},\quad  \bar \Upsilon_{0,e} < \frac{B^{M+2}}{\sqrt{s_0}},$$

$$ \left\|\frac{\bar \Lambda_{0,-}(y)}{1 + |y|^{M+1}} \right\| < \frac{B^{M+1}}{s_0^{\frac{M+2}{2}}},\quad  \left\|\frac{\bar \Upsilon_{0,-}(y)}{1 + |y|^{M+1}} \right\| < \frac{B^{M+1}}{s_0^{\frac{M+2}{2}}},$$

$$|\bar \theta_{0,j}|< \frac{B^j}{s_0^\frac{j+1}{2}},\quad |\tilde{\bar \theta}_{0,j}| < \frac{A^j}{s_0^\frac{j+1}{2}} \;\; \text{for}\;\; 3\leq j\leq M,$$

$$ |\tilde {\bar\theta}_{0,i}| < \frac{B^2}{s_0^2}\;\; \text{for}\;\; i = 0, 1,2,$$

$$|\bar \theta_{0,2}| < \frac{B^4 \log s_0}{s_0^2},$$

$$|\bar \theta_{0,0}| \leq \frac{A}{s_0^2}, \quad |\bar \theta_{0,1}| \leq \frac{A}{s_0^2}.$$
\end{proposition}
\begin{proof} The proof directly follows from the expansion of $\binom{\bar{\Lambda}_0}{\bar \Upsilon_0}$ given in \eqref{def:barLamstab} and \eqref{def:barUpstab} for $(T,a)$ close to $(\hat T, \hat a)$. It happens that the proof is completely analogous to the case of equation \eqref{eq:Scalar} treated in \cite{MZdm97} (see also \cite{TZpre15}). For this reason, we omit the proof and kindly refer the interested reader to  Lemma B.4, page 186 in \cite{MZdm97} and Lemma 6.2 in \cite{TZpre15} for analogous proofs.
\end{proof}

With the result of Proposition \ref{prop:initstab} in hands, we are ready to complete the proof of Proposition \ref{prop:stab}. Recall that in the existence proof given in Section \ref{sec:existence}, we had to specific choice of the two parameters $(d_0,d_1) \in \Rb^2$ appearing in \eqref{eq:intialdata} in order to guarantee that $\binom{\Lambda(s)}{\Upsilon(s)}_{d_0,d_1} \in \Vc_A(s)$ for all $s \geq s_0$ for some $A \geq 1$ and $s_0(A) \geq 1$ large enough. In particular, we choose $(d_0,d_1) \in \Dc_{s_0}$ so that the initial data at $s= s_0$  of \eqref{eq:LamUp} is small in $\Vc_A(s_0)$. Together with the dynamics of system \eqref{eq:LamUp} given in Proposition \ref{prop:dyn}, we show that it stays small in $\Vc_A(s)$ up to $s = s_0 + \lambda$ for some $\lambda= \log A$ (see Subsection \ref{sec:redufromdyn}). In the case $s \geq s_0 + \lambda$, we didn't use the data at $s = s_0$, we only used  Proposition \ref{prop:dyn} to derive the smallness of the solution. In particular, we derive the so-called reduction of the problem to a finite-dimensional one (see Proposition \ref{prop:redu}). Then the topological argument for the finite-dimensional problem involving two parameters $(d_0,d_1)$ allows us to conclude the existence of $(d_0,d_1) \in \Dc_{s_0}$ such that the solution of \eqref{eq:LamUp} with initial data \eqref{eq:intialdata} is trapped in $\Vc_A(s)$ for all $s \geq s_0$. Now, starting from $\binom{\bar \Lambda_0}{\bar \Upsilon_0}(T,a, u_0,v_0, \sigma_0)$ at time $s = s_0$ and applying the same procedure as for the existence proof including the reduction to a finite dimensional problem (see Proposition \ref{prop:redu}) and the topological argument involving the two parameters $(T,a)$, we end-up with the existence of $(\bar T(u_0,v_0), \bar a(u_0,v_0)) \in \bar \Dc_{B, \sigma_0, u_0,v_0}$ such that system \eqref{eq:LamUp} with initial data at time $s = s_0$, $\binom{\bar \Lambda_0}{\bar \Upsilon_0}(\bar T(u_0,v_0), \bar a(u_0,v_0), u_0,v_0, \sigma_0)$, has a solution $\binom{\bar \Lambda}{\bar \Upsilon}_{\sigma_0, u_0,v_0}$ such that 
\begin{equation*}
\binom{\bar \Lambda(s)}{\bar \Upsilon(s)}_{ u_0,v_0,\sigma_0} \in \Vc_B(s) \quad \text{for all}\; s \geq s_0.
\end{equation*}

By definition, $\binom{\bar \Lambda_0}{\bar \Upsilon_0}(\bar T(u_0,v_0), \bar a(u_0,v_0), u_0,v_0, \sigma_0)$ is the initial data also at time $s = s_0$ defined in \eqref{def:barLamstab} and \eqref{def:barUpstab}, of $\binom{\bar \Lambda}{\bar \Upsilon}_{u_0,v_0,\sigma_0}$, another solution of the same equation \eqref{eq:LamUp}. From the uniqueness of the Cauchy problem, both solutions are equal and have the same domain of the definition and the same trapping property in $\Vc_B(s)$. Reminding that $\binom{\Lambda}{\Upsilon}_{\bar T(u_0,v_0), \bar a(u_0,v_0), u_0,v_0,\sigma_0}(y,s)$ is defined for all $(y,s) \in \Rb \times \big[-\log \bar T(u_0,v_0), - \log ((\bar T(u_0,v_0) - T(u_0,v_0))_+) \big]$, which implies that 
$$\bar T(u_0,v_0) = T(u_0,v_0)$$
and 
\begin{equation*}
\binom{\Lambda(s)}{\Upsilon(s)}_{\bar T(u_0,v_0), \bar a(u_0,v_0), u_0,v_0,\sigma_0} = \binom{\bar \Lambda(s)}{\bar \Upsilon(s)}_{u_0,v_0,\sigma_0} \in \Vc_B(s) \quad \text{for all}\; s \geq s_0.
\end{equation*}
This concludes the proof of Proposition \ref{prop:stab} as well as Theorem \ref{theo2}.

\appendix

\section{Some elementary lemmas.}

The following lemma is the integral version of Gronwall's inequality:
\begin{lemma}[A Gronwall's inequality]\label{lemm:Gronwall} If $\lambda(s)$, $\alpha(s)$ and $\beta(s)$ are continuous defined on $[s_0, s_1]$ such that 
$$\lambda(s) \leq \lambda(s_0) + \int_{s_0}^s\alpha(\tau) \lambda(\tau) d\tau + \int_{s_0}^{s}\beta(\tau)d\tau, \quad s_0 \leq s\leq s_1,$$
then 
$$\lambda(s) \leq \exp\left(\int_{s_0}^s\alpha(\tau) d\tau \right) \left[\lambda(s_0) + \int_{s_0}^s \beta(\tau) \exp \left(-\int_{s_0}^\tau \alpha(\tau')d\tau' \right)d\tau\right].$$  
\end{lemma}
\begin{proof} See Lemma 2.3 in \cite{GKcpam89}.

\end{proof}

In the following lemma, we recall some linear regularity estimates of the linear operator $\Lc_\eta$ defined in \eqref{def:opH}:

\begin{lemma}[Properties of the semigroup $e^{\tau \Lc_\eta}$]\label{lemm:prokernalL} The kernel $e^{\tau \Lc_\eta}(y,x)$ of the semigroup $e^{\tau \Lc_\eta}$ is given by
\begin{equation}\label{def:kernelE}
e^{\tau \Lc_\eta}(y,x) =  \frac{1}{\big[4\pi(1 - e^{-\tau})\big]^{N/2}}\exp \left(-\frac{|ye^{-\tau/2} - x|^2}{4\eta (1 - e^{\tau})} \right), \quad \forall \tau > 0,
\end{equation}
and $e^{\tau \Lc_\eta}$ is defined by
\begin{equation}\label{def:semieL}
e^{\tau \Lc_\eta}g(y) = \int_{\RN}e^{\tau \Lc_\eta}(y,x)g(x)dx.
\end{equation}

We also have the following:\\

\noindent $(i)\;$ $\left\|e^{\tau \Lc_\eta}g\right\|_{L^\infty(\RN)} \leq \|g\|_{L^\infty(\RN)}$ for all $g \in L^\infty(\RN)$,\\

\noindent $(ii)\;$ $\left\|e^{\tau \Lc_\eta}\, \text{div} (g)\right\|_{L^\infty(\RN)} \leq \frac{C}{\sqrt{1 - e^{-\tau}}}\|g\|_{L^\infty(\RN)}$ for all $g \in L^\infty(\RN)$,\\

\noindent $(iii)\;$ If $|g(x)| \leq (1 + |x|^{M+1})$ for all $x \in \RN$, then 
$$\left|e^{\tau \Lc_\eta} \Pi_{-,M}(g(y))\right| \leq Ce^{-\frac{(M+1)\tau}{2}} (1 + |y|^{M+1}), \quad \forall y \in \RN.$$

\noindent $(iv)\;$ For all $k \geq 0$, we have 
\begin{equation*}
\left\|\frac{\Pi_{-,M}(g)}{1 + |y|^{M+k}} \right\|_{L^\infty(\RN)} \leq C\left\|\frac{g}{1 + |y|^{M + k}} \right\|_{L^\infty(\RN)}.
\end{equation*}
\end{lemma}
\begin{proof} The expressions of $e^{\tau \Lc_\eta}(y,x)$ and $e^{\tau \Lc_\eta}$ are given in \cite{BKnon94}, page 554. The proof of $(i)-(ii)$ follow by straightforward calculations using \eqref{def:kernelE} and \eqref{def:semieL}. For $(iii) - (iv)$, see Lemmas A.2 and A.3 in \cite{MZjfa08}.
\end{proof}

\def\cprime{$'$}

%

\bigskip

\end{document}